\documentclass[11pt,leqno]{article}
\usepackage[utf8]{inputenc}

\usepackage{amsmath}
\usepackage{amsthm}
\usepackage{enumerate}
\usepackage{mathtools}
 
\usepackage[colorlinks=false]{hyperref}

\usepackage{color}

 \usepackage{tikz}
 \usetikzlibrary{arrows.meta}
 \usetikzlibrary{decorations.markings}

\makeatletter
\g@addto@macro\bfseries{\boldmath}
\makeatother

\usepackage[charter]{mathdesign}

\usepackage{xy}
\xyoption{all}

\numberwithin{equation}{section}

\theoremstyle{plain}
    \newtheorem{theorem}[equation]{Theorem}
    \newtheorem{lemma}[equation]{Lemma}
    \newtheorem{corollary}[equation]{Corollary}
    \newtheorem{proposition}[equation]{Proposition}

    \newtheorem*{theorem*}{Theorem}
    \newtheorem*{proposition*}{Proposition}
    \newtheorem*{corollary*}{Corollary}
    \newtheorem*{lemma*}{Lemma}
    \newtheorem*{conjecture*}{Conjecture}
    \newtheorem{definition-theorem}[equation]{Definition/Theorem}
    \newtheorem{definition-lemma}[equation]{Definition/Lemma}
\theoremstyle{definition}
    \newtheorem{definition}[equation]{Definition}
    \newtheorem{example}[equation]{Example}

    \newtheorem{remark}[equation]{Remark}

    \newcommand{\R}{\mathbb{R}}
    \newcommand{\C}{\mathbb{C}}
    
    \newcommand{\Z}{\mathbb{Z}}
    \newcommand{\Q}{\mathbb{Q}}
    \newcommand{\T}{\mathbb{T}}
  
    \newcommand{\germ}{\mathfrak}

   	\renewcommand{\phi}{\varphi}
	\let\epsilon\varepsilon
     
    \newcommand{\Bounded}{\operatorname{B}}
    
    \newcommand{\Compact}{\operatorname{K}}
    
    \newcommand{\Unitary}{\operatorname{U}}

\newcommand{\into}{\hookrightarrow}
\newcommand{\onto}{\twoheadrightarrow}
\newcommand{\dd}{d}  
\newcommand{\restrict}{\raisebox{-.5ex}{$|$}}

\newcommand{\id}{\mathrm{id}}

    \DeclareMathOperator{\Hom}{Hom}

    \DeclareMathOperator{\res}{res}
    \DeclareMathOperator{\Ind}{Ind}
    
    \DeclareMathOperator{\Ad}{Ad}
    \DeclareMathOperator{\Aut}{Aut}
    \DeclareMathOperator{\Rep}{Rep}
    
    \DeclareMathOperator{\lspan}{span}
	\DeclareMathOperator{\image}{image}
    \DeclareMathOperator{\SL}{SL}
    \DeclareMathOperator{\GL}{GL}
    \DeclareMathOperator{\Sp}{Sp}
    \DeclareMathOperator{\supp}{supp}
    \DeclareMathOperator{\spectrum}{spectrum}
    \DeclareMathOperator{\transpose}{t}
    \DeclareMathOperator{\order}{order}
    \DeclareMathOperator{\triv}{triv}
    \DeclareMathOperator{\sign}{sign}
    
	\newcommand{\bra}[1]{{\langle{#1}\vert}}
	\newcommand{\ket}[1]{{\vert {#1}\rangle}}
    \DeclareMathOperator{\UnitaryMultiplier}{UM}

\newcommand{\smallbmat}[1]{{\ensuremath \left[\begin{smallmatrix} #1 \end{smallmatrix}\right]}}
\renewcommand{\setminus}{-}
\renewcommand{\i}{\iota}

\makeatletter
\DeclareFontFamily{OMX}{MnSymbolE}{}
\DeclareSymbolFont{MnLargeSymbols}{OMX}{MnSymbolE}{m}{n}
\SetSymbolFont{MnLargeSymbols}{bold}{OMX}{MnSymbolE}{b}{n}
\DeclareFontShape{OMX}{MnSymbolE}{m}{n}{
    <-6>  MnSymbolE5
   <6-7>  MnSymbolE6
   <7-8>  MnSymbolE7
   <8-9>  MnSymbolE8
   <9-10> MnSymbolE9
  <10-12> MnSymbolE10
  <12->   MnSymbolE12
}{}
\DeclareFontShape{OMX}{MnSymbolE}{b}{n}{
    <-6>  MnSymbolE-Bold5
   <6-7>  MnSymbolE-Bold6
   <7-8>  MnSymbolE-Bold7
   <8-9>  MnSymbolE-Bold8
   <9-10> MnSymbolE-Bold9
  <10-12> MnSymbolE-Bold10
  <12->   MnSymbolE-Bold12
}{}

\let\llangle\@undefined
\let\rrangle\@undefined
\DeclareMathDelimiter{\llangle}{\mathopen}%
                     {MnLargeSymbols}{'164}{MnLargeSymbols}{'164}
\DeclareMathDelimiter{\rrangle}{\mathclose}%
                     {MnLargeSymbols}{'171}{MnLargeSymbols}{'171}
\makeatother

\title{Reduced $C^*$-algebras and $K$-theory for reductive $p$-adic groups}
\author{Pierre Clare, Tyrone Crisp}
\date{\empty}
 
\begin{document}

\maketitle

\begin{abstract}
We calculate the $K$-theory of the reduced $C^*$-algebra $C^*_r(G)$ of a reductive $p$-adic group $G$. To do so, we show that each direct summand in Plymen's Plancherel decomposition of $C^*_r(G)$ is Morita equivalent to a twisted crossed product for an action of a finite group on the blow-up of a compact torus along the zero-locus of a certain Plancherel density. It follows that the $K$-theory of $C^*_r(G)$ is the direct sum of the  twisted equivariant $K$-theory groups of these blow-ups, which can be computed using an Atiyah-Hirzebruch spectral sequence. As an illustration, the case of $\Sp_4$ is treated in some detail. Our main result is obtained from a more general study of $C^*$-algebras of compact operators on twisted equivariant Hilbert modules, from which we also recover results due to Wassermann for real groups, and to Afgoustidis and Aubert in the $p$-adic case, as well as a new description of the tempered dual of $G$ as a topological space.
\end{abstract}

\section{Introduction}

Let $G$ be a reductive $p$-adic group (i.e., the  group of $F$-points of a connected reductive group defined over a nonarchimedean local field $F$ of characteristic zero), and let $C^*_r(G)$ be the $C^*$-algebra generated by the left-regular representation of $G$ on $L^2(G)$. In this paper we give a new description of the algebra $C^*_r(G)$ up to Morita equivalence, and a new description of the $K$-theory of this $C^*$-algebra  that is amenable to computation. 

To help motivate this work on $p$-adic reductive groups,  let us first recall what is known in the parallel setting of \emph{real} reductive groups. In \cite{Wassermann} A.~Wassermann gave a simple description, up to Morita equivalence, of the $C^*$-algebra $C^*_r(G)$ of a real reductive group $G$, and he indicated how this computation leads to a proof of the  {Connes-Kasparov conjecture} for $G$, which identifies the $K$-theory of $C^*_r(G)$ in index-theoretic terms. (Full details of Wassermann's argument can be found in the union of \cite{CCH-1}, \cite{CHST}, and \cite{CHS}.) Several different proofs of the Connes-Kasparov conjecture  have since appeared, and one consequence of (and motivation for) the ongoing work on this problem has been to clarify the representation-theoretic content of the conjecture. It is now understood that the validity of the conjecture  is closely bound up with Mackey's observation of an analogy between the tempered dual of a real reductive group and the unitary dual of the associated Cartan motion group \cite{Higson_08,Afgoustidis-Mackey}; and also with Vogan's classification of the tempered dual in terms of minimal $K$-types \cite{BHY}.

The Connes-Kasparov conjecture is a special case of the {Baum-Connes conjecture} \cite{BCH}, which gives a formula for the $K$-theory of $C^*_r(G)$ for an arbitrary locally compact group $G$. This conjecture has been verified for many classes  of groups---including $p$-adic reductive groups, for which the conjecture (proved by V.~Lafforgue \cite{Lafforgue-CR,Lafforgue-Invent}, see also \cite{CEN_CK_almost_connected}) gives an isomorphism between the $K$-theory of $C^*_r(G)$ and the $G$-equivariant $K$-homology of the Bruhat-Tits building of $G$. The problem of elucidating the representation-theoretic content of this isomorphism for $p$-adic reductive groups has been discussed from various angles, in \cite[Section 6]{BCH}, \cite{BHP}, and \cite[Section 4.2]{ABPS}, for example; but compared to the substantial progress that has been made on the analogous problem for real groups, in the $p$-adic case there is much that remains to be understood. To quantify this difference we note that there are now at least four genuinely different proofs of the conjecture in the real case (\cite{Wassermann}, \cite{Lafforgue-Invent}, \cite{Afgoustidis-Mackey}, \cite{BHY}), while in the $p$-adic case we still only have Lafforgue's proof. (But let us add two qualifications to this last assertion. Firstly, \cite{BHP-GLn} gives a proof of the conjecture for the special case of $p$-adic $\GL_n$, that relies on ideas that are very different from those of Lafforgue's proof. Secondly, combining results from \cite{BBH}, \cite{Higson-Nistor}, \cite{Schneider}, \cite{Voigt}, and \cite{Solleveld-HP} gives a proof that the conjecture holds modulo torsion, that is different again from Lafforgue's approach.)

Among the proofs of the Baum-Connes/Connes-Kasparov conjecture for real reductive groups, the approach we are aiming at here is closest to the one set out in \cite{Wassermann}. Wassermann begins by noting that the Plancherel theorem for real reductive groups---due principally to Harish-Chandra---gives a decomposition of $C^*_r(G)$ as a direct sum of $C^*$-algebras of Weyl-group-invariant, compact-operator-valued functions on parameter spaces for the discrete-series representations of the Levi subgroups of $G$. (This result is modeled on a similar decomposition of the Harish-Chandra Schwartz algebra, due to Arthur \cite{Arthur}.)  Wassermann uses results about intertwining operators due to Knapp and Stein \cite{Knapp-Stein-II} to identify each summand, up to Morita equivalence, with a crossed product of a commutative $C^*$-algebra by a finite group. The $K$-theory of these crossed products is then computed, and shown to coincide with the computation of $K_*(C^*_r(G))$ given by the Connes-Kasparov map.

Parts of this approach have previously been adapted to $p$-adic groups. A version of the $C^*$-algebra Plancherel decomposition for $p$-adic reductive groups appears in \cite{Plymen-Reduced}, with some examples computed in \cite{Plymen-SL2}, \cite{Plymen-GLn}, and recently in \cite{Aubert-Plymen-Klein}, for instance.  Morita equivalences and $K$-theory computations similar to Wassermann's were established in some special cases in \cite{Plymen-Reduced}, \cite{Plymen-Leung}, \cite{Jawdat-Plymen}, and \cite{AA}, for example. In a somewhat different direction, Aubert, Baum, Plymen, and Solleveld (see for instance \cite{ABPS}) have formulated, and in many cases proved, a series of conjectures on the smooth dual of a $p$-adic reductive group $G$, among which is   a formula for $K_*(C^*_r(G))$ \cite[Conjecture 5]{ABPS} that has recently been verified in \cite[Corollary 6.26]{Solleveld-notes}.

In this paper we  prove that each of the direct-summands in the Plancherel decomposition of $C^*_r(G)$ is Morita equivalent to a twisted crossed product of a commutative $C^*$-algebra by a finite group. The $K$-theory of these summands can thus be expressed in terms of twisted equivariant $K$-theory \cite{Dwyer}, and can be computed using the general machinery of \cite{Luck-Chern}. In this way we obtain a new formula for the $K$-theory of $C^*_r(G)$, modulo the classification of the discrete series (a difficult open problem, in general).

To illustrate our results in a simple special case, consider the symplectic group $G=\Sp_4(F)$. Let $M$ be a minimal Levi subgroup of $G$, let $X\cong \T^2$ be the torus of unramified unitary characters of $M$, and let $W=N_G(M)/M$ be the Weyl group of $M$. To each character $\chi\in X$ there corresponds a parabolically induced representation of $G$, and in the `compact picture' of parabolic induction these representations all act on the same Hilbert space, which we shall denote by $H$. One of the direct-summands of $C^*_r(G)$ appearing in the Plancherel theorem of \cite{Plymen-Reduced} is the $C^*$-algebra of $W$-invariants $C(X,\Compact(H))^W$, where $\Compact$ denotes the compact operators, and where the action of $W$ on $C(X,\Compact(H))$ combines the coadjoint action on $X$ and conjugation by normalised intertwining operators $I_{w,\chi}\in \Unitary(H)$ (for $w\in W$ and $\chi\in X$.) It was noted in \cite{AA} that the Morita equivalence results of \cite{Plymen-Leung} and \cite{AA} do not apply to this summand, the chief obstacle being the fact that the  groups
\[
W'_\chi = \left\{ w\in W\ \middle|\ w\chi=\chi\text{ and the operator $I_{w,\chi}$  is a scalar } \right\}
\]
do not depend solely on the isotropy groups $W_\chi$. For instance, the action of $W$ on $X$ has two fixed points, but \cite[Theorem 1, p.364]{Keys} implies that only one of these two points (namely, the trivial character) has $W'_\chi=W$.

We show that the Plancherel formula can be modified, in this example, by replacing the torus $X$ by the space $\widetilde{X}$ obtained by slicing $X$ along the lines in the picture below:
\begin{center}
\includegraphics[scale=.15]{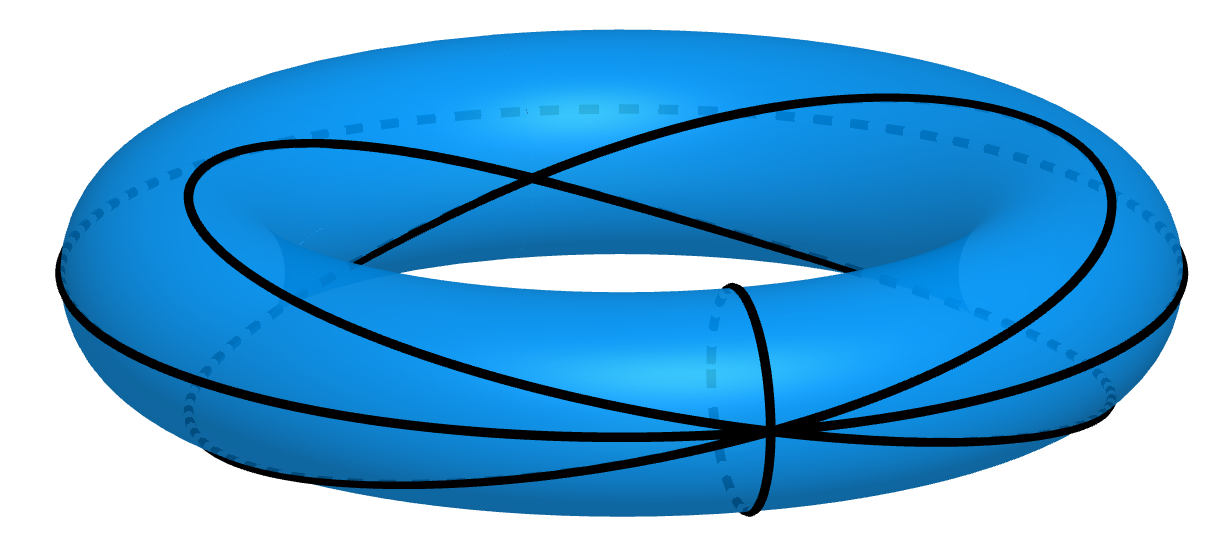}
\end{center}
These lines are the points $\chi\in X$ for which the group $W'_\chi$ is non-trivial \cite[Theorem 1, p.364]{Keys}. They are also the zeros of the Plancherel density associated to this component of the tempered dual of $G$ \cite{Silberger-dimension}. Slicing the torus $X$ in this way yields a new space $\widetilde{X}$ that is a disjoint union of four closed triangles. The action of the Weyl group $W$ on $X$ induces an action on $\widetilde{X}$, and it is not hard to see that the canonical `gluing' map $\widetilde{X}\to X$ induces an isomorphism of $C^*$-algebras $C(X,\Compact(H))^W \xrightarrow{\cong} C(\widetilde{X},\Compact(H))^W$. 

For each $\chi\in X$ the group $W'_\chi$ is the Weyl group of a root system \cite{Silberger-dimension}. The fibre of the gluing map $\widetilde{X}\to X$ over $\chi$ is naturally in bijection with the set of Weyl chambers for $W'_\chi$, and for each $\widetilde\chi$ in this fibre we have $W_{\widetilde\chi}=R_\chi$, the Knapp-Stein $R$-group at $\chi$ determined by the choice of Weyl chamber corresponding to $\widetilde{\chi}$. Thus the scalar-intertwining group $W'_{\widetilde\chi}$ is trivial for each $\widetilde\chi$, and this implies that the $C^*$-algebra $C(\widetilde{X},\Compact(H))^W$ is Morita equivalent to the crossed product $C(\widetilde{X})\rtimes W$. This implies, firstly, that the spherical principal-series component of the tempered dual of $G$ (i.e., the part of the tempered dual coming from parabolic induction of unitary unramified characters of $M$) is homeomorphic to the {extended quotient}  $\widetilde{X}/\!\!/ W$, in the sense of \cite[Section 2.1]{ABPS}, equipped with the Jacobson topology. Moreover, the Green-Julg theorem \cite{Julg} implies that the $K$-theory of the direct-summand of $C^*_r(G)$ associated with this component of the tempered dual is isomorphic to the $W$-equivariant $K$-theory of the space $\widetilde{X}$. This $K$-theory can easily be computed using the equivariant Atiyah-Hirzebruch spectral sequence \cite{Segal}, and we find that $K_0\cong \Z^2$ and $K_1=0$. 

Our results work in a similar way for a general reductive $p$-adic group $G$ and a discrete-series representation $\sigma$ of a Levi subgroup $M$ of $G$. We slice (or more precisely, we blow up, in the sense of \cite[Chapter 5]{Melrose}) the torus $X$ of unramified unitary characters of $M$ along certain hypersurfaces---the zeros of the Plancherel density associated to the family of parabolically induced $G$-representations $\Ind(\sigma\otimes\chi)$---to obtain a new space $\widetilde{X}$.  We prove that the direct-summand in the $C^*$-algebraic Plancherel formula for $G$ corresponding to the pair $(M,\sigma)$ is Morita equivalent to a twisted crossed product $C(\widetilde{X})\rtimes_\gamma W$, for an appropriate finite group $W$ and $2$-cocycle $\gamma:W\times W\to \T$. This gives a homeomorphism from the  component of the tempered dual of $G$ associated to the pair $(M,\sigma)$ to the {twisted extended quotient} $(\widetilde{X}/\!\!/ W)_\gamma$ (again in the sense of \cite[Section 2.1]{ABPS}). Our Morita equivalence result also implies that the $K$-theory of $C^*_r(G)$ is isomorphic to the direct sum of the twisted equivariant $K$-theory groups $K_*(C(\widetilde{X})\rtimes_\gamma W)\cong K^*_{W,\gamma}(\widetilde{X})$. These $K$-theory groups can be computed---at least rationally, and in low-rank cases exactly---using an Atiyah-Hirzebruch spectral sequence.

Returning to our initial motivation, the results in this paper do not yet furnish a new proof of the Baum-Connes conjecture, since our focus here is entirely on the `right-hand side' of the conjecture. Connecting our computation of $K_*(C^*_r(G))$  to the `left-hand side'---the equivariant $K$-homology of the Bruhat-Tits building---will likely require substantial additional work. 
 
In preparation for that work we have included in this paper a reformulation of our results in terms of the original parameter spaces $X$ (the tori of unramified unitary characters of the Levi subgroups of $G$), since these tori are more closely related to the Bruhat-Tits building than are our blown-up spaces $\widetilde{X}$. In terms of these tori, our results identify each of the summands $C(X,\Compact(H))^W$ appearing in the Plancherel decomposition of $C^*_r(G)$, up to Morita equivalence, with an ideal in the twisted crossed product $C(X)\rtimes_\gamma W$ that we identify explicitly; and we show that the $K$-theory of this ideal can be computed using a spectral sequence whose $E^2$ page is the equivariant cohomology of a coefficient system on $X$ that is defined in terms of certain projective representations of the $R$-groups $R_\chi$.

The paper is organised as follows. In Section \ref{sec:general} we establish an abstract Morita equivalence result for $C^*$-algebras of the form $\Compact_{C_0(X)}(E)^W$, where $X$ is a locally compact Hausdorff space, $W$ is a finite group, and $E$ is a twisted $W$-equivariant Hilbert $C_0(X)$-module.  
Though our focus in this paper is on $p$-adic groups, we note (in Example \ref{example:Wassermann}) that our abstract  result also gives a new proof of the Morita equivalences established for real reductive groups in \cite{Wassermann}. 

In Section \ref{sec:p-adic} we apply our abstract theorem to establish two new Morita equivalences for the Plancherel components the $C^*$-algebra of a $p$-adic reductive group, using the tori $X$ and their blow-ups $\widetilde{X}$, respectively.  In Section \ref{sec:K-theory} we use these equivalences, together with results of \cite{Dwyer} and \cite{Luck-Chern}, to compute the $K$-theory of $C^*_r(G)$ up to a convergent spectral sequence. Finally, Section \ref{sec:examples} contains several examples of our results, including a complete determination of the $K$-theory of all of the possible Plancherel components of $C^*_r(G)$ for the group $G=\Sp_4(F)$.

Let us conclude this introduction with a few words about the relationship between our results and the work of Aubert, Baum, Plymen and Solleveld. As we noted above, a different formula for $K_*(C^*_r(G))$ in terms of twisted equivariant $K$-theory, based on the Bernstein decomposition of $C^*_r(G)$, has recently been established by Solleveld in \cite[Corollary 6.26]{Solleveld-notes}, verifying the conjecture \cite[Conjecture 5]{ABPS}. (Earlier work of Solleveld, e.g. \cite{Solleveld-AKT,Solleveld-HH}, is also highly relevant to this computation.) Since the Plancherel decomposition refines the Bernstein decomposition, each of the equivariant $K$-theory groups appearing in Solleveld's formula is isomorphic to a direct sum of certain of the groups appearing in our formula. The isomorphisms are far from trivial, however, as we explain in Section \ref{sec:ABPS}; and the study of the underlying topological phenomena seems like an interesting direction for future work.

\paragraph*{Acknowledgements:} Our collaboration on this project was supported by funding from the American Mathematical Society, the Simons Foundation, the Fondation Sciences Mathématiques de Paris and the Université de Lorraine. We thank Anne-Marie Aubert, Nigel Higson, and Maarten Solleveld for helpful comments and corrections.

\section{Morita equivalences from twisted equivariant Hilbert modules}\label{sec:general}

Here, we shall review some basic aspects of the $C^*$-algebraic machinery to be used throughout the paper, assuming some familiarity with the language of Hilbert modules. An elementary introduction to these topics with an eye towards the representation theory of reductive groups can be found in \cite{Notes_Luminy}.

\subsection{Basic definitions}

\begin{definition}[See \S 2.4 in~\cite{Zeller-Meier}]\label{def:twisted-action-B}
A \emph{twisted action} of a finite group $W$ on a $C^*$-algebra $B$ is a pair of maps
\[
\alpha: W\to \Aut(B)\qquad \text{and}\qquad \gamma:W\times W\to \UnitaryMultiplier(B)
\]
(into the group of $*$-automorphisms and the group of unitary multipliers of $B$, respectively,) subject to the compatibility relations
\begin{itemize}
\item $\alpha_v(\alpha_w(b)) = \gamma(v,w) \alpha_{vw}(b) \gamma(v,w)^*$ for all $v,w\in W$ and $b\in B$,
\item $\alpha_u(\gamma(v,w)) \gamma(u,vw)=\gamma(uv,w)\gamma(u,v)$ for all $u,v,w\in W$, and
\item $\gamma(1_W,w)=\gamma(w,1_W)=1_{\UnitaryMultiplier(B)}$ for all $w\in W$, and $\alpha_{1_W}=1_{\Aut(B)}$.
\end{itemize}
To these data we associate the \emph{twisted crossed product} $C^*$-algebra
\[
B\rtimes_\gamma W = \left\{ \textstyle\sum_{w\in W} b_w w\ \middle|\ b_w\in B\right\},
\]
with multiplication
\[
b_1 w_1 \cdot b_2 w_2 = b_1 \alpha_{w_1}(b_2)\gamma(w_1,w_2) w_1w_2
\]
and involution
\[
(bw)^* = \gamma(w^{-1},w)^*\alpha_{w^{-1}}(b^*)w^{-1}.
\]
The norm on $B\rtimes_\gamma W$ is defined by choosing a faithful Hilbert-space representation $B\into \Bounded(H)$, and then representing $B\rtimes_\gamma W$ faithfully on the Hilbert space $\ell^2(W)\otimes H$ using the formula
\[
bw (\delta_v\otimes h) \coloneqq \delta_{wv}\otimes \alpha_{wv}^{-1}(b\gamma(w,v)).
\]
\end{definition}

We next consider twisted actions on Hilbert modules. Let $E$ be a right Hilbert module over a $C^*$-algebra $B$, with $B$-valued inner product $\langle\ |\ \rangle$. The $C^*$-algebra of compact operators on $E$---that is, the operator-norm closure of the span of the operators $\ket{\xi}\bra{\eta}:\zeta\mapsto \xi\langle\eta\ |\ \zeta\rangle$---is denoted by $\Compact_B(E)$. 

\begin{definition}\label{def:twisted-action-E}
Let $(\alpha,\gamma)$ be a twisted action of a finite group $W$ on a $C^*$-algebra $B$, and let $E$ be a right Hilbert $B$-module. An  \emph{$(\alpha,\gamma)$-twisted action} of $W$ on   $E$ is a map $U:W\to \GL(E)$ satisfying
\begin{itemize}
\item $U_w(\xi b)=U_w(\xi)\alpha_w(b)$ for all $w,\in W$, $\xi\in E$, and $b\in \UnitaryMultiplier(B)$,
\item $\langle U_w\xi\ |\ U_w\eta\rangle = \alpha_w\langle \xi\ |\ \eta\rangle$ for all $w\in W$ and $\xi,\eta\in E$, and
\item $U_v U_w(\xi) = U_{vw}(\xi)\gamma(v,w)^*$ for all $v,w\in W$ and $\xi\in E$.
\end{itemize}
Note that the second condition ensures that each $U_w$ is an isometry.
\end{definition}

Given an $(\alpha,\gamma)$-twisted action $U$ of $W$ on $E$ as above, we define
\begin{equation}\label{eq:xi-bw-definition}
\xi\cdot bw\coloneqq U_w^{-1}(\xi b) = U_{w^{-1}}(\xi)\alpha_{w^{-1}}(b)\gamma(w^{-1},w)
\end{equation}
and
\begin{equation}\label{eq:BW-valued-ip}
\llangle \xi\ |\ \eta\rrangle \coloneqq \sum_{w\in W}\langle\xi\ |\ U_w\eta\rangle w\in B\rtimes_\gamma W
\end{equation}
for all $\xi,\eta\in E$, $b\in B$, and $w\in W$.

We also define an action of the group $W$ on the $C^*$-algebra of compact operators $\Compact_B(E)$, by
\[
w:k\mapsto U_w k U_w^{-1}.
\]
This is an action by $*$-automorphisms, so the set of fixed points
\[
\Compact_B(E)^W \coloneqq \{k\in \Compact_B(E)\ |\ U_w k=kU_w\text{ for all }w\in W\}
\]
is a $C^*$-algebra.

\begin{lemma}\label{lem:E-BW-module}
The definitions \eqref{eq:xi-bw-definition} and \eqref{eq:BW-valued-ip} make $E$ into a right Hilbert $B\rtimes_\gamma W$-module, and we have 
\(
\Compact_{B\rtimes_\gamma W}(E) = \Compact_B(E)^W.
\)
\end{lemma}

\begin{proof}
Routine computations confirm that \eqref{eq:xi-bw-definition} defines a right $B\rtimes_\gamma W$-module structure on $E$, and that the inner product \eqref{eq:BW-valued-ip} satisfies $\llangle \xi \ |\ \eta bw\rrangle = \llangle \xi\ |\ \eta\rrangle bw$ and $\llangle \xi\ |\ \eta\rrangle^* = \llangle \eta\ |\ \xi\rrangle$ for all $\xi,\eta\in E$, $b\in B$, and $w\in W$. For each nonzero $\xi\in E$ we have $\llangle \xi\ |\ \xi\rrangle\neq 0$, because the coefficient of $1_W$ in this inner product is the nonzero element $\langle \xi\ |\ \xi\rangle\in B$. 

To prove that $\llangle \xi\ |\ \xi\rrangle\geq 0$ for all $\xi\in E$, choose a faithful representation $B\into \Bounded(H)$ and consider the faithful representation of $B\rtimes_\gamma W$ on $\ell^2(W)\otimes H$ from Definition \ref{def:twisted-action-B}. A straightforward computation shows that for each element $f=\sum_{w\in W} \delta_w \otimes h_w \in \ell^2(W)\otimes H$ we have
\begin{equation}\label{eq:BW-ip-positive}
\left\langle f\ \middle|\ \llangle \xi\ |\ \xi\rrangle f\right\rangle = \left\langle \sum_{x\in W} U_x^{-1}(\xi)\otimes h_x\ \middle|\ \sum_{y\in W} U_y^{-1}(\xi)\otimes h_y\right\rangle,
\end{equation}
where the right-hand side is the inner product in the Hilbert space $E\otimes_B H$. Thus $\llangle\xi\ |\ \xi\rrangle\geq 0$.

To see that $E$ is complete in the norm $\|\xi\|_{\llangle\ |\ \rrangle}=\|\llangle\xi\ |\ \xi\rrangle\|^{1/2}$, we will prove that this norm is equivalent to the norm $\|\xi\|_{\langle\ |\ \rangle}=\|\langle\xi\ |\ \xi\rangle\|^{1/2}$. For each $h\in H$ we have
\[
\left\langle \delta_1\otimes h\ \middle| \llangle \xi\ |\ \xi\rrangle (\delta_1\otimes h)\right\rangle = \left\langle h\ \middle|\ \langle\xi\ |\ \xi\rangle h\right\rangle,
\]
which implies that $\|\xi\|_{\llangle\ |\ \rrangle}\geq \|\xi\|_{\langle\ |\ \rangle}$. On the other hand, the identity \eqref{eq:BW-ip-positive} shows that $\|\xi\|_{\llangle\ |\ \rrangle}$ is equal to the norm of the operator
\[
t_\xi :\ell^2(W)\otimes H \to E\otimes_B H,\qquad \sum_{w\in W} \delta_w\otimes h_w \mapsto \sum_{w\in W} U_w^{-1}(\xi)\otimes h_w.
\]
Each of the operators $\sum_w \delta_w\otimes h_w \mapsto U_v^{-1}(\xi)\otimes h_v$ has norm equal to $\|\xi\|_{\langle\ |\ \rangle}$, and so
\[
\|\xi\|_{\llangle\ |\ \rrangle} = \|t_{\xi}\| \leq |W| \cdot \|\xi\|_{\langle\ |\ \rangle}.
\]

Turning to the algebra $\Compact_{B\rtimes_\gamma W}(E)$ of compact operators, for all $\xi,\eta\in E$ we have
\[
|\xi\rrangle\llangle \eta| = \sum_{w\in W} U_w^{-1} \ket{\xi}\bra{\eta} U_w,
\]
which is $|W|$ times the projection of $\ket{\xi}\bra{\eta}$ onto the fixed-point algebra $\Compact_{B}(E)^W$. Since on the one hand the span of the operators $|\xi\rrangle\llangle\eta|$ is dense in $\Compact_{B\rtimes_\gamma W}(E)$, and on the other hand the span of the $W$-fixed projections of the operators $\ket{\xi}\bra{\eta}$ is dense in $\Compact_B(E)^W$, we have $\Compact_{B\rtimes_\gamma W}(E)=\Compact_B(E)^W$ as claimed.
\end{proof}

\begin{remark}
The above construction gives an equivalence of categories, between twisted $W$-equivariant Hilbert $B$-modules and Hilbert $B\rtimes_\gamma W$-modules, that preserves the spaces of `compact' operators on each side. This is a version of the Green-Julg theorem \cite{Julg}.
\end{remark}

\subsection{\texorpdfstring{$C_0(X)\rtimes_\gamma W$ and its dual}{C0(X)⋊gamma W and its dual}}

We now specialise to the case where the $C^*$-algebra $B$ is commutative.

Let $X$ be a locally compact Hausdorff space, equipped with an action $\alpha$ of a finite group $W$. We will usually drop $\alpha$ from the notation for the action of $W$ on $X$ itself, but include $\alpha$ for the induced action on functions: thus if $b$ is a function on $X$ then for each $w\in W$ we let $\alpha_w b$ denote the function $\alpha_w b(x) \coloneqq b(w^{-1}x)$. With $\T$ denoting the unit circle in $\C$, let $\gamma:W\times W\to C(X,\T)$ be a map such that the pair $(\alpha,\gamma)$ forms a twisted action of $W$ on $C_0(X)$, in the sense of Definition \ref{def:twisted-action-B}: that is, 
\begin{equation}\label{eq:twisted-2-cocycle}
\alpha_u\gamma(v,w)\gamma(u,vw) = \gamma(u,v)\gamma(uv,w)
\end{equation}
for all $u,v,w\in W$, and $\gamma(1,w)=\gamma(w,1)=1$ (the constant function) for all $w\in W$. We note for future use that these conditions imply the identity
\[
\alpha_v\gamma(w,w^{-1})= 
\gamma(v,w)\gamma(vw,w^{-1}).
\]
For $x\in X$ and $v,w\in W$ we write $\gamma_x(v,w)$, instead of $\gamma(v,w)(x)$, for the evaluation of the function $\gamma(v,w):X\to \T$ at the point $x$.

The spectrum of the twisted crossed product $C_0(X)\rtimes_\gamma W$ can be computed by the `Mackey machine' (see for instance \cite[Chap.~2, \S12-13]{Echterhoff-MRG_machine}), as we shall now recall. 

\begin{example}\label{ex:imprimitivity}
Suppose that $\mathcal{O}$ is a transitive (hence  finite) $W$-space. For each $x\in \mathcal{O}$ we consider the projection $P_x\coloneqq \delta_x 1_W\in C(\mathcal{O})\rtimes_\gamma W$, where $\delta_x\in C(\mathcal{O})$ denotes the characteristic function of the point $x$. It is easily checked that the two-sided ideal of $C(\mathcal{O})\rtimes_\gamma W$ generated by $P_x$ is all of $C(\mathcal{O})\rtimes_\gamma W$, and that
\[
P_x (C(\mathcal{O})\rtimes_\gamma W) P_x = \left\{ \sum_{w\in W_x} b_w \delta_x w\ \middle|\ b_w\in \C \right\} \cong C^*_{\gamma_x}(W_x),
\]
where $C^*_{\gamma_x}(W_x)\coloneqq \C\rtimes_{\gamma_x} W_x$ is the $\gamma_x$-twisted group $C^*$-algebra of the finite group $W_x$. The bimodule $(C(\mathcal{O})\rtimes_\gamma W)P_x$ thus implements a Morita equivalence $C(\mathcal{O})\rtimes_\gamma W\sim C^*_{\gamma_x}(W_x)$. 
\end{example}

\begin{definition}
Let $\mathcal{O}$ be a $W$-orbit in $X$. We let $\res_{\mathcal{O}}:C_0(X)\rtimes_\gamma W\to C(\mathcal{O})\rtimes_\gamma W$ be the $*$-homomorphism defined by
\[
\res_{\mathcal{O}}(bw) \coloneqq b\restrict_{\mathcal{O}} w,
\]
for $b\in C_0(X)$ and $w\in W$.
\end{definition}

We shall need some basic terminology about projective representations of finite groups:

\begin{definition}
Let $G$ be a finite group, and let $\eta:G\times G\to \T$ be a map satisfying $\eta(v,w)\eta(u,vw)=\eta(u,v)\eta(uv,w)$ and $\eta(w,1)=\eta(1,w)=1$ for all $u,v,w\in W$. A \emph{unitary $\eta$-representation} of $G$ is a map $\pi : G\to \Unitary(H_\pi)$, where $H_\pi$ is a complex Hilbert space, satisfying $\pi(g_1)\pi(g_2)=\eta(g_1,g_2)\pi(g_1g_2)$ for all $g_1,g_2\in G$. As usual we say that $\pi$ is \emph{irreducible} if $H_\pi$ has no proper nonzero $G$-invariant subspace. We let $\widehat{G}^{\eta}$ denote the set of isomorphism classes of irreducible unitary $\eta$-representations of $G$. This is the same as the dual of the $C^*$-algebra $C^*_{\eta}(G)$.

We also recall, for future use, that the \emph{contragredient} of an $\eta$-representation $\pi$ is the $\overline{\eta}$-representation $\overline{\pi} : G\to \Unitary(\overline{H_\pi})$ given by $\overline{\pi}(g)\overline{h}\coloneqq \overline{\pi(g)h}$, where $\overline{H_\pi}$ is the complex conjugate of the Hilbert space $H_\pi$.
\end{definition}

We now describe the dual of $C_0(X)\rtimes_\gamma W$.

\begin{definition}\label{def:rho-x-pi}
Fix $x\in X$. Let $\mathcal{O}=Wx$, let $P_x=\delta_x 1_W\in C(\mathcal{O})\rtimes_\gamma W$, and let $F_x=\left(C(\mathcal{O})\rtimes_\gamma W\right)P_x$. In Example \ref{ex:imprimitivity} we observed that the $C(\mathcal{O})\rtimes_{\gamma}W$--$C^*_{\gamma_x}(W_x)$ bimodule $F_x$ is a Morita equivalence. Thus each $\pi\in \widehat{W_x}^{\gamma_x}$ induces an irreducible $*$-representation
\[
\rho_{x,\pi}:C(\mathcal{O})\rtimes_\gamma W \to \Bounded(F_x\otimes_{C^*_{\gamma_x}(W_x)} H_\pi),\qquad \rho_{x,\pi}(c)(f\otimes h) \coloneqq cf\otimes h.
\]
Composing $\rho_{x,\pi}$ with the restriction map $\res_{\mathcal{O}}:C_0(X)\rtimes_\gamma W\to C(\mathcal{O})\rtimes_\gamma W$ gives an irreducible representation of $C_0(X)\rtimes_\gamma W$, which we also denote by $\rho_{x,\pi}$. 
\end{definition}

\begin{lemma}\label{lem:Mackey-machine}
\begin{enumerate}[\rm(a)]
\item Every irreducible representation of $C_0(X)\rtimes_\gamma W$ is isomorphic to $\rho_{x,\pi}$ for some $x\in X$ and $\pi\in \widehat{W_x}^{\gamma_x}$.
\item $\rho_{x,\pi} \cong \rho_{y,\tau}$ if and only if there is a $w\in W$ such that $y=wx$ and $\tau \cong {}^w\pi$, where ${}^w\pi:W_{wx}\to \Unitary(H_\pi)$ is the unitary $\gamma_{wx}$-representation 
\[
{}^w\pi(v) \coloneqq \overline{\gamma_x(w^{-1},w)}\gamma_x(w^{-1},v)\gamma_x(w^{-1}v,w) \pi(w^{-1}vw).
\]
\end{enumerate}
\end{lemma}

\begin{proof}
If $b\in C_0(X)$ is a $W$-invariant function then $b1_W$ lies in the centre of $C_0(X)\rtimes_\gamma W$. Schur's lemma thus implies that each irreducible representation of $C_0(X)\rtimes_\gamma W$ factors through the restriction $\res_{\mathcal{O}}$ for a unique orbit $\mathcal{O}\subseteq X$. Since for each $x\in\mathcal{O}$ the bimodule $F_x$ is a Morita equivalence, the irreducible representations of $C(\mathcal{O})\rtimes_\gamma W$ are precisely those of the form $F_x\otimes_{C^*_{\gamma_x}(W_x)} H_\pi$, for $\pi$ an irreducible unitary $\gamma_x$-representation of $W_x$. This proves part (a).

For part (b), fix $x\in X$ and let $\mathcal{O}=Wx$. For each $w\in W$ the $*$-automorphism
\[
\Ad_w: C(\mathcal{O})\rtimes_{\gamma} W \to C(\mathcal{O})\rtimes_{\gamma} W,\qquad c\mapsto wcw^*
\]
satisfies $\Ad_w(P_x)=P_{wx}$, and so this map restricts to an isomorphism
\begin{equation}\label{eq:Adw-x-wx}
C^*_{\gamma_x}(W_x)\cong P_x\left( C(\mathcal{O})\rtimes_{\gamma} W\right)P_x \xrightarrow{\cong} P_{wx} \left( C(\mathcal{O})\rtimes_\gamma W\right) P_{wx} \cong C^*_{\gamma_{wx}}(W_{wx}).
\end{equation}
The map $\pi\mapsto {}^w\pi$, from $\widehat{W_x}^{\gamma_x}$ to $\widehat{W_{wx}}^{\gamma_{wx}}$, is the pushforward of representations along the isomorphism \eqref{eq:Adw-x-wx}. It follows from this that the map
\[
F_x\otimes_{C^*_{\gamma_x}(W_x)} H_\pi \to F_{wx}\otimes_{C^*_{\gamma_{wx}}(W_{wx})} H_\pi,\qquad f\otimes h\mapsto fw^*\otimes h
\]
is a unitary isomorphism of $C_0(X)\rtimes_\gamma W$-representations $\rho_{x,\pi}\cong \rho_{wx,{}^w \pi}$. 

For the ``only if'' assertion in (b), suppose that $\rho_{x,\pi}\cong \rho_{y,\tau}$. Then these representations coincide on the centre of $C_0(X)\rtimes_{\gamma} W$, so $x$ and $y$ must lie in the same $W$-orbit: thus $y=wx$ for some $w\in W$. We then have $\rho_{y,\tau}\cong \rho_{x,\pi}\cong \rho_{y,{}^w\pi}$, and since $F_y$ is a Morita equivalence this implies that $\tau\cong{}^w\pi$.
\end{proof}

\subsection{\texorpdfstring{The ideal $C(X,E,W)$}{The ideal C(X,E,W)}}

Let $E$ be a right Hilbert $C_0(X)$-module equipped with an $(\alpha,\gamma)$-twisted action $U$ of $W$, as in Definition \ref{def:twisted-action-E}. Let $\llangle\ |\ \rrangle$ be the $C_0(X)\rtimes_\gamma W$-valued inner product on $E$ defined by \eqref{eq:BW-valued-ip}, and define
\[
C(X,E,W)\coloneqq \overline{\lspan}\left\{ \llangle\xi\ |\ \eta\rrangle\in C_0(X)\rtimes_\gamma W\ \middle|\ \xi,\eta\in E\right\}.
\]

Each Hilbert module $F$ over a $C^*$-algebra $B$ induces a Morita equivalence between  the ideal $\overline{\lspan}\{\langle\xi\ |\ \eta\rangle\ |\ \xi,\eta\in F\}$ in $B$, and the $C^*$-algebra $\Compact_B(F)$. Lemma \ref{lem:E-BW-module} thus implies that $E$ implements a Morita equivalence between $C(X,E,W)$ and the fixed-point algebra $\Compact_B(E)^W$. The ideal $C(X,E,W)$ in $C_0(X)\rtimes_\gamma W$ determines, and is determined by, the closed subset
\[
\left\{ \rho_{x,\pi} \ \middle|\ \rho_{x,\pi}(C(X,E,W))\neq 0 \right\}
\]
of $\widehat{C_0(X)\rtimes_\gamma W}$. (See Definition \ref{def:rho-x-pi} for the meaning of $\rho_{x,\pi}$.) We are going to describe this set in terms of local data, at each $x\in X$.

For each $x\in X$ we let $\C_x=\C$ equipped with the representation $b\mapsto b(x)$ of $C_0(X)$. The Hilbert-module tensor product $E_x\coloneqq E\otimes_{C_0(X)} \C_x$ is a Hilbert space. For each $\xi\in E$ we let $\xi(x) \coloneqq \xi\otimes 1\in E_x$. We have $\langle \xi\ |\ \eta\rangle(x) = \langle \xi(x)\ |\ \eta(x)\rangle$ for all $\xi,\eta\in E$ and all $x\in X$.

For each $w\in W$ and each $x\in X$ the formula
\[
I_{w,x} : E_x\to E_{wx}, \qquad I_{w,x}(\xi(x)) \coloneqq (U_w\xi)(wx) 
\]
gives a well-defined unitary operator. Writing the same thing slightly differently, 
\[
(U_w\xi)(x) = I_{w,w^{-1}x}(\xi(w^{-1}x)).
\]
 
It will be useful to note the following local formula for the $C_0(X)\rtimes_\gamma W$-module structure on $E$: 
\begin{equation}\label{eq:xiw-def-I}
(\xi\cdot bw)(x) = I_{w,x}^*(\xi(wx))b(wx).
\end{equation}

The cocycle relation $U_v U_w(\xi) = U_{vw}(\xi)\gamma(v,w)^*$ ensures that for all $v,w\in W$ and $x\in X$ we have
\begin{equation}\label{eq:I-gamma}
I_{v,wx}I_{w,x} = \overline{\gamma_{vw x}(v,w)} I_{vw,x},
\end{equation}
where we recall that $\gamma_x(v,w)$ denotes the scalar $\gamma(v,w)(x)\in \T$. 
Thus the map $I_x:W_x\to\Unitary(E_x)$ defined by $w\mapsto I_{w,x}$ is a unitary $\overline{\gamma_x}$-representation.

\begin{definition}
For each $x\in X$ we define
\[
\supp_{E,x}\coloneqq \left\{ \pi\in \widehat{W_x}^{\gamma_x}\ \middle|\ \begin{aligned}
&\pi\text{ appears in the contragredient of the }\\ &\text{$\overline{\gamma_x}$-representation $I_x:w\mapsto I_{w,x}\in \Unitary(E_x)$}\end{aligned}\right\}.
\]
\end{definition}

\begin{proposition}\label{prop:CXEW-support}
For each $x\in X$ and $\pi\in \widehat{W_x}^{\gamma_x}$ we have $\rho_{x,\pi}(C(X,E,W))\neq 0$  if and only if $\pi\in \supp_{E,x}$. In other words (since each closed two-sided ideal in a $C^*$-algebra is an intersection of primitive ideals),
\begin{equation}\label{eq:CXEW-support}
C(X,E,W) = \bigcap_{Wx\in W\backslash X} \bigcap_{\substack{\pi\in \widehat{W_x}^{\gamma_x},\\ \pi\notin  \supp_{E,x}}} \ker \rho_{x,\pi}.
\end{equation}
\end{proposition}

\begin{proof}
Fix $x\in X$,  let $\mathcal{O}$ denote the orbit $Wx$, and let $E_{\mathcal{O}}$ denote the Hilbert $C(\mathcal{O})$-module $E\otimes_{C_0(X)} C(\mathcal{O})$. The given twisted action of $W$ on $E$ induces a twisted action $U_w(\xi\otimes b)\coloneqq U_w(\xi)\otimes \alpha_w(b)$ on $E_{\mathcal{O}}$, and so the formulas \eqref{eq:xi-bw-definition} and \eqref{eq:BW-valued-ip} make $E_{\mathcal{O}}$ into a right Hilbert $C(\mathcal{O})\rtimes_\gamma W$-module. For all $\xi,\eta\in E$ we have
\[
\res_{\mathcal{O}}\llangle \xi\ |\ \eta\rrangle = \left\llangle \xi\otimes 1_{\mathcal{O}}\ \middle|\ \eta\otimes 1_{\mathcal{O}}\right\rrangle,
\]
where $1_{\mathcal{O}}$ denotes the constant function $1$ on $\mathcal{O}$. This equality ensures that $\res_{\mathcal{O}}C(X,E,W) = C(\mathcal{O}, E\restrict_{\mathcal{O}}, W)$, and since each of the representations $\rho_{x,\pi}$ factors through $\res_{\mathcal{O}}$ we may assume without loss of generality that $X=\mathcal{O}$.

Next, recall that the irreducible representation $\rho_{x,\pi}$ of $C(\mathcal{O})\times_\gamma W$ is the one corresponding to the irreducible representation $\pi$ of $C^*_{\gamma_x}(W_x)$ via the full projection $P_x=\delta_x 1_W\in C(\mathcal{O})\rtimes_\gamma W$ (see Example \ref{ex:imprimitivity}). The map $J\mapsto P_x J P_x$ is a bijection between the set of ideals in $C(\mathcal{O})\rtimes_\gamma W$ and the set of ideals in $P_x\left(C(\mathcal{O})\rtimes_\gamma W\right)P_x\cong C^*_{\gamma_x}(W_x)$. This bijection preserves inclusions, and sends the primitive ideal $\ker \rho_{x,\pi}\subseteq C(\mathcal{O})\rtimes_\gamma W$ to the primitive ideal $\ker \pi \subseteq C^*_{\gamma_x}(W_x)$. Our goal is to prove that $C(\mathcal{O},E,W)=\bigcap_{\pi\not\in \supp_{E,x}} \ker \rho_{x,\pi}$, and the preceding remarks show that this is equivalent to proving that $P_x C(\mathcal{O},E,W)P_x = \bigcap_{\pi\not\in \supp_{E,x}} \ker \pi$.

We will now compute  $P_x C(\mathcal{O},E,W)P_x$. 
For all $\xi,\eta\in E$ we have
\[
\begin{aligned}
& P_x\left\llangle \xi \ \middle|\ \eta \right\rrangle P_x  = \left\llangle \xi \delta_x \ \middle|\ \eta \delta_x \right\rrangle  
= \sum_{w\in W} \left\langle \xi \delta_x\ \middle| \ U_w(\eta \delta_x)\right\rangle w \\
& =  \sum_{w\in W} \left\langle \xi \ \middle| \ U_w(\eta)\right\rangle \delta_x\delta_{wx} w = \sum_{w\in W_x} \left\langle \xi \ \middle| \ U_w(\eta )\right\rangle(x) \delta_x w \\
&=  \sum_{w\in W_x} \left\langle  \xi(x)\ \middle| \ U_w(\eta(x))\right\rangle \delta_x w  = \llangle \xi(x)\ |\ \eta(x)\rrangle \in C(x,E_x,W_x).
\end{aligned}
\]
Thus the bijection $J\mapsto P_x J P_x$ sends the ideal $C(\mathcal{O}, E, W)$ to the ideal $C(x,E_x, W_x)$ in $C^*_{\gamma_x}(W_x)$. 

We are left to prove that $C(x,E_x,W_x)=\bigcap_{\pi\not\in \supp_{E,x}} \ker \pi$, which follows from Schur orthogonality: for $\pi\in \widehat{W_x}^{\gamma_x}$, and $\xi(x),\eta(x)\in E_x$, we have
\[
\pi\left( \llangle \xi(x)\ |\ \eta(x)\rrangle\right) = \sum_{w\in W_x} \langle \xi(x)\ |\ I_{w,x}\eta(x) \rangle \pi(w),
\]
which is equal to $0$ for all $\xi,\eta$ if and only if $\pi$ is not contained in the contragredient of $I_{x}$.
\end{proof}

\begin{corollary}\label{cor:E-equivalence-easy}
The Hilbert module $E$ is a Morita equivalence between $\Compact_{C_0(X)}(E)^W$ and $C_0(X)\rtimes_\gamma W$ if and only if, for every $x\in X$, every irreducible $\gamma_x$-representation of $W_x$ occurs in the contragredient of the representation $I_x:w\mapsto I_{w,x}$.\hfill\qed
\end{corollary}

\subsection{Scalar intertwiners and the completeness condition}

In our study  of Morita equivalences for  $p$-adic reductive groups in Section \ref{sec:p-adic},  we will see that it is possible to arrange things so that the condition in Corollary \ref{cor:E-equivalence-easy}---that $\supp_{E,x} = \widehat{W_x}^{\gamma_x}$ for all $x\in X$---is satisfied. This condition is in general not satisfied, however, for the most natural choice of the space $X$ (namely, the torus of unramified unitary characters of a Levi subgroup of our reductive group); in order to apply Corollary \ref{cor:E-equivalence-easy} we will need to work with a modification of this space. Since there may be situations where one wants to work directly with the torus of unramified characters, we will now give an explicit description of the ideal $C(X,E,W)$ under a less stringent hypothesis than the one in Corollary \ref{cor:E-equivalence-easy}, that is satisfied in the case of these tori.

An obvious reason why we might have $\supp_{E,x}\neq \widehat{W_x}^{\gamma_x}$ is if there is a non-identity element $w\in W_x$ whose associated operator $I_{w,x}$ is a scalar; then $\supp_{E,x}$ can only contain representations in which $w$ acts by the complex-conjugate of that scalar. This is the obstruction that we will now deal with. 

\begin{definition}
For each $x\in X$ we define
\[
W_x' \coloneqq \{w\in W_x\ |\ I_{w,x} \text{ is a scalar operator on $E_x$} \}.
\]
For $w\in W'_x$ we let $\i_{w,x}\in\T$ be the scalar satisfying $I_{w,x}=\i_{w,x}\id_{E_x}$.
\end{definition}

\begin{lemma}\label{lem:Wprime-normal}
For each $x\in X$ and each $w\in W$ we have $wW'_x w^{-1} = W'_{wx}$, and for each $v\in W'_x$ we have
\begin{equation}\label{eq:i-conjugation}
\i_{wvw^{-1},wx} = \i_{v,x} \gamma_{wx}(wv,w^{-1})\gamma_{wx}(w,v)\overline{\gamma_{wx}(w,w^{-1})}.
\end{equation}
In particular, $W'_x$ is a normal subgroup of $W_x$, and the projection
\[
P_{\i,x}\coloneqq \frac{1}{|W'_x|}\sum_{v\in W'_x} \i_{v,x}v 
\]
is central in $C^*_{\gamma_x}(W_x)$.
\end{lemma}

\begin{proof}
For $w\in W$ and $v\in W'_x$ the relation \eqref{eq:I-gamma} gives
\begin{align*}
I_{wvw^{-1},wx} &= \gamma_{wx}(wv,w^{-1})I_{wv,x}I_{w^{-1},wx} \\
& = \gamma_{wx}(wv,w^{-1})\gamma_{wx}(w,v) I_{w,x} I_{v,x} I_{w^{-1},wx}\\
& = \i_{v,x} \gamma_{wx}(wv,w^{-1})\gamma_{wx}(w,v)\overline{\gamma_{wx}(w,w^{-1})} \id_{E_{wx}}.
\end{align*}
This proves \eqref{eq:i-conjugation}. Taking $w\in W_x$ in \eqref{eq:i-conjugation} shows that $W'_x$ is normal in $W_x$ and that $P_{\i,x}$ is central in $C^*_{\gamma_x}(W_x)$.
\end{proof}

For each unitary $\gamma_x$-representation $\pi:W'_x\to \Unitary(H_\pi)$, the operator $\pi(P_{\i,x})$ is the orthogonal projection of $H_\pi$ onto the subspace
\[
\pi(P_{\i,x})H_\pi=\left\{h\in H_\pi\ \middle|\ \pi(w')h=\overline{\i_{w',x}}h\text{ for all }w'\in W'_x\right\}.
\]
If $\pi$ is the restriction to $W'_x$ of a unitary $\gamma_x$-representation of $W_x$, then the fact that $P_{\i,x}$ is central in $C^*_{\gamma_x}(W_x)$ ensures that $\pi(P_{\i,x})H_\pi$ is a $W_x$-invariant subspace of $H_\pi$. In particular, if $\pi$ is irreducible then this subspace is either $0$ or all of $H_\pi$. 

\begin{definition}\label{def:lie-over}
We say that an irreducible unitary $\gamma_x$-representation $\pi$ of $W_x$ \emph{lies over $\overline{\i_x}$} if $\pi(w')=\overline{\i_{w',x}}\id_{H_\pi}$ for all $w'\in W'_x$. (Or, equivalently, if $H_\pi$ contains a nonzero vector $h$ satisfying $\pi(w')h=\overline{\i_{w',x}}h$ for all $w'\in W'_x$.) The set of isomorphism classes of such representations is denoted by $\widehat{W_x}^{\gamma_x}_{\overline{\i_x}}$.
\end{definition}

Recall that we defined $\supp_{E,x}$ to be the subset of $\widehat{W_x}^{\gamma_x}$ consisting of those representations that occur in the contragredient of the representation $I_x:w\mapsto I_{w,x}$. Since the subgroup $W'_x$ acts through the scalar character $\i_x$ in the representation $I_x$, and thus by the scalar $\overline{\i_x}$ in the contragredient, we necessarily have
\[
\supp_{E,x}\subseteq \widehat{W_x}^{\gamma_x}_{\overline{\i_x}}.
\]

\begin{definition}\label{def:completeness-condition}
We say that the module $E$ satisfies the \emph{completeness condition} if for each $x\in X$ we have $\supp_{E,x} = \widehat{W_x}^{\gamma_x}_{\overline{\i_x}}$; that is, if every irreducible unitary $\gamma_x$-representation of $W_x$ that lies over $\overline{\i_x}$ appears in the contragredient of the representation $I_x$. 
\end{definition}

Note that if the group $W'_x$ is trivial for every $x\in X$, then the completeness condition is the condition that appears in Corollary \ref{cor:E-equivalence-easy}: i.e., that $\supp_{E,x}=\widehat{W_x}^{\gamma_x}$ for every $x$.

\begin{lemma}\label{lem:C-completeness}
If the module $E$ satisfies the completeness condition, then we have
\[
C(X,E,W) = \left\{ \sum_{w\in W} b_w w\in C_0(X)\rtimes_{\gamma} W\ \middle|\ \begin{aligned} &\forall x\in X,\ \forall w'\in W_x',\ \forall w\in W:\\ & b_{w'w}(x)  = \i_{w',x}\gamma_{x}(w',w) b_w(x)\end{aligned} \right\}.
\]
\end{lemma}

\begin{proof}
Temporarily let $D(X,E,W)$ denote the right-hand side of the equality asserted in Lemma \ref{lem:C-completeness}. We know (from Proposition \ref{prop:CXEW-support}) that $C(X,E,W)$ is the intersection of the ideals $\res_{\mathcal{O}}^{-1} C(\mathcal{O},E_{\mathcal{O}}, W)$, as $\mathcal{O}$ ranges through the orbit space $W\backslash X$. On the other hand, the fact that $D(X,E,W)$ is defined by a pointwise condition ensures that $D(X,E,W)=\bigcap_{\mathcal{O}} \res_{\mathcal{O}}^{-1} D(\mathcal{O},E_{\mathcal{O}},W)$. So it will suffice to restrict our attention to a single orbit $\mathcal{O}$, and prove that $C(\mathcal{O},E_{\mathcal{O}}, W)=D(\mathcal{O},E_{\mathcal{O}},W)$.

Let $P_{\i,\mathcal{O}}\in C(\mathcal{O})\rtimes_\gamma W$ be the element
\begin{equation}\label{eq:PO-def}
P_{\i,\mathcal{O}}\coloneqq \sum_{z\in \mathcal{O}} |W'_z|^{-1}\sum_{w\in W'_z} \i_{w,z} \delta_z w.
\end{equation}
Straightforward computations, using Lemma \ref{lem:Wprime-normal}, show that $P_{\i,\mathcal{O}}$ is a central projection in $C(\mathcal{O})\rtimes_\gamma W$, and that we have
\begin{equation}\label{eq:PO}
D(\mathcal{O}, E_{\mathcal{O}}, W) = P_{\i,\mathcal{O}}(C(\mathcal{O})\rtimes_\gamma W).
\end{equation}

The equality \eqref{eq:PO} shows that $D(\mathcal{O},E_{\mathcal{O}},W)$ is an ideal in $C(\mathcal{O})\rtimes_\gamma W$, and that for each $x\in \mathcal{O}$ and each $\pi\in \widehat{W_x}^{\gamma_x}$ we have $\rho_{x,\pi}(D(\mathcal{O},E_{\mathcal{O}},W))\neq 0$ if and only if $\rho_{x,\pi}(P_{\i,\mathcal{O}})\neq 0$. A further computation shows that $P_{\i,\mathcal{O}}P_x = P_{\i,x}$, where $P_x=\delta_x 1_W$ as in Example \ref{ex:imprimitivity}, and $P_{\i,x}$ is as defined in Lemma \ref{lem:Wprime-normal}. Recalling that for each $\pi\in \widehat{W_x}^{\gamma_x}$ we defined $\rho_{x,\pi} = (C(\mathcal{O})\rtimes_\gamma W)P_x\otimes_{C^*_{\gamma_x}(W_x)}H_\pi$, we find that 
\[
\begin{aligned}
\rho_{x,\pi}(P_{\i,\mathcal{O}})=0 \iff \pi(P_{\i,\mathcal{O}}P_x)=0 \iff \pi(P_{\i,x})=0 & \iff \pi\not\in \widehat{W_x}^{\gamma_x}_{\overline{\i_x}} \\
& \iff \pi\not\in \supp_{E,x}
\end{aligned}
\]
where the last equivalence is the completeness condition. Proposition \ref{prop:CXEW-support} thus implies that the ideals $C(\mathcal{O},E_{\mathcal{O}},W)$ and $D(\mathcal{O},E_{\mathcal{O}},W)$ are annihilated by the same irreducible representations of $C(\mathcal{O})\rtimes_\gamma W$, so these ideals are equal.
\end{proof}

Combining Lemmas \ref{lem:E-BW-module} and \ref{lem:C-completeness}, we arrive at:

\begin{theorem}\label{thm:general-Morita}
Let $(\alpha,\gamma)$ be a twisted action of a finite group $W$ on a commutative $C^*$-algebra $C_0(X)$, and let $E$ be a Hilbert $C_0(X)$-module equipped with an $(\alpha,\gamma)$-twisted action of $W$, satisfying the completeness condition. The module $E$ implements a Morita equivalence between the fixed-point algebra 
$\Compact_{C_0(X)}(E)^W$ and the ideal 
\[
C(X,E,W) = \left\{ \sum_{w\in W} b_w w\in C_0(X)\rtimes_{\gamma} W\ \middle|\ \begin{aligned} &\forall x\in X,\ \forall w'\in W_x',\ \forall w\in W:\\ & b_{w'w}(x)  = \i_{w',x}\gamma_{x}(w',w) b_w(x)\end{aligned} \right\}
\]
in the twisted crossed product $C_0(X)\rtimes_\gamma W$.
\hfill\qed
\end{theorem}

\subsection{Localisation}\label{subsec:localisation}

The localisation procedure that was used to restrict to a single $W$-orbit in the proof of Proposition \ref{prop:CXEW-support} can be generalised, as follows. Let $X$, $W$, $\alpha$, $\gamma$, $E$, and $U$ be as above, and let $Y$ be a $W$-invariant subset of $X$ that is locally closed (that is, $Y=Y_1\setminus Y_2$, where $Y_2$ and $Y_1$ are closed subsets of $X$; we can assume without loss of generality that $Y_1$ and $Y_2$ are $W$-invariant). Then $Y$ is a locally compact Hausdorff space in the subspace topology, and all of the structure that we have defined for $X$ restricts to $Y$: 
\begin{itemize}
\item The action $\alpha$ restricts to an action $\alpha_Y$ of $W$ on $Y$, because $Y$ is assumed $W$-invariant.
\item The  $2$-cocycle $\gamma:W\times W\to C(X,\T)$ yields, upon restriction of functions from $X$ to $Y$, a $2$-cocycle $\gamma_Y:W\times W\to C(Y,\T)$.
\item The Hilbert module $E$ can be localised to the Hilbert $C_0(Y)$-module $E_Y\coloneqq E\otimes_{C_0(X)} C_0(Y)$. Here $C_0(X)$ acts on $C_0(Y)$ via the restriction map $C_0(X) \to C_b(Y)$. 
\item The $(\alpha,\gamma)$-twisted action $U$ of $W$ on $E$ yields an $(\alpha_Y,\gamma_Y)$-twisted action $U_Y$ of $W$ on $E_Y$, through the formula $U_{Y,w}(\xi\otimes b)\coloneqq U_w(\xi)\otimes \alpha_w(b)$.
\end{itemize}

Localisation is functorial with respect to inclusions, in the sense that if $Z\subseteq Y \subseteq X$, where $Z$ and $Y$ are locally closed in $X$, then we have $(\alpha_Y)_Z=\alpha_Z$, $(\gamma_Y)_Z=\gamma_Z$, and $(E_Y)_Z\cong E_Z$ via the canonical isomorphism 
\[
C_0(Y)\otimes_{C_0(Y)} C_0(Z) \xrightarrow{b\otimes c\mapsto bc} C_0(Z).
\]
In the special case where $Y$ is actually a closed subset of $X$, we have a surjective restriction map $E\xrightarrow{\xi\mapsto \xi\restrict_Y} E_Y$, given by writing each $\xi\in E$ in the form $\eta b$ for some $\eta\in E$ and $b\in C_0(X)$, and then defining $\xi\restrict_Y\coloneqq \eta\otimes b\restrict_Y\in E\otimes_{C_0(X)}C_0(Y)$. 

The following observation will be useful in $K$-theory computations.

\begin{lemma}\label{lem:C-localisation}
For each $W$-invariant closed subset $Y$ of $X$, the canonical short exact sequence
\[
0 \to C_0(X\setminus Y)\rtimes_{\gamma_{X\setminus Y}} W \to C_0(X)\rtimes_\gamma W \xrightarrow{\res_Y} C_0(Y)\rtimes_{\gamma_Y} W\to 0
\]
restricts to a short exact sequence
\[
0 \to C(X\setminus Y, E_{X\setminus Y}, W) \to C(X,E,W) \to C(Y,E_Y,W)\to 0.
\]
\end{lemma}

\begin{proof}
The identity $\res_Y\llangle \xi\ |\ \eta\rrangle = \llangle \xi\restrict_Y\ |\ \eta\restrict_Y\rrangle$ ensures that $\res_Y(C(X,E,W))=C(Y,E_Y,W)$. On the other hand, the identity  \eqref{eq:CXEW-support} implies that $C(X,E,W)\cap (C_0(X\setminus Y)\rtimes_\gamma W)=C(X\setminus Y, E_{X\setminus Y},W)$.
\end{proof}

As a special case, suppose that $X=Y\sqcup Z$, a disjoint union of $W$-invariant open subsets. Then we have a canonical isomorphism
\begin{equation}\label{eq:C-direct-sum}
C(X,E,W) \cong C(Y,E_Y,W)\oplus C(Z,E_Z,W).
\end{equation}

\subsection{A special case}

In this section we explain the relationship between our Morita equivalence result Theorem \ref{thm:general-Morita} and an earlier result of Afgoustidis and Aubert \cite{AA} (which generalised a still earlier result of Plymen and Leung \cite{Plymen-Leung}, which was in turn an adaptation of an argument of Wassermann \cite{Wassermann}.)

\begin{corollary}\label{cor:AA}
Let $X$, $E$, $W$,  $\alpha$, and $\gamma$ be as above. Suppose that the completeness condition (Definition \ref{def:completeness-condition}) is satisfied, and assume in addition that:
\begin{enumerate}[\rm(1)]
\item $\i_{w',x}=1$ for all $x\in X$ and all $w'\in W'_x$;
\item $\gamma_x(v,w)=\gamma_y(v,w)$ for all $x,y\in X$ and all $v,w\in W$; 
\item there are subgroups $W'$ and $R$ in $W$ such that $W=W'\rtimes R$;
\item there is a point $x_0\in X$ with $W_{x_0}=W$; and
\item $W'_x=W_x\cap W'$ for all $x\in X$.
\end{enumerate}
Then the $C^*$-algebra $\Compact_{C_0(X)}(E)^W$ is Morita equivalent to $C_0(X/W')\rtimes_{\gamma} R$.
\end{corollary}

\begin{proof}
In view of condition (2) we may regard the cocycle $\gamma$ as taking values in $\T$. We first note that $\gamma(v,w)=1$ for all $v,w\in W'$: indeed, conditions (4) and (5) together imply that $W'=W'_{x_0}$, and then using  the relation \eqref{eq:I-gamma} and condition (1) we find that for $v,w\in W'$ we have
\[
\gamma(v,w) = \i_{vw,x_0}\overline{\i_{v,x_0}\i_{w,x_0}} =1.
\]

Theorem \ref{thm:general-Morita} implies that $\Compact_{C_0(X)}(E)^W$ is Morita equivalent to the $C^*$-algebra
\[
C=C(X,E,W) = \left\{ \sum_{w\in W} b_w w\in C_0(X)\rtimes_{\gamma} W\ \middle|\ \begin{aligned} &\forall x\in X,\ \forall w'\in W_x',\ \forall w\in W:\\ & b_{w'w}(x)  = \gamma(w',w) b_w(x)\end{aligned} \right\}.
\]
The decomposition $W=W'\rtimes R$ gives a decomposition of $C^*$-algebras $C_0(X)\rtimes_{\gamma} W = (C_0(X)\rtimes W')\rtimes_\gamma R$, where on the right-hand side the crossed product $C_0(X)\rtimes W'$ has no twisting (because, as we just observed, $\gamma$ is identically $1$ on $W'$); and where the $R$-action on $C_0(X)\rtimes W'$ is given by 
\[
\beta_r(b w') = r\cdot bw'\cdot r^* = \alpha_r(b) \gamma(r,w')\gamma(rw',r^{-1})\overline{\gamma}(r^{-1},r) rw'r^{-1}
\]
for $r\in R$, $b\in C_0(X)$, and $w'\in W'$. Condition (5) implies that under this identification we have $C=C'\rtimes_\gamma R$, where $C' = C(X,E,W')$.

Now consider the Hilbert $C_0(X)$-module $B=C_0(X)$, equipped with the $W'$-action $\alpha$. This is an action by $(\alpha,1)$-twisted unitaries, where $1$ denotes the  $2$-cocycle $(v,w)\mapsto 1$.  We have $\Compact_{C_0(X)}(B)=C_0(X)$, and $\Compact_{C_0(X)}(B)^{W'} = C_0(X)^{W'}\cong C_0(X/W')$. Each of the Hilbert spaces $B_x$ is one-dimensional, and for each $w\in W'_x$ the operator on $B_x$ induced by $\alpha_w$ is the identity. It follows that $C(X,B,W')=C(X,E,W')=C'$, and so Theorem \ref{thm:general-Morita} applied to the module $B$ gives a Morita equivalence between $C'$ and $C_0(X/W')$. 

The group $R$ acts on the bimodule $B$ through the $C^*$-algebra automorphisms $\alpha_r$. It is easy to check that these operators satisfy the conditions of \cite[Theorem 2.3]{Bui-1995} with respect to the inner products ${}_{C_0(X)^{W'}}\langle b\ |\ c\rangle = \sum_{w\in W'}\alpha_w(bc^*)$  and $\langle b \ |\ c\rangle_{C'} = \llangle b \ |\ c \rrangle = \sum_{w\in W'} \xi\alpha_w(\eta)w$ (see \eqref{eq:BW-valued-ip}), and so the cited theorem of Bui  implies that $C_0(X)^{W'}\rtimes_\gamma R$ is Morita equivalent to $C'\rtimes_{\gamma} R$. We observed above that $C'\rtimes_{\gamma} R = C$ is Morita equivalent to $\Compact_{C_0(X)}(E)^W$, so this completes the proof.
\end{proof}

Corollary \ref{cor:AA} is closely related to \cite[Theorem 1.4]{AA}. The assumptions (1)--(4) of Corollary \ref{cor:AA} are all explicitly imposed in \cite{AA}, but condition (5) is not. Instead, in the situation considered in \cite{AA} one has for each $x\in X$ a decomposition $W_x = W_x'\rtimes R^x$ (where we are using a superscript rather than a subscript to avoid confusion with the stabiliser of $x$ inside the group $R=R_{x_0}$), and Afgoustidis and Aubert impose the following assumption (cf.~\cite[Assumption 1.3]{AA}):
\begin{enumerate}
\item[(5$'$)] For each $x\in X$ we have $W'_x\subseteq W_x\cap W'$, and $R^x$ is isomorphic to a subgroup of $R$.
\end{enumerate}
In the general setting that we are considering here, conditions (1)--(4) and (5$'$) are not equivalent to (1)--(5), and the former set of conditions does not always imply that $\Compact_{C_0(X)}(E)^W$ is Morita equivalent to $C_0(X/W')\rtimes_\gamma R$. However, in the examples of $p$-adic classical groups considered in \cite[Part II]{AA} Afgoustidis and Aubert prove that whenever condition (5$'$) is satisfied, one in fact has the stronger property that $R^x$ is actually a subgroup of $R$ for every $x\in X$. This stronger property is easily seen to imply our property (5) (in the presence of properties (3) and (4)).

\begin{example}\label{example:Wassermann}
Let $G$ be a \emph{real} reductive group. Let $P$ be a parabolic subgroup of $G$, with Langlands decomposition $P=MAN$, and let $\sigma$ be a discrete-series representation of $M$. The group $W= \{w\in N_G(A)/M\ |\ {}^w\sigma\cong\sigma\}$ acts in a natural way on the real vector space $\germ{a}^*$, where $\germ{a}$ is the Lie algebra of the split component $A$. This action lifts to an action on the module $C_0(\germ{a}^*,\Ind_P^G(H_\sigma))$ by normalised intertwining operators as constructed by Knapp and Stein \cite{Knapp-Stein-II}. Results of Knapp and Stein (particularly \cite[Lemma 14.1]{Knapp-Stein-II}) show that the point $x_0=0\in \germ{a}^*$ satisfies the hypotheses of Corollary \ref{cor:AA}, and we thus recover the Morita equivalence
\begin{equation}\label{eq:Wassermann}
C_0(\germ{a}^*,\Compact(\Ind_P^G(H_\sigma)))^W \sim C_0(\germ{a}^*/W')\rtimes R
\end{equation}
first established in \cite{Wassermann}. See \cite{CCH-1} and \cite{CHST} for further details. We will give another proof of \eqref{eq:Wassermann} below, in  Remark \ref{rem:Wassermann-proof2}.
\end{example}

\section{\texorpdfstring{Reduced $C^*$-algebras of $p$-adic reductive groups}{Reduced C*-algebras of p-adic reductive groups}}\label{sec:p-adic}

\subsection{Background on parabolic induction and intertwining operators}\label{sec:p-adic-background}

Let $F$ be a nonarchimedean local field of characteristic $0$, and let $G$ be the group of $F$-points of a connected reductive group $\mathbb{G}$ defined over $F$. (See, e.g., \cite[Chapter V]{Renard} for the basic structural theory of $p$-adic reductive groups that we shall use here.) 

Let $P$ be a parabolic subgroup of $G$ with Levi factor $M$ and unipotent radical $N$. (As usual, we mean by this that $P$, $M$, and $N$ are the groups of $F$-points of appropriate subgroups of $\mathbb{G}$ that are defined over $F$.) Fix a discrete-series representation $\sigma:M\to \Unitary(H_\sigma)$: this means that $\sigma$ is an irreducible unitary representation of $M$ whose matrix coefficients are square-integrable modulo the centre of $M$. Let $K$ be a maximal compact subgroup of $G$, in \emph{good position} with respect to $M$ (see \cite[\S 4.4]{Bruhat-Tits}) so that $G=KP$. 

Let $X$ be the compact torus of {unitary unramified characters} of $M$: that is, continuous homomorphisms $M\to \T$ that are trivial on the open subgroup
\[
M^\circ \coloneqq \left\{m\in M\ \middle|\ |\phi(m)|_F=1\text{ for all }\phi\in \Hom_F(M,F^\times) \right\}.
\]
Here $|\, \cdot\, |_F$ denotes the absolute value on $F$, and $\Hom_F(M,F^\times)$ is the group of $F$-rational characters of $M$. The group $M/M^\circ$ is free-abelian of finite rank, and so its Pontryagin dual $X$ is isomorphic, as a topological group, to $\T^n$ for some $n$. 

Let $W_M = N_G(M)/M$. This finite group acts on $X$ by $s:\chi\mapsto {}^s\chi$, where ${}^s\chi(m)=\chi(\tilde{s}^{-1}m\tilde{s})$, with $\tilde{s}\in N_G(M)$ any representative of the coset $s\in W_M$. The semidirect product group $X\rtimes W_M$ then acts on $X$ by $(\psi s)\chi\coloneqq \psi {}^s\chi$. The group $X\rtimes W_M$ also acts on the set $E_2(M)$ of isomorphism classes of discrete-series representations of $M$, by $(\psi s):\tau\mapsto {}^s\tau\otimes \psi^{-1}$, where ${}^s\tau$ denotes the isomorphism class of the representation $m\mapsto \tau(\tilde{s}^{-1}m\tilde{s})$, for any representative $\tilde{s}\in G$ of $s\in W_M$. 

\begin{remark}
Note that we have chosen to let $\psi\in X$ act on $X$ by multiplying by $\psi$, and on $E_2(M)$ by tensoring with $\psi^{-1}$. This choice is convenient for what follows, though it does introduce a potential for ambiguity in the case where $M$ is compact modulo centre, and $X$ is (canonically identified with) a subset of $E_2(M)$. In practice, we shall only refer to the action on $E_2(M)$ in Definition \ref{def:W} and Lemma \ref{lem:W-WM}, and so the ambiguity is relatively benign.
\end{remark}

\begin{definition}\label{def:W}
Let
\[
W \coloneqq (X\rtimes W_M)_{\sigma} = \{ \psi s\in X\rtimes W_M\ |\ {}^s \sigma\cong \sigma\otimes\psi\}.
\]
\end{definition}

It is clear from the definitions that the quotient map $X\rtimes W_M\to W_M$ restricts to a surjection $W\to W_M(X\sigma)$, where
\[
W_M(X\sigma) \coloneqq \{s\in W_M\ |\ {}^s\sigma \cong \sigma\otimes \psi\text{ for some }\psi\in X\};
\]
and the kernel of this surjection is
\[
X_\sigma\coloneqq \{\psi\in X\ |\ \sigma\otimes \psi\cong \sigma\}.
\]
In particular, since the groups $X_\sigma$ and $W_M$ are finite (see, eg, \cite[V.2.7 Lemma]{Renard}), so is $W$. The isotropy groups for the action of $W$ on $X$ are (canonically isomorphic to) the isotropy groups for the action of $W_M$ on the orbit $X\sigma\subset E_2(M)$:

\begin{lemma}\label{lem:W-WM}
For all $\chi_1,\chi_2\in X$, the quotient map $X\rtimes W_M\to W_M$ restricts to a bijection
\[
\{\psi s\in W\ |\ (\psi s)\cdot\chi_1 = \chi_2\} \xrightarrow{\cong} \{s\in W_M\ |\ {}^s(\sigma\otimes\chi_1)\cong \sigma\otimes \chi_2\},
\]
whose inverse is given by $s\mapsto {}^s\overline{\chi_1}\chi_2 w$. In particular, for each $\chi\in X$ the projection $X\rtimes W_M\to W_M$ restricts to an isomorphism of groups
\[
W_\chi \xrightarrow{\cong} (W_M)_{\sigma\otimes \chi} = \{s\in W_M\ |\ {}^s(\sigma\otimes\chi)\cong \sigma\otimes \chi\},
\]
with inverse $s\mapsto {}^s\overline{\chi}\chi s$.
\end{lemma}

\begin{proof}
For all $\psi\in X$ and $s\in W_M$ we have
\[
\begin{aligned}
& \psi s\in W\ \text{and}\ (\psi s)\cdot \chi_1=\chi_2 \\
\iff & {}^s\sigma\cong \sigma\otimes \psi\ \text{and}\ \psi {}^s\chi_1 = \chi_2 \\
\iff &{}^s\sigma\otimes {}^s\chi_1 \cong \sigma\otimes \psi{}^s\chi_1\ \text{and}\ \psi{}^s\chi_1=\chi_2\\
\iff &{}^s(\sigma\otimes\chi_1) \cong \sigma\otimes\chi_2\ \text{and}\ \psi=\overline{{}^s\chi_1}\chi_2.
\end{aligned}
\]
So if ${}^s(\sigma\otimes\chi_1)\cong \sigma\otimes\chi_2$ then $\psi\coloneqq {}^s\overline{\chi_1}\chi_2$ is the unique element of $X$ for which $\psi s\in W$ and $(\psi s)\cdot \chi_1=\chi_2$.
\end{proof}

We next consider the family of unitary parabolically induced representations $\Ind_P^G(\sigma\otimes\chi)$ for $\chi\in X$, defined by pulling the representation $\sigma\otimes\chi$ of $M$ back to the parabolic subgroup $P$ along the quotient map $P\to M$, and then applying Mackey's unitary induction functor. Since  $G=KP$ every $P$-equivariant function on $G$ is determined by its restriction to the compact subgroup $K$, and so we can realise these induced representations on a Hilbert space of $K\cap P$-equivariant functions $K\to H_\sigma$; this is the `compact picture' of parabolic induction. Since the unramified characters $\chi$ are trivial on every compact subgroup of $G$, the Hilbert space in the compact picture does not depend on $\chi$. We denote this Hilbert space by $H$. Explicitly, 
\begin{equation}\label{eq:HInd-definition}
H = \big\{\xi\in L^2(K,H_\sigma)\ \big|\ \xi(kmn)=\sigma(m)^* \xi(k)\text{ for all }k\in K,m\in K\cap M,n\in K\cap N \big\}
\end{equation}
with inner product $\langle\xi\ |\ \eta\rangle \coloneqq \int_K \langle \xi(k)\ |\ \eta(k)\rangle\, \dd k$. 
See for instance \cite[Section 2.1]{AA} for further details.

For each $s\in W_M$ and each $\chi\in X$ we consider the normalised intertwining operator
\[
R_{s^{-1}Ps|P}(\sigma\otimes\chi):\Ind_P^G(\sigma\otimes\chi) \to \Ind_{s^{-1}Ps}^G(\sigma\otimes\chi)
\]
as in \cite[Theorem 2.1]{Arthur-Intertwining-I} (where our $\sigma\otimes\chi$ is denoted $\sigma_\chi$). These operators can be used to define intertwining operators between the representations $\Ind_P^G(\sigma\otimes\chi)$, via the following construction from \cite[Section 2]{Arthur-Elliptic}.

\begin{definition}\label{def:pI}
\begin{enumerate}[(1)]
\item For each $s\in W_M=N_G(M)/M$ we choose a lift $\tilde{s}\in K\cap N_G(M)$.
\item For each $w=\psi s\in {W}$ we choose an $M$-equivariant unitary isomorphism $T_{w}: {}^{\tilde{s}}\sigma \to \sigma\otimes \psi$.
\item For each $T\in\Unitary(H_\sigma)$ we let $\mu_T\in \Unitary(L^2(K,H_\sigma))$ be the operator of pointwise multiplication by $T$: $(\mu_T\xi)(k)=T \xi(k)$.
\item For each $\ell\in K$ we let $\rho_\ell\in \Unitary (L^2(K,H_\sigma))$ be the operator of right-translation by $\ell$: $(\rho_\ell\xi)(k) = \xi(k\ell)$.
\item We then define, for each $w=\psi s\in W$ and each $\chi\in X$, a unitary operator $I_{w,\chi}:H\to H$ by the following commuting diagram:
\[
\xymatrix@C=60pt{
\Ind_P^G(\sigma\otimes \chi)\ar[d]^-{\coloneqq I_{w, \chi}}  \ar[r]^-{R_{s^{-1}Ps|P}(\sigma\otimes\chi)} & \Ind_{s^{-1}Ps}^G(\sigma\otimes \chi) \ar[d]^-{\rho_{\tilde{s}}}  \\  \Ind_P^G(\sigma\otimes \psi {}^s\chi) & \Ind_P^G({}^{\tilde{s}}\sigma\otimes {}^s\chi) \ar[l]_-{\mu_{T_w}} 
}
\]
\end{enumerate}
\end{definition}

Let us explain how in the special case $w\in W_\chi$, the above construction is the same as the one in \cite[Section 2]{Arthur-Elliptic}. In \cite{Arthur-Elliptic} Arthur associates an intertwining operator $R(s,\sigma\otimes \chi):\Ind_P^G(\sigma\otimes\chi)\to \Ind_P^G(\sigma\otimes \chi)$ to each $s\in (W_M)_{\sigma\otimes \chi}$ by choosing an isomorphism $J_s:{}^{\tilde s}(\sigma\otimes\chi)\to \sigma\otimes \chi$, and then letting $R(s,\sigma\otimes\chi)$ be the operator making the diagram
\[
\xymatrix@C=60pt{
\Ind_P^G(\sigma\otimes \chi)\ar[d]^-{\coloneqq R(s,\sigma\otimes\chi)}  \ar[r]^-{R_{s^{-1}Ps|P}(\sigma\otimes\chi)} & \Ind_{s^{-1}Ps}^G(\sigma\otimes \chi) \ar[d]^-{\rho_{\tilde s}}  \\  \Ind_P^G(\sigma\otimes \chi) & \Ind_P^G({}^{\tilde s}\sigma\otimes {}^s\chi) \ar[l]_-{\mu_{J_s}} 
}
\]
commute. The operator $J_s$ can alternatively be viewed as an isomorphism $T_{w}:{}^{\tilde s}\sigma\to \sigma \otimes \psi$, where $\psi={}^s{\overline{\chi}}\chi$ is the unique element of $X$ for which $\psi s\in W_\chi$ (cf.~Lemma \ref{lem:W-WM}). So the choice of a $J_s$ for each $s\in (W_M)_{\sigma\otimes\chi}$, as in Arthur's construction, is equivalent to the choice of a $T_{w}$ for each $w\in W_\chi$, as in Definition \ref{def:pI}. Comparing the two commutative diagrams above shows that choosing $J_s=T_{w}$ gives
\begin{equation}\label{eq:I-R-equality}
I_{ w,\chi} = R(s,\sigma\otimes\chi)\quad\text{ for all }\quad w=\psi s\in W_\chi.
\end{equation}

The next theorem collects some important properties of the intertwining operators $I_{w, \chi}$, due principally to Arthur, Langlands, and Harish-Chandra.

\begin{theorem}\label{thm:pI-properties}
\begin{enumerate}[\rm(1)]
\item For all $\chi\in X$, $w\in W$, and $g\in G$ we have
\[
I_{w,\chi}\Ind_P^G(\sigma\otimes\chi)(g) = \Ind_P^G(\sigma\otimes w\chi)(g) I_{ w,\chi}.
\]
\item For all $w_1=\psi_1 s_1$ and $w_2=\psi_2 s_2$ in $W$ we have
\[
I_{w_2,  w_1\chi}I_{ w_1,\chi} = \overline{\gamma}( w_2,w_1) I_{ w_2 w_1,\chi},
\]
where $\gamma : {W}\times {W}\to \T$ is the $2$-cocycle defined by
\[
T_{w_2}T_{w_1} = \overline{\gamma}( w_2, w_1) T_{ w_2 w_1}\sigma\left((\widetilde{s_2s_1})^{-1}\widetilde{s_2}\widetilde{s_1}\right),
\]
and where the meaning of $\tilde{s}$ is as in Definition \ref{def:pI}(1).
\item For each $w\in {W}$ the map $\chi\mapsto I_{w,\chi}$ is continuous in the {strong-operator} topology on $\Unitary(H)$.
\item For all $\chi_1,\chi_2\in X$ we have
\[
\Bounded_G\left(\Ind_P^G(\sigma\otimes\chi_1),\Ind_P^G(\sigma\otimes\chi_2)\right) = \lspan\{ I_{ w, \chi_1}\ |\  w\in W,\ w\chi_1 = \chi_2\},
\]
where the left-hand side is the space of $G$-equivariant bounded linear maps on $H$ that intertwine the indicated representations of $G$.
\end{enumerate}
\end{theorem}

\begin{proof}
Part (1) holds because each of the three other arrows in the diagram defining $I_{w,\chi}$ is an intertwiner between the indicated representations. 

As explained in \cite[Proposition 2.4]{AA}, the equality (2) follows from a straightforward computation using properties of the operators $R_{s^{-1}Ps|P}(\sigma\otimes\chi)$ established in \cite[Theorem 2.1]{Arthur-Intertwining-I}. 

To prove part (3) we recall from \cite[Theorem 2.1]{Arthur-Intertwining-I} that the intertwining operators $R_{s^{-1}Ps|P}(\sigma\otimes\chi)$ are in fact defined for $\chi$ in the complex torus of not-necessarily-unitary unramified characters, and that each matrix coefficient of this family of operators is a meromorphic function of $\chi$, regular on the subset $X$ of unitary characters. In particular each of these matrix coefficients is continuous on $X$, and so the operators $I_{ w,\chi}$ are continuous in the weak-operator topology. To get from here to strong-operator continuity, note that $H$ can be written as a direct sum of finite-dimensional subspaces, each of which is invariant under all of the operators $I_{ w, \chi}$: namely, the isotypical components for the compact open subgroup $K\subset G$. The weak-operator topology coincides with the norm topology on each of these finite-dimensional components, and then an easy approximation argument shows that the map $\chi\mapsto I_{w,\chi}$ is strong-operator continuous.

In part (4), the case where $\chi_1=\chi_2$ is essentially Harish-Chandra's `commuting algebra theorem' (cf.~\cite[Theorem 5.5.3.2]{Silberger-book}), together with the identification \eqref{eq:I-R-equality}. In the general case, \cite[Proposition III.4.1]{Waldspurger-Plancherel} implies that there is a nonzero intertwiner $\Ind_P^G(\sigma\otimes\chi_1)\to \Ind_P^G(\sigma\otimes\chi_2)$ if and only if there is an $s\in W_M$ such that ${}^s(\sigma\otimes \chi_1)\cong \sigma\otimes \chi_2$. The element  $v\coloneqq {}^s\overline{\chi_1}\chi_2 s\in W$ then satisfies $v\chi_1=\chi_2$ (cf.~Lemma \ref{lem:W-WM}), and we have
\begin{align*}
\Bounded_G\left(\Ind_P^G(\sigma\otimes\chi_1),\Ind_P^G(\sigma\otimes\chi_2)\right) &= I_{v,\chi_1}\Bounded_G\left(\Ind_P^G(\sigma\otimes\chi_1),\Ind_P^G(\sigma\otimes\chi_1)\right) \\
& = \lspan\{I_{v,\chi_1}I_{ w,\chi_1}\ |  w\in W_{\chi_1}\} \\
& = \lspan\{I_{w,\chi_1}\ |\ w\in W,\ w\chi_1 = \chi_2\}.\qedhere
\end{align*}
\end{proof}

Note that the isomorphisms $T_{w}$ in Definition \ref{def:pI} can  be chosen so that 
\begin{equation}\label{eq:gamma-inverse}
\gamma\big( w, w^{-1}\big)=1\quad\text{ for all }\quad w\in W.
\end{equation}
Indeed, if we first choose the $T_w$s arbitrarily, then for each pair $w,w^{-1}$ we choose a square root $c_{w,w^{-1}}$ of $\gamma(w,w^{-1})=\gamma(w^{-1},w)\in \T$, then the renormalised operators $T_{w}^{\operatorname{new}}\coloneqq c_{w,w^{-1}}T_{ w}^{\operatorname{old}}$ satisfy \eqref{eq:gamma-inverse}. We will henceforth assume that \eqref{eq:gamma-inverse} holds.

\begin{remark}\label{rem:gamma}
It is known that in many cases the cocycle $\gamma$ is a coboundary, which implies that by rescaling the operators $T_w$ we can arrange that $\gamma \equiv 1$. As a simple example, \eqref{eq:gamma-inverse} implies that we can always arrange $\gamma\equiv 1$ if $W$ has order $2$. Less trivially, if $G$ is a quasi-split classical group then results of \cite{Herb}, \cite {Goldberg-Sp-SO}, and \cite{Goldberg-U} imply that $\gamma$ can always be trivialised in this way. (This has previously been remarked in \cite[Remarque 1.6]{Waldspurger-formule-2e} and \cite[Remark 2.1]{AA}.)
\end{remark}

\subsection{A twisted action from intertwining operators}

Recall that the finite group ${W}$ acts on the torus $X$ by homeomorphisms. We let $\alpha$ denote the induced action of $W$ on $C(X)$ by $*$-automorphisms. Letting $H$ continue to denote the Hilbert space \eqref{eq:HInd-definition}, we consider the right Hilbert $C(X)$-module  $E=C(X,H)$, with inner product $\langle \xi\ |\ \eta\rangle(\chi) \coloneqq \langle \xi(\chi)\ |\ \eta(\chi)\rangle$.

\begin{definition}\label{def:pU}
For each $w\in {W}$ and each $\xi\in E=C(X,H)$ we let $U_{ w}\xi:X\to H$ be the function
\[
(U_{w}\xi)(\chi) \coloneqq I_{ w, w^{-1}\chi} \xi(w^{-1}\chi).
\]
\end{definition}

\begin{lemma}
The operators $U_{ w}$ define an $(\alpha,\gamma)$-twisted action of ${W}$ on $E$.
\end{lemma}

\begin{proof}
The fact that the operators $I_{w,\chi}$ are unitary, and vary strong-operator-continuously with $\chi$, ensures that for each $\xi\in C(X,H)$ and each $w\in W$ the function $U_{w}\chi$ is continuous on $X$. The relations $U_{ w}(\xi b) = U_{w}(\xi)\alpha_{w}(b)$ and $\langle U_{ w}\xi\ |\ U_{ w}\eta\rangle = \alpha_{ w}\langle \xi\ |\ \eta\rangle$
for all $w\in W$, $\xi\in E$, and $b\in C(X)$ follow immediately from the definitions. As for the third condition in Definition \ref{def:twisted-action-E}, let $w_1,w_2\in W$; then for $\chi \in X$ we have
\[
\begin{aligned}
U_{ w_1} U_{ w_2}(\xi)(\chi)  &= I_{w_1,w_1^{-1}\chi}I_{w_2,w_2^{-1}w_1^{-1}\chi}\xi\left(w_2^{-1}w_1^{-1}\chi\right) \\
& = \overline{\gamma}(w_1, w_2)I_{w_1w_2,w_2^{-1}w_1^{-1}\chi} \xi\left(w_2^{-1}w_1^{-1}\chi\right) \\
& = \overline{\gamma}(w_1, w_2)U_{w_1 w_2}(\xi)(\chi)
\end{aligned}
\]
where the second property in Theorem \ref{thm:pI-properties} is used in the second line.
\end{proof}

\subsection{\texorpdfstring{The Plancherel theorem for $C^*_r(G)$}{The Plancherel theorem for C*r(G)}}

In this section we recall Plymen's $C^*$-algebra version of the Plancherel theorem for $p$-adic reductive groups \cite[2.5 Theorem]{Plymen-Reduced}. Our approach differs slightly from that of \cite{Plymen-Reduced} in that we use the torus $X$, rather than the orbit $X\sigma\subset E_2(M)$, as the parameter space for the induced representations $\Ind_P^G(\sigma\otimes\chi)$. Most of the arguments are, nevertheless, essentially the same as those of \cite{Plymen-Reduced}, and we make no claim to novelty here.

Each of the unitary representations $\Ind_P^G(\sigma\otimes\chi):G\to \Unitary(H)$ is tempered (see \cite[Lemme III.2.3]{Waldspurger-Plancherel}), and thus induces a $C^*$-algebra representation $\Ind_P^G(\sigma\otimes\chi):C^*_r(G)\to \Bounded(H)$.

Consider the Hecke algebra $C_c^\infty(G)$ of compactly supported, locally constant functions on $G$. Every $f\in C_c^\infty(G)$ is invariant under translation by some compact open subgroup $J$ of $G$, which implies that for each $\chi\in X$ the operator $\Ind_P^G(\sigma\otimes\chi)(f)$ factors through the orthogonal projection of $H$ onto its subspace of $J$-fixed vectors. The latter subspace is finite-dimensional, because the representation $\Ind_P^G(\sigma\otimes\chi)$ is admissible (\cite{Bernstein-tame} and \cite[Lemme III.2.3]{Renard}), and so $C_c^\infty(G)$ acts by finite-rank operators in $\Ind_P^G(\sigma\otimes\chi)$. Since $C_c^\infty(G)$ is dense in $C^*_r(G)$, we thus have $\Ind_P^G(\sigma\otimes\chi):C^*_r(G)\to \Compact(H)$.

Now for each $g\in G$ and each compact open subgroup $J\subseteq K$ consider the family of operators 
\begin{equation}\label{eq:Ind-delta-JgJ}
\Ind_P^G(\sigma\otimes\chi)(\delta_{JgJ})\in \Compact(H^J)\subseteq \Compact(H),
\end{equation}
where $\delta_{JgJ}\in C_c^\infty(G)$ is the characteristic function of the double coset $JgJ$, and $H^J$ is the space of $J$-fixed vectors in $H$ (which is independent of $\chi$). It is a straightforward matter to verify that for each $g'\in G$ the function $\chi\mapsto \Ind_P^G(\sigma\otimes\chi)(g')$ is continuous in the strong-operator topology on $\Unitary(H)$, and hence that the operators \eqref{eq:Ind-delta-JgJ} vary strong-operator-continuously with $\chi$. Since $H^J$ is finite dimensional these operators in fact vary norm-continuously. Functions of the form $\delta_{JgJ}$ span $C_c^\infty(G)$, and so a density argument implies that for each $f\in C^*_r(G)$ the function $\chi\mapsto \Ind_P^G(\sigma\otimes\chi)(f)$ is continuous in the norm topology. 

We thus obtain a $*$-homomorphism
\begin{equation}\label{eq:pCstar-map}
C^*_r(G) \to C(X,\Compact(H)),\qquad f\mapsto \big(\chi\mapsto \Ind_P^G(\sigma\otimes\chi)(f)\big).
\end{equation}

\begin{definition}\label{def:Cstar-block}
We let $C^*_r(G)_{(M,\sigma)}$  denote the image of the homomorphism \eqref{eq:pCstar-map}.
\end{definition}

As in the general situation considered in Section \ref{sec:general}, the operators $U_{w}$ induce an action $k\mapsto U_{ w}k U_{ w}^{-1}$ of ${W}$ on the $C^*$-algebra $C(X,\Compact(H))=\Compact_{C(X)}(E)$ by $*$-automorphisms. In terms of the operators $I_{w,\chi}$, we have for each $w\in W$, $k\in C(X,\Compact(H))$, and $\chi\in X$,
\[
\big(U_{w}k U_{ w}^{-1}\big)(\chi) = I_{ w, w^{-1}\chi} k(w^{-1}\chi) I_{ w, w^{-1}\chi}^*.
\]
The fixed-point algebra $C(X,\Compact(H))^W$ is thus
\begin{equation}\label{eq:k-fixed-I}
C(X,\Compact(H))^W = \{k\in C(X,\Compact(H))\ |\ I_{ w,\chi}k(\chi) = k(w\chi)I_{ w,\chi}\text{ for all }\chi\in X\}.
\end{equation}

\begin{theorem}[Plymen] \label{thm:Plancherel}
\begin{enumerate}[\rm(1)]
\item For each Levi subgroup $M$ of $G$, and each discrete-series representation $\sigma$ of $M$, we have $C^*_r(G)_{(M,\sigma)}=C(X,\Compact(H))^{W}$.
\item The homomorphisms \eqref{eq:pCstar-map} assemble to give an isomorphism of $C^*$-algebras $C^*_r(G) \cong \bigoplus_{[M,\sigma]} C^*_r(G)_{(M,\sigma)}$, where the direct sum is over the set of equivalence classes of discrete-series representations of Levi subgroups of $G$, with $(M_1,\sigma_1)\sim (M_2,\sigma_2)$ if and only if $M_2=gM_1 g^{-1}$ and $\sigma_2 \cong {}^g\sigma_1\otimes \chi$ for some $g\in G$ and $\chi\in X_{M_2}$.
\end{enumerate}
\end{theorem}

\noindent Note that the direct sum appearing in part (2) is the direct sum of $C^*$-algebras, i.e., the $C_0$ direct sum.

\begin{proof}
The fact that the operators $I_{ w, \chi}$ are intertwiners (part (1) of Theorem \ref{thm:pI-properties}) ensures that $C^*_r(G)_{(M,\sigma)}\subseteq C(X,\Compact(H))^{W}$. 

The reverse inclusion follows from part (4) of Theorem \ref{thm:pI-properties}, together with Kaplansky's Stone-Weierstrass theorem, which says that if $B$ is a $C^*$-subalgebra of a CCR $C^*$-algebra $A$, such that every $B$-equivariant bounded operator between irreducible $A$-representations is in fact $A$-equivariant, then $B=A$. (See \cite[11.1.6]{Dixmier}.) To verify the hypothesis of Kaplansky's theorem for $B=C^*_r(G)_{(M,\sigma)}$ and $A=C(X,\Compact(H))^W$, note that each irreducible representation of $C(X,\Compact(H))^W$ factors through one of the (surjective) evaluation maps $k\mapsto k(\chi)\in \Compact(H)^{W_\chi}$ (where $W_\chi$ acts on $\Compact(H)$ by conjugation by the unitary operators $I_{w,\chi}$); that the composition of the homomorphism \eqref{eq:pCstar-map} with evaluation at $\chi$ is precisely the representation $\Ind_P^G(\sigma\otimes\chi):C^*_r(G)\to \Compact(H)^{W_\chi}$; and that every irreducible representation of $\Compact(H)^{W_\chi}$ can be realised on some subspace of $H$. Harish-Chandra's commuting algebra theorem (part (4) of Theorem \ref{thm:pI-properties}) ensures that every $G$-intertwiner between the representations $\Ind_P^G(\sigma\otimes\chi)$ is a linear combination of the intertwining operators $I_{w,\chi}$, which are by definition also intertwiners for $C(X,\Compact(H))^W$. Thus the hypothesis of Kaplansky's theorem is satisfied, and we conclude that $C^*_r(G)_{(M,\sigma)} = C(X,\Compact(H))^W$.

To prove part (2), one first notes that Bernstein's uniform admissibility theorem \cite{Bernstein-tame} implies that the homomorphism $\Phi:C^*_r(G)\to \prod_{[M,\sigma]} C^*_r(G)_{(M,\sigma)}$ given by combining the homomorphisms \eqref{eq:pCstar-map} actually maps into the direct sum (cf.~\cite[VI.10.5]{Renard}, \cite[Th\`eor\`eme VIII.1.2]{Waldspurger-Plancherel}.) Every irreducible representation of $C^*_r(G)$ is isomorphic to a subrepresentation of some $\Ind_P^G(\sigma\otimes\chi)$ (\cite[Proposition III.4.1(i)]{Waldspurger-Plancherel}), so $\Phi$ is injective. On the other hand, if $(M_1,\sigma_1)\not\sim(M_2,\sigma_2)$ then for all $\chi_1\in X_{M_1}$ and $\chi_2\in X_{M_2}$ we have 
\[
\Bounded_G\left(\Ind_{P_1}^G(\sigma_1\otimes\chi_1),\Ind_{P_2}^G(\sigma_2\otimes\chi_2)\right)=0
\]
(\cite[Proposition III.4.1(ii)]{Waldspurger-Plancherel}), and so another application of Kaplansky's Stone-Weierstrass theorem ensures that $\Phi$ is surjective.
\end{proof}

\subsection{Scalar intertwiners}

We first recall some notation from Section \ref{sec:general}. For each $\chi\in X$ we consider the normal subgroup
\[
W_\chi'\coloneqq \{w\in W_\chi\ |\ I_{ w,\chi}\text{ is a scalar}\}
\]
of the stabiliser $W_\chi$, and for each $ w\in W'_\chi$ we let $\i_{ w,\chi}\in \T$ be the scalar satisfying $I_{ w,\chi}=\i_{ w,\chi}\id_H$. The map $\i_\chi:w\mapsto \i_{w,\chi}$ is a one-dimensional $\overline{\gamma}$-representation of $W'_\chi$. The following argument, due to Knapp and Stein, shows that the variability of the groups $W'_\chi$ and the characters $\i_{w,\chi}$ is rather tightly constrained:

\begin{lemma}[cf.~{\cite[Lemma 14.1]{Knapp-Stein-II}}]\label{lem:homotopy}
Let $t\mapsto \chi_t$ ($t\in [0,1]$) be a continuous path in $X$, and suppose that $w\in W$ satisfies $w\in W_{\chi_t}$ for all $t$, and $w\in W'_{\chi_0}$. Then for all $t\in [0,1]$ we have $w\in W'_{\chi_t}$ and $\i_{w,\chi_t}=\i_{w,\chi_0}$.
\end{lemma}

\begin{proof}
Fix $w$ as in the lemma, and consider the family of operators $I_t\coloneqq I_{ w ,\chi_t}$ (cf.~Definition \ref{def:pI}). Recall that these operators have the following properties (cf.~Theorem \ref{thm:pI-properties}):
\begin{enumerate}[\rm(i)]
\item Each $I_t\in \Unitary(H)$ is an intertwiner of the parabolically induced representation $\Ind_P^G(\sigma\otimes \chi_t)$ of $G$.
\item The map $t\mapsto I_t$ is continuous for the strong-operator topology on $\Unitary(H)$.
\item If $w$ has order $n$ in the finite group ${W}$, then for each $t$ we have $I_t^n = \prod_{j=1}^{n-1}\overline{\gamma}( w , w ^j)\in \T$.
\end{enumerate}
(Property (iii) follows from the identity $I_{ w ,\chi_t}I_{ w ^j,\chi_t} = \overline{\gamma}( w , w ^j) I_{ w ^{j+1},\chi_t}$ of Theorem \ref{thm:pI-properties}(2).)

Let $K_0 \supset K_1 \supset K_2 \supset \cdots$ be a sequence of compact open subgroups of our chosen maximal compact subgroup $K\subset G$, with $\bigcap_j K_j=\{1\}$. The parabolically induced representations $\Ind_P^G(\sigma\otimes\chi):G\to \Unitary(H)$, for $\chi\in X$, all coincide on the subgroup $K$, and so the subspaces $H^{K_j}$ of $K_j$-fixed vectors do not depend on $\chi$. Moreover, since the induced representations of $G$ are admissible, each of the subspaces $H^{K_j}$ is finite-dimensional; and since $\bigcap_j K_j= \{1\}$ we have $\overline{\bigcup_j H^{K_j}}=H$.

Property (i) of the operators $I_t$ ensures that each $I_t$ restricts to an operator on each of the subspaces $H^{K_j}$. Since these subspaces are finite-dimensional, property (ii) implies that the map $t\mapsto I_t\restrict_{H^{K_j}}$ is norm-continuous, and hence that the spectra of the operators $I_t\restrict_{H^{K_j}}$ vary continuously with $t$. Property (iii) implies that these spectra are all contained in a fixed finite subset of $\T$---the set of $n$th roots of $\prod_{j=1}^{n-1}\overline{\gamma}( w , w ^i)$---and so the spectra must in fact be constant in $t$. Since the union of the subspaces $H^{K_j}$ is dense in $H$, we have for each $t\in [0,1]$
\[
\begin{aligned}
 \big( w \in {W}'_{\chi_t}\ \ \text{ and }\ \ \i_{w,\chi_t}=z\big)  &\iff \spectrum(I_t)=\{z\}  \\
 &\iff \spectrum(I_t\restrict_{H^{K_j}})=\{z\}\text{ for all $j$}.
\end{aligned}
\]
Since $\spectrum(I_t\restrict_{H^{K_j}})$ is independent of $t$, and since we are assuming that $w\in W'_{\chi_0}$, we conclude that $w\in W'_{\chi_t}$ for all $t$, and that the scalar $\i_{w,\chi_t}$ is constant in $t$.
\end{proof}

\begin{remark}\label{rem:i=1}
For a fixed choice of $\chi_0\in X$ we can always choose the operators $T_{w}$ in Definition \ref{def:pI} so that $\i_{w, \chi_0}=1$ for all $w\in {W}'_{\chi_0}$. (This was noted in \cite[Remark (1) after Proposition 2.1]{Arthur-Elliptic}.) It would simplify the computations in Section \ref{sec:KCXEW}  if we could normalise the intertwining operators so that $\i_{w,\chi}=1$ for all $w\in {W}'_\chi$, for all $\chi$ at once, while keeping the cocycle $\gamma$ constant in $\chi$. We do not know whether this is possible in general, though it certainly is possible in some cases: e.g., when $G$ is a split Chevalley group and $M$ is a minimal Levi subgroup: see the proof of Theorem 1 in Section 3 of \cite{Keys}.

Lemma \ref{lem:homotopy} shows that for each $w\in W$ the function $\i_{w,\chi}$ is automatically constant on each connected component of the set ${X'}^{w}=\{\chi\in X\ |\ w\in W'_\chi\}$; but unlike in the case of real groups, for $p$-adic groups the sets ${X'}^w$ need not be connected, so we cannot deduce from this that the function $\i_{w,\chi}$ is always constant in $\chi$.
\end{remark}

\begin{theorem}[Harish-Chandra, Silberger]\label{thm:E-completeness}
The operators $U_{w}$ satisfy the completeness condition  of Definition \ref{def:completeness-condition}: that is, for each $\chi\in X$, every irreducible $\gamma$-representation of ${W}_\chi$ that lies over $\overline{\i_{\chi}}$ occurs in the contragredient of the $\overline{\gamma}$-representation $w\mapsto I_{ w,\chi}$. 
\end{theorem}

\begin{proof}
Fix $\chi\in X$. Since $W'_\chi$ acts on $\overline{H}$ by the projective character $\overline{\i_\chi}$, the $*$-homomorphism \begin{equation}\label{eq:I-rep}
\overline{I_\chi}:C^*_{\gamma}(W_\chi)\to \Bounded(\overline{H})
\end{equation}
induced by the $\gamma$-representation $w\mapsto \overline{I_{w,\chi}}$ factors through the projection onto the direct-summand $C^*_\gamma(W_\chi)P_{\i,\chi}$, where $P_{\i,\chi}$ is as in Lemma \ref{lem:Wprime-normal}. The identifications $(W_M)_{\sigma\otimes\chi}\cong {W}_\chi$ (Lemma \ref{lem:W-WM}) and $I_{w, \chi} = R(s,\sigma\otimes \chi)$ (\eqref{eq:I-R-equality}; here $w=\psi s$), together with theorems of Harish-Chandra 
\cite[Theorem 5.5.3.2]{Silberger-book} and Silberger \cite{Silberger-dimension}, ensure that the representation \eqref{eq:I-rep} is injective on $C^*_\gamma(W_\chi)P_{\i,\chi}$, and this injectivity is equivalent to the assertion that $\overline{I_\chi}$ contains all of the irreducible $\gamma$-representations of $W_{\chi}$ that lie over $\overline{\iota_{\chi}}$.
\end{proof}

We now obtain, as a special case of Theorem \ref{thm:general-Morita},

\begin{theorem}\label{thm:p-Morita}
The Hilbert module $E=C(X, \Ind_P^G H_\sigma)$, equipped with the twisted action of the finite group $W = (X\rtimes W_M)_\sigma$ by normalised intertwining operators from Definition \ref{def:pU},  implements a Morita equivalence between the direct-summand $C^*_r(G)_{(M,\sigma)}$ of $C^*_r(G)$, and the ideal 
\[
C(X,E,W) = \left\{ \sum_{w\in W} b_w w\in C(X)\rtimes_{\gamma} W\ \middle|\ \begin{aligned} &\forall \chi\in X,\ \forall w'\in W_\chi',\ \forall w\in W:\\ & b_{w'w}(\chi)  = \i_{w',\chi}\gamma(w',w) b_w(\chi)\end{aligned} \right\}.
\]
of the twisted crossed product $C(X)\rtimes_\gamma W$.
\hfill\qed
\end{theorem}

The Morita equivalence in Theorem \ref{thm:p-Morita} is somewhat complicated by the fact that (a) the algebra $C(X,E,W)$ is not the whole crossed product $C(X)\rtimes_\gamma W$, but an ideal therein; and (b) the identification of this ideal relies on detailed knowledge about the normalisation of the intertwining operators, namely the scalars $\i_{w',\chi}$. In the next section we will remove these complications by constructing a modified version $\widetilde{X}$ of the torus $X$.

\subsection{Blowing up the parameter space}

Let $G$, $M$, $X$, $\sigma$, and $W$ be as in Section \ref{sec:p-adic-background}. For the construction below it will be convenient to recall the following alternative description of the torus $X$. (See \cite[Chapitre V]{Renard}, for instance, for more details.)

Let $A_M$ denote the maximal $F$-split torus in the centre of $M$, and let $\germ{a}_M^* = \Hom_F(A_M,F^\times)\otimes_{\Z}\R$, where as above $\Hom_F(A_M,F^\times)$ denotes the finitely generated free abelian group of $F$-rational characters of $A_M$. The restriction map $\Hom_F(M,F^\times)\to \Hom_F(A_M,F^\times)$ is injective, and induces an isomorphism of $\R$-vector spaces $\Hom_F(M,F^\times)\otimes_{\Z} \R \xrightarrow{\cong}\germ{a}_M^*$. Each $\phi\otimes t\in \Hom_F(M,F^\times)\otimes_{\Z} \R$ determines an unramified unitary character $m\mapsto |\phi(m)|_F^{it}$ of $M$, and this construction gives rise to a surjective homomorphism of abelian groups $\Theta:\germ{a}_M^* \to X$. The group $W_M=N_G(M)/M$ acts by conjugation on $A_M$, and the covering map $\Theta:\germ{a}_M^*\to X$ is $W_M$-equivariant.

The vector space $\germ{a}^*_M$ is home to a root system $\Phi$: namely, the set of indivisible characters $\rho\in \Hom_{F}(A_M,F^\times)$ that appear in the adjoint action of $A_M$ on the Lie algebra of $M$. For each $\rho\in \Phi$ we let $s_\rho:\germ{a}^*_M\to\germ{a}^*_M$ denote the corresponding reflection, through the hyperplane $V_\rho\subset \germ{a}_M^*$ orthogonal to $\rho$ (with respect to an appropriate Weyl-group-invariant inner product).

\begin{definition}
$X' \coloneqq \{\chi\in X\ |\ W'_\chi\neq 1\}$. 
\end{definition}

Thus $X'$ is the set of unramified unitary characters for which at least one of the normalised intertwiners $I_{w,\chi}$ (or equivalently, $R(s,\sigma\otimes\chi)$) of $\Ind_P^G(\sigma\otimes\chi)$, for $w\neq 1$, is a scalar. Note that this set $X'$ depends on the discrete-series representation $\sigma$, but it does not depend on the chosen normalisation of the intertwining operators.

\begin{lemma}
$X'$ is a $W$-invariant, closed, nowhere-dense subset of $X$.
\end{lemma}

\begin{proof}\label{lem:X'-properties}
The fact that $W'_{w \chi}=wW'_\chi w^{-1}$ (Lemma \ref{lem:Wprime-normal})  ensures that $X'$ is $W$-invariant. 

For each $w\in W$ the set ${X'}^{w}=\{\chi\in X\ |\ w\in W'_\chi\}$ is closed in $X$. Indeed, if $\chi_i$ is a net in ${X'}^{w}$ converging to $\chi\in X$, then certainly $w$ fixes $\chi$ (because $w$ fixes each $\chi_i$); and moreover, the scalar intertwining operators $I_{w,\chi_i}$ converge to $I_{w,\chi}$ in the strong-operator topology (Theorem \ref{thm:pI-properties}), which forces $I_{w,\chi}$ to be a scalar too, and so $\chi\in {X'}^w$.  Now $X'$ is the union of the finitely many closed subsets ${X'}^w$ (for $w\neq 1$), so $X'$ is closed.

To see that $X'$ is nowhere dense in $X$, note that the only element of $X\rtimes W_M$ that fixes a nonempty open subset of $X$ is the identity. (This is true in general for the action of $L\rtimes \Aut(L)$ on $L$, where $L$ is a connected Lie group.) So for each $w\in W\setminus \{1\}$ the fixed-point set $X^w$ is nowhere dense in $X$. We have ${X'}^w\subseteq X^w$, and $X'=\bigcup_{w\in W\setminus \{1\}} {X'}^{w}$, so $X'$ is nowhere dense.
\end{proof}

The space $\widetilde{X}$ that we shall now construct can alternatively be obtained as an iterated blow-up, in the sense of \cite[Chapter 5]{Melrose}, of the torus $X$ along the space $X'$ (which is a union of finitely many pairwise-transverse, codimension-$1$ submanifolds of $X$.) Since we are dealing here with a very simple special case of the blow-up construction we shall not invoke the general definition, but instead proceed as follows.

\begin{definition}
We set $(\germ{a}_M^*)'\coloneqq \Theta^{-1}(X')$, and we then define  $\widetilde{\germ{a}_M^*}$ to be the disjoint union of the closures of the connected components of the complement $\germ{a}_M^*\setminus (\germ{a}_M^*)'$.
\end{definition}

Let $\Lambda$ denote the lattice $\ker\Theta\subseteq \germ{a}_M^*$. The subset $(\germ{a}_M^*)'\subseteq \germ{a}_M^*$ is obviously $\Lambda$-invariant, and the action of $\Lambda$ on $\germ{a}_M^*$ is continuous, and so this action induces an action of $\Lambda$ on $\widetilde{\germ{a}_M^*}$. 

\begin{definition}\label{def:widetildeX}
We let $\widetilde{X}$ denote the quotient space $\widetilde{\germ{a}_M^*}/\Lambda$.
\end{definition}

\begin{remark}
Note that $\widetilde{X}$ is not, in general, the disjoint union of the closures of the connected components of $X\setminus X'$. For example, let  $G=\SL_2(F)$, let $M$ be the diagonal subgroup, and let $\sigma$ be the trivial representation of $M$. Then the torus $X$ is a circle, and $X'$ is a single point. Since $X\setminus X'$ is connected, the disjoint union of the closures of the connected components of $X\setminus X'$ is just the circle $X$, whereas $\widetilde{X}$ is a closed interval.
\end{remark}

The action of $\Lambda$ on $\widetilde{\germ{a}_M^*}$ is proper, so $\widetilde{X}$ is a Hausdorff space. We will show below that $\widetilde{X}$ is compact: see Lemma \ref{lem:widetildeX-compact}. The space $\widetilde{X}$ is computed in several examples in Section \ref{sec:examples}.

Since $\widetilde{\germ{a}_M^*}$ is a disjoint union of closed subsets of $\germ{a}_M^*$, there is a canonical `gluing' (or `blowing-down') map $\Pi:\widetilde{\germ{a}_M^*}\to \germ{a}_M^*$, defined to be the identity on each connected component of $\widetilde{\germ{a}_M^*}$. This map is $\Lambda$-equivariant, and so it induces a surjection  $\Pi: \widetilde{X}\to X$. 

\begin{lemma}
There is a unique action $\widetilde{\alpha}$ of $W$ on $\widetilde{X}$ for which the map $\Pi:\widetilde{X}\to X$ is $W$-equivariant.
\end{lemma}

\begin{proof}
The uniqueness follows from the fact that the map $\Pi$ is one-to-one over the dense, $W$-invariant subset $X\setminus X'$ of $X$.

For the existence: given $w=\psi s\in W$ (with $\psi\in X$ and $s\in W_M$), choose a point $y\in \Theta^{-1}(\psi)$. The affine isometry $a:x\mapsto s(x)+y$ of $\germ{a}_M^*$ then satisfies $\Theta\circ a=w\circ \Theta$. This isometry preserves the subset $(\germ{a}_M^*)'$ (because $X'$ is $W$-invariant), and so it permutes the closures of the connected components of $\germ{a}_M^*\setminus (\germ{a}_M^*)'$. Passing back to the quotient by $\Lambda$, we obtain a homeomorphism $\widetilde{\alpha}(w):\widetilde{X}\to \widetilde{X}$ satisfying $\Pi\circ \widetilde{\alpha}(w)=\alpha(w)\circ\Pi$ for all $\widetilde\chi\in \widetilde{X}$. It is easily verified that this construction is compatible with composition in $W$, and so we get an action of $W$ on $\widetilde{X}$ making the map $\Pi$ equivariant.
\end{proof}

Here is the key property of the space $\widetilde{X}$:

\begin{lemma}\label{lem:Wprime-simply-transitive}
For each $\chi\in X$ the group $W'_\chi$ permutes the set $\Pi^{-1}(\chi)\subset \widetilde{X}$ simply transitively.
\end{lemma}

\begin{proof}
We can assume without loss of generality that $\chi=1$, by replacing the discrete-series representation $\sigma$ by $\sigma\otimes\chi$.  Consulting the definitions of $\widetilde{X}$ and $\Pi$ shows that the set $\Pi^{-1}(1)$ is in $W_1$-equivariant bijection with the set of connected components of $\germ{a}_M^*\setminus (\germ{a}_M^*)'$ whose closure contains $0$.

The group $W'_1 = \{s\in (W_M)_\sigma\ |\ I_{s,1}\text{ is a scalar}\}$ is known (by \cite{Silberger-dimension}) to be the Weyl group of the root system
\[
\Phi' \coloneqq \{\rho\in \Phi\ |\ s_\rho\in W'_1\}.
\]
(Strictly speaking $\Phi'$ is a root system in the subspace of $\germ{a}_M^*$ that it generates; but we will continue to work in the full space $\germ{a}_M^*$.) 

We observed in the proof of Lemma \ref{lem:X'-properties} that for each nontrivial $w\in W$ the set ${X'}^{w}=\{\phi\in X\ |\ w\in W'_\phi\}$ is nowhere dense in $X$. So we can find an open neighbourhood $U$ of $1$ in $X$ that intersects only those ${X'}^{w}$s that contain $1$: in other words, such that for all $\phi\in U$ we have $W'_\phi\subseteq W'_1$. We then choose an open ball $B$ centred at $0\in \germ{a}^*_M$, small enough so that the quotient map $\Theta:\germ{a}_M^*\to X$ is injective on $B$, and so that $\Theta(B)\subseteq U$.

We claim that 
\begin{equation}\label{eq:aM'-cap-B}
(\germ{a}_M^*)'\cap B = \bigcup_{\rho\in \Phi'}(V_\rho\cap B),
\end{equation}
where $V_\rho$ is the hyperplane in $\germ{a}_M^*$ orthogonal to the root $\rho$. To prove the inclusion `$\subseteq$', take $x\in (\germ{a}_M^*)'\cap B$. Since $x\in (\germ{a}_M^*)'$ the group $W_{\Theta(x)}'$ is nontrivial. Since $\Theta(x)\in U$ we have $W'_{\Theta(x)}\subseteq W'_1$, so $\Theta(x)$ is fixed by some nontrivial element $s\in W'_1$. Since $\Theta$ is $W_M$-equivariant and injective on $B$, it follows that $x$ is fixed by $s$. Since $W'_1$ is the Weyl group of $\Phi'$, $x$ must lie in one of the root hyperplanes $V_\rho$.

To prove the inclusion `$\supseteq$' in \eqref{eq:aM'-cap-B}, suppose that $x\in V_\rho\cap B$ for some $\rho\in \Phi'$. The continuous path $\Theta(tx)$ ($t\in [0,1]$) in $X$ consists of $s_\rho$-fixed points, and $s_\rho\in W'_{\Theta(0x)}=W'_1$, so Lemma \ref{lem:homotopy} ensures that $s_\rho\in W'_{\Theta(x)}$, and so $x\in (\germ{a}_M^*)'$.

Having established the equality \eqref{eq:aM'-cap-B}, we deduce from it that the set of connected components of $\big(\germ{a}_M^*\setminus (\germ{a}_M^*)'\big)\cap B$ is in  $W_1$-equivariant bijection with the set of Weyl chambers for the root system $\Phi'$. 

If $x_0$ and $x_1$ lie in distinct connected components of $\big(\germ{a}_M^*\setminus (\germ{a}_M^*)'\big)\cap B$, then the equality \eqref{eq:aM'-cap-B} shows that $x_0$ and $x_1$ must lie on opposite sides of at least one of the hyperplanes $V_\rho$, and so every continuous path in $\germ{a}_M^*$ connecting $x_0$ to $x_1$ must contain a point $x\in V_\rho$. Another application of Lemma \ref{lem:homotopy} shows that $s_\rho\in W'_{\Theta(x)}$, and so $x\in (\germ{a}_M^*)'$. Thus $x_0$ and $x_1$ also lie in distinct connected components of $\germ{a}_M^*\setminus (\germ{a}_M^*)'$.

We now have $W_1$-equivariant bijections
\[
\begin{aligned}
\Pi^{-1}(1) & \cong \left\{\text{connected components of $\germ{a}_M^*\setminus(\germ{a}_M^*)'$ whose closure contains $0$} \right\} \\
& \cong \left\{\text{connected components of $\big(\germ{a}_M^*\setminus(\germ{a}_M^*)\big)\cap B$}\right\} \\
& \cong \left\{\text{Weyl chambers for the root system $\Phi'$}\right\}.
\end{aligned}
\]
The Weyl group $W'_1$ permutes the Weyl chambers for $\Phi'$ simply transitively, so the same is true of the action on $\Pi^{-1}(1)$.
\end{proof}

\begin{lemma}\label{lem:W_chi-splitting}
For each $\chi\in X$ and $\widetilde{\chi}\in \Pi^{-1}(\chi)$ we have $W_\chi = W'_\chi\rtimes W_{\widetilde\chi}$.
\end{lemma}

\begin{proof}
Since $W'_\chi$ is a normal subgroup of $W_\chi$ (Lemma \ref{lem:Wprime-normal}), it is certainly normalised by the subgroup $W_{\widetilde{\chi}}$ of $W_\chi$. Consider the multiplication map $W'_\chi\times W_{\widetilde{\chi}}\to W_\chi$. This map is injective because $W'_\chi\cap W_{\widetilde{\chi}}=1$ (by the `simply' part of Lemma \ref{lem:Wprime-simply-transitive}). On the other hand, the `transitive' part of  Lemma \ref{lem:Wprime-simply-transitive} and the orbit-stabiliser theorem imply that $|W_\chi|=|W'_\chi|\cdot |W_{\widetilde{\chi}}|$, so the multiplication map is also surjective.
\end{proof}

\begin{remark}
Examining the proof of Lemma \ref{lem:Wprime-simply-transitive} shows that for each $\widetilde{\chi}\in \widetilde{X}$ the group $W_{\widetilde\chi}$ is the Knapp-Stein $R$-group $R_{\sigma\otimes\chi}$ associated to the character $\chi=\Pi(\widetilde\chi)$ and the system of positive roots for $W'_\chi$ determined by the point $\widetilde{\chi}$ (cf.~\cite[\S13]{Knapp-Stein-II}.)
\end{remark}

\begin{lemma}\label{lem:widetildeX-compact}
The space $\widetilde{X}$ is compact.
\end{lemma}

\begin{proof}
The proof of Lemma \ref{lem:Wprime-simply-transitive} showed that each point in $\germ{a}_M^*$ has an open neighbourhood intersecting only finitely many connected components of $\germ{a}_M^*\setminus (\germ{a}_M^*)'$, and it follows from this that the gluing map $\Pi:\widetilde{\germ{a}_M^*}\to \germ{a}_M^*$ is proper. Now $X$ is the image, under the quotient map $\germ{a}_M^*\to X$, of a compact subset $S\subseteq \germ{a}_M^*$, and so $\widetilde{X}$ is the image, under the quotient map $\widetilde{\germ{a}_M^*}\to \widetilde{X}$, of the compact set $\Pi^{-1}(S)$.
\end{proof}

\begin{lemma}\label{lem:blow-up-quotient}
The gluing map $\Pi:\widetilde{X}\to X$ induces a homeomorphism $\widetilde{X}/W\cong X/W$.
\end{lemma}

\begin{proof}
Since $\Pi$ is a continuous, $W$-equivariant surjection, it induces a continuous surjection on orbit spaces. This map on orbit spaces is also injective, since Lemma \ref{lem:Wprime-simply-transitive} implies  that each fibre of $\Pi$ is contained in a single $W$-orbit.
\end{proof}

Turning to $C^*$-algebras, we let $W$ act on the $C^*$-algebra $C(\widetilde{X},\Compact(H))$ by
\[
(w\widetilde{k})(\widetilde\chi)\coloneqq I_{w,w^{-1}\chi} \widetilde{k}(w^{-1}\widetilde\chi) I_{w,w^{-1}\chi}^*
\]
where $w\in W$, $\widetilde{k}\in C(\widetilde{X},\Compact(H))$, $\widetilde{\chi}\in \widetilde{X}$, $\chi=\Pi(\widetilde\chi)\in X$, and where $I_{w,w^{-1}\chi}$ is the unitary operator on $H$ defined in Definition \ref{def:pI}. The pullback homomorphism $\Pi^*:C(X,\Compact(H))\to C(\widetilde{X},\Compact(H))$ (defined by $(\Pi^* k)(\widetilde{\chi})=k(\Pi\widetilde{\chi})$) is $W$-equivariant, and thus restricts to a morphism of $W$-fixed-point algebras.

\begin{lemma}\label{lem:Pi-Cstar-iso}
The pullback map $\Pi^*:C(X,\Compact(H))^W\to C(\widetilde{X},\Compact(H))^W$ is an isomorphism of $C^*$-algebras.
\end{lemma}

\begin{proof}
Since $\Pi$ is surjective, $\Pi^*$ is injective. To see that $\Pi^*$ is surjective on the $W$-invariants, fix a function $\widetilde k\in C(\widetilde X, \Compact(H))^W$, and let $\widetilde{\chi}_1,\widetilde{\chi}_2\in \widetilde{X}$ lie in the same fibre $\Pi^{-1}(\chi)$ of $\Pi$. Lemma \ref{lem:Wprime-simply-transitive} implies that $\widetilde{\chi}_2 = w\widetilde{\chi}_1$ for some $w\in W'_\chi$, and we then have
\[
\widetilde{k}(\widetilde{\chi}_2) =  I_{w,\chi} \widetilde{k}(w^{-1}\widetilde{\chi_2}) I_{w, \chi}^* = \widetilde{k}(\widetilde{\chi}_1),
\]
where the first equality holds because $\Pi(w^{-1}\widetilde{\chi_2})=\Pi(\widetilde{\chi}_1)=\chi$, and the second equality holds because $I_{w,\chi}$ is a scalar. Thus the function $\widetilde{k}$ is constant on the fibres of the map $\Pi$, and therefore  lies in the image of $\Pi^*$.
\end{proof}

\begin{theorem}\label{thm:Plancherel-2}
$C^*_r(G)_{(M,\sigma)} \cong C(\widetilde{X},\Compact(H))^W \sim C(\widetilde{X})\rtimes_\gamma W$.
\end{theorem}

\begin{proof}
The isomorphism holds by Lemma \ref{lem:Pi-Cstar-iso} and Theorem \ref{thm:Plancherel}. The Morita equivalence follows from Corollary \ref{cor:E-equivalence-easy}, applied to the Hilbert module $\widetilde{E}=C(\widetilde{X},H)$ and the family of operators
\[
(U_w\widetilde{\xi})(\widetilde\chi) \coloneqq I_{w, \Pi(w^{-1}\widetilde\chi)} \widetilde{\xi}(w^{-1}\widetilde\chi).
\]
Indeed, Lemma \ref{lem:Wprime-simply-transitive} implies that $W'_{\widetilde\chi}=1$ for every $\widetilde\chi\in\widetilde{X}$, and so $C(\widetilde{X},\widetilde{E},W)=C(\widetilde{X})\rtimes_\gamma W$. The completeness condition for the module $\widetilde{E}$ follows from the  completeness condition for the $C(X)$-module $E=C(X,H)$ (Theorem \ref{thm:E-completeness}), together with the decomposition $W_\chi = W'_\chi\rtimes W_{\widetilde{\chi}}$ (where $\chi=\Pi(\widetilde\chi)$) of Lemma \ref{lem:W_chi-splitting}.
\end{proof}

Since a Morita equivalence between two $C^*$-algebras induces a homeomorphism between their spectra, and since the spectrum of the twisted crossed product $C(\widetilde{X})\rtimes_\gamma W$ can be identified (cf.~Lemma \ref{lem:Mackey-machine}) with the twisted extended quotient $(\widetilde{X}/\!\!/ W)_\gamma$ (cf.~\cite[Section 2.1]{ABPS}), we obtain from Theorem \ref{thm:Plancherel-2} the following description of the tempered dual of $G$ as a topological space:

\begin{corollary}\label{cor:extended-quotient}
The component $\widehat{C^*_r(G)_{(M,\sigma)}}$ of the tempered dual of $G$ is homeomorphic to the twisted extended quotient $\left(\widetilde{X}/\!\!/W\right)_\gamma$, with the Jacobson topology.\hfill
\end{corollary}

\begin{remark}
A convenient description of the Jacobson topology on an extended quotient with trivial twisting can be found in \cite[Section 4]{Echterhoff-Emerson}. The extension to the twisted case is straightforward.
\end{remark}

\begin{remark}\label{rem:Wassermann-proof2}
The blow-up construction can also be applied in the setting of real reductive groups, giving another route to Wassermann's Morita equivalence \eqref{eq:Wassermann}. Let $G$ be a real reductive group, and let $P$, $M$, $A$, $\sigma$, $\germ{a}^*$, $W$, $W'$, $R$, etc., be as in Example \ref{example:Wassermann}. The set $(\germ{a}^*)'=\{\chi\in \germ{a}^*\ |\ W'_\chi\neq 1\}$ is a union of finitely many hyperplanes in the vector space $\germ{a}^*$, and we let $\widetilde{\germ{a}^*}$ be the disjoint union of the closed convex cones cut out by these hyperplanes. (That is, $\widetilde{\germ{a}^*}$ is the disjoint union of the closures of the connected components of $\germ{a}^*\setminus (\germ{a}^*)'$.) The Weyl group $W'$ permutes these cones simply transitively, and so we have
\begin{multline*}
C_0(\germ{a}^*,\Compact(H))^W \cong C_0(\widetilde{\germ{a}^*},\Compact(H))^W \sim C_0(\widetilde{\germ{a}^*})\rtimes W  \cong \left( C_0(\widetilde{\germ{a}^*})\rtimes W'\right)\rtimes R \\
\sim C_0(\widetilde{\germ{a}^*}/W')\rtimes R \cong C_0(\germ{a}^*/W')\rtimes R.
\end{multline*}

\end{remark}

\section{\texorpdfstring{$K$-theory of $C^*_r(G)$}{K-theory of C*r(G)}}\label{sec:K-theory}

Let $G$ be a reductive $p$-adic group, as in Section \ref{sec:p-adic}. Theorem \ref{thm:Plancherel} gives an isomorphism of $C^*$-algebras $C^*_r(G)\cong \bigoplus_{[M,\sigma]} C^*_r(G)_{(M,\sigma)}$, and hence an isomorphism in  $K$-theory $K_*(C^*_r(G))\cong \bigoplus_{[M,\sigma]} K_*(C^*_r(G)_{(M,\sigma)})$. The Morita equivalences obtained in Theorems \ref{thm:p-Morita} and \ref{thm:Plancherel-2} yield isomorphisms in $K$-theory
\[
K_*\left(C^*_r(G)_{(M,\sigma)}\right)\cong K_*\left( C(X,E,W)\right) \cong K_*(C(\widetilde{X})\rtimes_\gamma W).
\]
In this section we describe a method for computing these $K$-theory groups. 

One can compute $K_*(C(\widetilde{X})\rtimes_\gamma W)$ by deploying general machinery from equivariant cohomology \cite{Bredon-Equivariant,Luck-Chern}. A priori this general machinery is not applicable to $K_*(C(X,E,W))$---which is not, in any obvious way, the value at $X$ of an equivariant cohomology theory---but the general techniques can be adapted to this case, thanks to the isomorphism $K_*(C(X,E,W))\cong K_*(C(\widetilde{X})\rtimes_\gamma W)$. 

Throughout this section we fix a Levi subgroup $M$ of $G$, and a discrete-series representation $\sigma$ of $M$. As in Section \ref{sec:p-adic} we let $X$ denote the compact torus of unramified unitary characters of $M$; we let $W=(X\rtimes W_M)_\sigma$ be the finite group from Definition \ref{def:W}; we let $\gamma:W\times W\to \T$ be the cocycle defined in Theorem \ref{thm:pI-properties}; and we let $\widetilde{X}$ denote the blow-up of $X$ defined in Definition \ref{def:widetildeX}.

\subsection{\texorpdfstring{Equivariant CW-structures on $X$ and $\widetilde{X}$}{Equivariant CW-structures on X and X˜}}

Since the torus $X$ is a smooth compact manifold, on which the finite group $W$ acts smoothly, $X$ carries a finite $W$-invariant CW-structure (indeed, $X$ has a $W$-invariant simplicial structure, by \cite{Illman}.) Let us briefly recall what this entails---referring to \cite[I.1]{Bredon-Equivariant} for further details---and establish some notation.

For each $j\geq 0$ we have a finite set of indices $Z_j$, whose elements are called \emph{$j$-cells}. We write $Z\coloneqq \bigcup_{j\geq 0}Z_j$. To each $z\in Z_j$ there corresponds a subset $e_z\subset X$, called the \emph{open $j$-cell} corresponding to the index $z$. Distinct open cells are disjoint from one another.

We denote by $X_j\subseteq X$ the union of the open $k$-cells for all $k\leq j$. We have $X=X_n$, where $n$ is the dimension of $X$ as a smooth manifold. For each $z\in Z_j$ we have, as part of the given structure, a continuous map $\theta_z : D^j \to \overline{e_z}$ that restricts to a homeomorphism $B^j\to e_z$, and that maps the boundary $S^{j-1}$ of $D^j$ into $X_{j-1}$. (Here, for $j\geq 1$, $D^j=\{x\in \R^j\ |\ \|x\|\leq 1\}$ is the closed $j$-dimensional ball, $S^{j-1}$ is its boundary sphere, and $B^j$ is its interior; while for $j=0$ both $D^0$ and $B^0$ denote the set $\{0\}$, and $S^{-1}=\emptyset$.) 

Though it is not strictly necessary, it will simplify matters to assume (as we certainly may, and henceforth do) that the closure of each open cell in $X$ is  a union of open cells. Thus for each $z\in Z_j$ there is a set of cells $\partial z\subset\bigcup_{k<j} Z_k$ such that $\overline{e_z}\setminus e_z = \bigsqcup_{{y}\in \partial z} e_{{y}}$. 

Our assumption that this CW-structure is $W$-equivariant means that for each open $j$-cell $e_z\subset X$ and each $w\in {W}$ the set $we_z$ is another open $j$-cell; thus ${W}$ acts on $Z_j$ for each $j$. Moreover, the attaching maps $\theta_z$ are assumed to satisfy $\theta_{wz}=w\theta_z$ for all $w\in {W}$. In particular, if $wz=z$ then $w$ fixes $e_z$ pointwise. 

Let $z\in Z_{j+1}$ and ${y}\in Z_{j}$ be a $(j+1)$-cell and a $j$-cell, respectively. The integer $[{y}:z]$ is defined to be the degree of the following map on $S^j$:
\begin{equation}\label{eq:degree}
[{y}:z]\coloneqq \deg\left(S^j\xrightarrow{\theta_z} X_{j} \xrightarrow{\text{quotient}} X_{j}/(X_{j}\setminus e_{{y}}) \xrightarrow{\theta_{{y}}^{-1}} D^j/S^{j-1} \xrightarrow{\cong} S^j\right).
\end{equation}
(Here $\theta_{{y}}^{-1}$ denotes the inverse of the homeomorphism $D^j/S^{j-1}\to \overline{e_{{y}}}/(\overline{e_{{y}}}\setminus e_{{y}})$ induced by the attaching map $\theta_{{y}}:D^j\to \overline{e_{{y}}}$.)

We now return to the specifics of the case under consideration, where $X$ is the torus of unramified characters of a Levi subgroup of $G$, etc. The homotopy argument of Knapp-Stein (Lemma \ref{lem:homotopy}) shows that each $W$-CW-structure on $X$ is compatible with the system of subgroups $W'_\chi$ and characters $\i_{w,\chi}$, in the following sense:

\begin{lemma}\label{lem:p-CW-properties}
Let $\chi\in X$ lie in the open cell $e_z$. For all $\phi\in \overline{e_z}$ we have:
\begin{enumerate}[\rm(1)]
\item ${W}_\chi \subseteq {W}_{\phi}$,
\item ${W}'_\chi = {W}_\chi \cap {W}'_\phi$, and
\item $\i_{w,\chi}=\i_{w,\phi}$ for all $w\in {W}'_\chi$. 
\end{enumerate}
\end{lemma}

\begin{proof}
For part (1) just note that  if $w\chi=\chi$ then $we_z=e_z$, so $w$ fixes $e_z$ pointwise, and hence also fixes $\overline{e_z}$ pointwise.

For parts (2) and (3), note that since $\chi\in e_z$ and $\phi\in \overline{e_z}$ there is a continuous path $\chi_t$ ($t\in [0,1]$) in $\overline{e_z}$ with $\chi_0=\chi$ and $\chi_1=\phi$. If $w\in W_\chi$ then $w$ fixes each $\chi_t$, and so Lemma \ref{lem:homotopy} implies that $w\in W'_\chi$ if and only if $w\in W'_\phi$; and that if we do have $w\in W'_\chi$, then $\i_{w,\chi}=\i_{w,\phi}$.
\end{proof}

The equalities in Lemma \ref{lem:p-CW-properties} apply, in particular, for all  $\phi\in e_z$, and so the notation ${W}_z\coloneqq {W}_\chi$, ${W}'_z\coloneqq {W}'_\chi$, and $\i_{w,z}\coloneqq \i_{w,\chi}$ is well-defined. 

We now turn to $\widetilde{X}$ (Definition \ref{def:widetildeX}.) Our chosen $W$-CW-structure on $X$ lifts canonically to a $\Lambda$-invariant CW-structure on the covering space $\germ{a}_M^*$, and each connected component of $\widetilde{\germ{a}_M^*}$ is a closed subcomplex of $\germ{a}_M^*$ (because Lemma \ref{lem:p-CW-properties} implies that $X'$ is a subcomplex of $X$.) Taking the quotient by $\Lambda$ yields a finite $W$-CW-structure on $\widetilde{X}$ for which the gluing map $\Pi:\widetilde{X}\to X$ sends open $j$-cells to open $j$-cells, for each $j$. 

For each $j\geq 0$ we denote the set of $j$-cells in $\widetilde{X}$ by $\widetilde{Z}_j$. The gluing map $\Pi$ induces a $W$-equivariant map $\Pi_j:\widetilde{Z}_j\to Z_j$. The same argument as in Lemma \ref{lem:blow-up-quotient} shows:

\begin{lemma}\label{lem:Pi-bijection-on-orbits}
The map $\Pi_j:\widetilde{Z}_j\to Z_j$ induces a bijection  $W\backslash \widetilde{Z}_j \xrightarrow{\cong} W\backslash Z_j$.\hfill\qed
\end{lemma}

\subsection{Coefficient systems and equivariant cohomology}

We next recall some definitions from \cite{Bredon-Equivariant} related to equivariant cohomology. 

\begin{definition}
Let  $\mathcal{X}$ be a finite $W$-CW-complex.
\begin{enumerate}[\rm(1)]
\item For each cell $z$ of $\mathcal{X}$ we define $\mathcal{W}^\gamma_z\coloneqq \Rep_\gamma(W_z)$, the free abelian group of isomorphism classes of virtual finite-dimensional $\gamma$-representations of $W_z$. (That is, the free abelian group generated by the set $\widehat{W_z}^{\gamma}$.)
\item For ${y}\in \partial z$ we let
\[
\mathcal{W}^\gamma_{z,{y}} : \Rep_{\gamma}(W_{{y}}) \to \Rep_{\gamma}(W_z)
\]
be the map given by restriction of representations from $W_{{y}}$ to its subgroup $W_z$.
\item For $w\in W$ we let
\[
\mathcal{W}^\gamma_{w,z} :\Rep_{\gamma}(W_z)\to \Rep_{\gamma}(W_{wz})
\]
be the map $\pi\mapsto {}^w\pi$, where
\[
{}^w\pi(v)\coloneqq \pi(w^{-1}vw)  {\gamma}(w^{-1}v,w) {\gamma}(w^{-1},v).
\]
The fact that $\gamma(w,w^{-1})=1$ (cf.~\eqref{eq:gamma-inverse}) ensures that ${}^w\pi$ is a $\gamma$-representation of $W_{wz}$.
\end{enumerate}
\end{definition}

The groups $\mathcal{W}^\gamma_z$ and maps  $\mathcal{W}^\gamma_{z,{y}}$ and $\mathcal{W}^\gamma_{w,z}$ constitute a \emph{local coefficient system} in the sense of \cite[Section I.5]{Bredon-Equivariant}. That is to say, for all $z\in Z$:
\begin{itemize}
\item the maps $\mathcal{W}^\gamma_{z,z}$ and $\mathcal{W}^\gamma_{1_W,z}$ are the identity on $\mathcal{W}^\gamma_z$;
\item $\mathcal{W}^\gamma_{z,{y}'} = \mathcal{W}^\gamma_{z,{y}}\circ \mathcal{W}^\gamma_{{y},z}$ for all ${y}\in \partial z$ and ${y}'\in \partial {y}$; 
\item $\mathcal{W}^\gamma_{wz,w{y}}\circ \mathcal{W}^\gamma_{w,{y}} = \mathcal{W}^\gamma_{w,z}\circ \mathcal{W}^\gamma_{z,{y}}$ for all ${y}\in \partial z$ and $w\in W$; and
\item $\mathcal{W}^\gamma_{w_2,w_1z}\circ \mathcal{W}^\gamma_{w_1,z}=\mathcal{W}^\gamma_{w_2w_1,z}$ for all $w_1,w_2\in W$.
\end{itemize}

The coefficient system $\mathcal{W}^\gamma$ is in fact an example of a \emph{generic coefficient system} in the sense of \cite[Section I.4]{Bredon-Equivariant} (in other words, $\mathcal{W}^\gamma$ comes from the functor $W/H\mapsto \Rep_\gamma(H)$ on the orbit category of $W$.) The same is not true for the following coefficient system; here we specialise to $\mathcal{X}=X$, the torus of unitary unramified characters of $M$.

\begin{definition}
For each cell $z$ of $X$ we let $\Rep_{\gamma,\overline{\i_z}}(W_z)$ be the direct-summand of $\Rep_\gamma(W_z)$ generated by those irreducible $\gamma$-representations of $W_z$ that lie over (i.e., contain: cf.~Definition \ref{def:lie-over}) the character $\overline{\i_z}$ of $W'_z$.
\end{definition}

\begin{lemma}\label{lem:R-def}
The subspaces $\mathcal{R}_z\coloneqq \Rep_{\gamma,\overline{\i_z}}(W_z)$ of $\mathcal{W}^\gamma_z=\Rep_\gamma(W_z)$ are preserved by the maps $\mathcal{W}^\gamma_{z,{y}}$ and $\mathcal{W}^\gamma_{w,z}$. Denoting the restrictions of these maps to $\mathcal{R}_z$ by $\mathcal{R}_{z,{y}}$ and $\mathcal{R}_{w,z}$, respectively, we thus obtain a  local coefficient system $\mathcal{R}$ on $X$.
\end{lemma}

\begin{proof}
Lemma \ref{lem:p-CW-properties} ensures that if ${y}\in \partial z$ then we have $W'_{z}\subseteq W'_{{y}}$, and $\i_{w,z}=\i_{w,{y}}$ for all $w\in W'_z$. Thus if $\pi$ is an irreducible representation of $W_{{y}}$ lying over $\overline{\i_{{y}}}$, then the restriction of $\pi$ to $W_z$ lies over $\overline{\i_z}$. This shows that $\mathcal{W}^\gamma_{z,{y}}$ maps $\mathcal{R}_{{y}}$ into $\mathcal{R}_z$. On the other hand, the formula \eqref{eq:i-conjugation} ensures that $\mathcal{W}^\gamma_{w,z}$ maps $\mathcal{R}_z$ into $\mathcal{R}_{wz}$.
\end{proof}

\begin{remark}\label{rem:R-groups}
If $\i_{w,z}=1$ for all cells $z$ and all $w\in W'_z$, and $\gamma(w_1,w_2)=1$ for all $w_1,w_2\in W$, then we have  obvious identifications $\mathcal{R}_z\cong \Rep(W_z/W'_z)$. In this picture, the structure map $\mathcal{R}_{z,{y}}$ is given by restriction of representations along the injective homomorphism $W_z/W'_z \into W_{{y}}/W'_{{y}}$ induced by the inclusion $W_z\subset W_{{y}}$; and the map $\mathcal{R}_{w,z}$ is given by pullback along the isomorphism $\Ad_{w^{-1}}:(W_{wz}/W'_{wz})\xrightarrow{\cong} W_z/W'_z$. This situation arises, for instance, in the minimal principal-series components of split Chevalley groups; this follows from work of Keys \cite{Keys}. 
\end{remark}

We next recall from \cite[Section I.6]{Bredon-Equivariant} the definition of the equivariant cohomology groups associated to our coefficient systems $\mathcal{W}^\gamma$ and $\mathcal{R}$:

\begin{definition}\label{def:Bredon-cohomology}
Let $\mathcal{X}$ be a $W$-CW-complex. For each $j\geq 0$ we consider the free abelian group
\[
C^j(\mathcal{X},\mathcal{W}^\gamma)\coloneqq \bigoplus_{z\in Z_j} \mathcal{W}^\gamma_z.
\]
The group ${W}$ acts on $C^j(\mathcal{X},\mathcal{W}^\gamma)$ through the isomorphisms $\mathcal{W}^\gamma_{w,z} : \mathcal{W}^\gamma_z\to \mathcal{W}^\gamma_{wz}$, and we let $C^j(X,\mathcal{W}^\gamma)^{{W}}$ denote the subgroup of invariants for this action. We define $\partial^j:C^j(\mathcal{X},\mathcal{W}^\gamma)\to C^{j+1}(X,\mathcal{W}^\gamma)$ to be the sum of the maps 
\[
\partial^j_{z,{y}}\coloneqq [{y}:z] \mathcal{W}^\gamma_{z,{y}} : \mathcal{W}^\gamma_{{y}} \to \mathcal{W}^\gamma_z,
\]
where $z\in Z_{j+1}$, ${y}\in \partial z\cap Z_j$, and $[{y}:z]$ is the integer defined in \eqref{eq:degree}. The map $\partial^j$ is ${W}$-equivariant, and so it restricts to a map $\partial^j:C^j(\mathcal{X},\mathcal{W}^\gamma)^{{W}}\to C^{j+1}(\mathcal{X},\mathcal{W}^\gamma)^{{W}}$. These maps give rise to a cochain complex
\[
C^0(\mathcal{X},\mathcal{W}^\gamma)^{{W}} \xrightarrow{\partial^0} C^1(\mathcal{X},\mathcal{W}^\gamma)^{{W}} \xrightarrow{\partial^1} \cdots \xrightarrow{\partial^d} C^n(\mathcal{X},\mathcal{W}^\gamma)^{{W}}
\]
whose cohomology we denote by $H^*_{{W}}(\mathcal{X},\mathcal{W}^\gamma)$. (This is the \emph{$\gamma$-twisted Bredon cohomology} of $\mathcal{X}$ studied in \cite{Dwyer}.)

Replacing $\mathcal{X}$ everywhere by the torus $X$, and replacing $\mathcal{W}^\gamma$ everywhere by $\mathcal{R}$, we get the definition of a cochain complex $C^*(X,\mathcal{R})^W$ and its cohomology $H^*_W(X,\mathcal{R})$.
\end{definition}

The coefficient system $\mathcal{W}^\gamma$  will emerge from our computation of $K_*(C^*_r(G)_{(M,\sigma)})$ in the following form. The functors 
\begin{equation*}\label{eq:H-definition}
\mathcal{H}^q:(S,T)\mapsto K_{-q}(C_0(S\setminus T)\rtimes_\gamma W),
\end{equation*}
from the category of nested pairs $T\subseteq S$ of $W$-CW complexes to the category of abelian groups, form an equivariant cohomology theory in the sense of \cite{Luck-Chern}. 
For each $q$ we get from $\mathcal{H}^q$ a coefficient system $\mathcal{H}^q:z\mapsto K_{-q}(C(Wz)\rtimes_\gamma W)$ on $\mathcal{X}$ (cf \cite[Sections I.4--5]{Bredon-Equivariant}). Since $C(Wz)\rtimes_\gamma W$ is finite-dimensional, this coefficient system is identically $0$ when $q$ is odd. As for the even $q$s, the imprimitivity theorem (see Example \ref{ex:imprimitivity}) implies that for each cell $z$ of $\mathcal{X}$ the map
\begin{equation}\label{eq:Phi_z}
\Psi_{z}:K_0(C^*_\gamma(W_z)) \xrightarrow{H\mapsto (C(Wz)\rtimes_\gamma W)\delta_z\otimes_{C^*_\gamma(W_z)} H} K_0(C(Wz)\rtimes_\gamma W)
\end{equation}
is an isomorphism, with inverse
\[
\Psi_z^{-1}:K_0(C(Wz)\rtimes_\gamma W) \xrightarrow{H\mapsto \delta_z H} K_0(C^*_\gamma(W_z)),
\]
and so we get an isomorphism
\begin{equation}\label{eq:coefficient-isomorphism}
\mathcal{W}^\gamma_z = \Rep_{\gamma}(W_z) \xrightarrow{\cong} K_0(C^*_\gamma(W_z)) \xrightarrow{\Psi_z} K_0(C(Wz)\rtimes_\gamma W) \xrightarrow{\text{Bott}}\mathcal{H}^q_z.
\end{equation}
Straightforward computations verify that these isomorphisms are compatible with the restriction maps and with the $W$-action.

\subsection{\texorpdfstring{A spectral sequence for $K_*(C(\widetilde{X})\rtimes_\gamma W)$}{A spectral sequence for K*(C(X˜)⋊gamma W)}}

We continue to let $\mathcal{X}$ denote a finite $W$-CW-complex of dimension $n$. Standard machinery (explained in \cite[Chapter IV]{Bredon-Equivariant}, for instance) applied to the equivariant cohomology theory $\mathcal{H}$ produces an Atiyah-Hirzebruch-style spectral sequence converging to $K_*(C(\mathcal X)\rtimes_\gamma W)$. Since we shall  apply these arguments to the non-standard case of $K_*(C(X,E,W))$ below, we briefly recall the outline of the construction from the point of view of $C^*$-algebras, following \cite{Schochet-I}.

Consider the filtration
\[
\emptyset=\mathcal{X}_{-1}\subseteq \mathcal{X}_0 \subseteq \cdots \subseteq \mathcal{X}_n = \mathcal{X}
\]
by its $j$-skeleta. For each $p\in \{0,\ldots,n+1\}$. consider the $W$-invariant open set $\mathcal{Y}_p\coloneqq \mathcal{X}\setminus \mathcal{X}_{n-p}\subseteq \mathcal{X}$. Setting $\mathcal{C}_p\coloneqq C_0(\mathcal{Y}_p)\rtimes_\gamma W$, we have an increasing sequence of ideals
\begin{equation}\label{eq:mathcalC-filtration}
0 = \mathcal{C}_0 \subset \mathcal{C}_1 \subset \cdots \subset \mathcal{C}_{n+1}=\mathcal{C}
\end{equation}
in the $C^*$-algebra $\mathcal{C}$.

For each $p$ we have an extension of $C^*$-algebras $\mathcal{C}_{p-1}\into \mathcal{C}_p\onto \mathcal{C}_p/\mathcal{C}_{p-1}$, giving maps in $K$-theory $K_*(\mathcal{C}_{p-1})\to K_*(\mathcal{C}_p)$, $K_*(\mathcal{C}_p)\to K_*(\mathcal{C}_p/\mathcal{C}_{p-1})$, and $K_*(\mathcal{C}_p/\mathcal{C}_{p-1})\to K_{*-1}(\mathcal{C}_{p-1})$. As explained in \cite[Section 6]{Schochet-I} these maps form an `exact couple', and thus give rise to a spectral sequence $(E^r(\mathcal{C}),d^r)$, with $d^r: E^r_{p,q}(\mathcal{C}) \to  E^r_{p-r,q+r-1}(\mathcal{C})$, that converges to $K_*(\mathcal{C})$ and has 
\[
E^1_{p,q}(\mathcal{C}) = K_{p+q}( \mathcal{C}_p/\mathcal{C}_{p-1}).
\]

The exactness of the crossed-product construction implies that for each $p$ we have a canonical isomorphism
\[
\mathcal{C}_p/\mathcal{C}_{p-1} \cong C_0(\mathcal{Y}_p\setminus \mathcal{Y}_{p-1})\rtimes_\gamma W=C_0(\mathcal{X}_{n-p+1}\setminus \mathcal{X}_{n-p})\rtimes_\gamma W.
\]
To save space let us define
\[
j\coloneqq n-p+1.
\]
We then have $E^1_{p,q}(\mathcal{C}) \cong \mathcal{H}^{-p-q}(X_j,X_{j-1})$. These isomorphisms (and a change of indices) identify the spectral sequence $E(\mathcal{C})$ with the equivariant Atiyah-Hirzebruch spectral sequence associated in \cite[Section IV.4]{Bredon-Equivariant} to the equivariant cohomology theory $\mathcal{H}$. (The apparent difference in the $d^1$ differentials  between Schochet's and Bredon's spectral sequences is reconciled by considering the commuting diagram
\[
\xymatrix@C=10pt{
0 \ar[r] & C_0(\mathcal X\setminus \mathcal X_j)\rtimes_\gamma W \ar[r]\ar[d] & C_0(\mathcal X\setminus \mathcal X_{j-1})\rtimes_\gamma W \ar[r]\ar[d] & C_0(\mathcal X_j\setminus \mathcal X_{j-1})\rtimes_\gamma W \ar[r] \ar[d]& 0 \\
0  \ar[r] & C_0(\mathcal X_{j+1}\setminus \mathcal X_j)\rtimes_\gamma W \ar[r] & C(\mathcal X_{j+1})\rtimes_\gamma W \ar[r]& C(\mathcal X_j)\rtimes_\gamma W \ar[r]& 0
}
\]
in which each arrow is given by the obvious combination of restriction and extension-by-zero of functions. The corresponding commuting diagram of $K$-theory long exact sequences establishes the equality between the two spectral sequences.)

The computation in \cite[pp. IV.6--10]{Bredon-Equivariant} shows that chain complex $E^1_{p,q}(\mathcal{C}) \xrightarrow{d^1} E^1_{p-1, q}(\mathcal{C})$ is isomorphic to complex
\[
C^j(\mathcal X,{\mathcal{H}}^{-q-n-1})^W \xrightarrow{\partial} C^{j+1}(\mathcal X,{\mathcal{H}}^{-q-n-1})^W
\]
computing the Bredon cohomology of $\mathcal X$ with coefficients in ${\mathcal{H}}^{-q-n-1}$. Explicitly, the isomorphism is given by noting that since $W$ acts trivially on $B^j$, the canonical isomorphism
\[
C_0(B^{j}\times Z_j)\cong C_0(B^{j})\otimes C(Z_j)
\]
induces an isomorphism 
\[
C_0(B^{j}\times Z_j)\rtimes_\gamma W \cong 
C_0(B^{j})\otimes \left(C(Z_j)\rtimes_\gamma W\right),
\] 
giving by suspension
\begin{equation}\label{eq:E1widetildeC-iso}
\begin{aligned}
& K_{p+q}(C_0(B^j\times Z_j)\rtimes_\gamma W)  \cong
K_{q+n+1}(C(Z_j)\rtimes_\gamma W)\\
&\cong \bigoplus_{Wz\in W\backslash Z_j} K_{q+n+1}(C(Wz)\rtimes_\gamma W) 
\cong C^j(X,{\mathcal{H}}^{-q-n-1})^W.
\end{aligned}
\end{equation}

We noted above that if $q+n$ is even then the coefficient system ${\mathcal{H}}^{-q-n-1}$ is identically zero, while if $q+n$ is odd we have $\mathcal{H}^{-q-n-1}\cong \mathcal{W}^\gamma$. We arrive at:

\begin{theorem}[{\cite{Bredon-Equivariant,Luck-Chern, Dwyer}}]\label{thm:Bredon-spectral-sequence}
Let $\mathcal{X}$ be a finite $W$-CW-complex of dimension $n$, and let $\mathcal{C}=C(\mathcal{X})\rtimes_\gamma W$. There is a spectral sequence $E(\mathcal{C})$ converging to $K_*(\mathcal C)$, and having
\[
E^2_{p,q}(\mathcal{C}) = \begin{cases} H^{n-p+1}_W(\mathcal{X},\mathcal{W}^\gamma) & \text{if $q+n$ is odd,}\\ 0 & \text{if $q+n$ is even,}
\end{cases}
\]
and the differential $d^2$ is zero. This spectral sequence collapses rationally at the $E^2$ page, giving an isomorphism 
\[
K_*(\mathcal{C})\otimes_{\Z}\Q \cong \bigoplus_{j\geq 0} H^{*+2j}_W(\mathcal{X},\mathcal{W}^\gamma)\otimes_{\Z}\Q.
\]
\end{theorem}

\begin{proof}
We have already explained the computation of $E^2_{p,q}(\mathcal{C})$. The differential $d^2:E^2_{p,q}(\mathcal{C})\to E^2_{p+2,q-1}(\mathcal{C})$ vanishes for the simple reason that at least one of its domain and codomain is zero. The assertion about the collapse of the spectral sequence over $\Q$ follows from  \cite[Theorem 4.6 \& Remark 4.7]{Luck-Chern}, which is applicable here because the coefficient system $\mathcal{W}^\gamma$ has a Mackey structure (as verified in \cite[Theorem 4.2]{Dwyer}.)
\end{proof}

Combining Theorem \ref{thm:Bredon-spectral-sequence} with Theorem \ref{thm:Plancherel-2}, we find:

\begin{corollary}\label{cor:widetildeX-spectral-sequence}
Let $G$, $M$, $X$, $\sigma$, and $W$ be as in Section \ref{sec:p-adic-background}; let $C^*_r(G)_{(M,\sigma)}$ be the direct-summand of $C^*_r(G)$ defined in Definition \ref{def:Cstar-block}; let $\widetilde{X}$ be the space defined in Definition \ref{def:widetildeX}; and let $n=\dim\widetilde{X}$.  There is a spectral sequence $E$ converging to $K_*(C^*_r(G)_{(M,\sigma)})$, with 
\[
E^2_{p,q} = \begin{cases} H^{n-p+1}_W(\widetilde{X},\mathcal{W}^\gamma) & \text{if $q+n$ is odd,}\\ 0 & \text{if $q+n$ is even,}
\end{cases}
\]
and the differential $d^2$ is zero. This spectral sequence collapses rationally at the $E^2$ page, giving an isomorphism 
\[
K_*(C^*_r(G)_{(M,\sigma)})\otimes_{\Z}\Q \cong \bigoplus_{j\geq 0} H^{*+2j}_W(\widetilde{X},\mathcal{W}^\gamma)\otimes_{\Z}\Q.
\]
\hfill\qed
\end{corollary}

\begin{corollary}\label{cor:dimension-2}
In the setting of Corollary \ref{cor:widetildeX-spectral-sequence}, suppose that the torus $X$ has dimension $\leq 2$ as a real manifold. Then we have
\[
\begin{aligned}
K_0(C^*_r(G)_{(M,\sigma)}) & \cong H^0_W(\widetilde{X},\mathcal{W}^\gamma) \oplus H^2_W(\widetilde{X},\mathcal{W}^\gamma), \\
K_1(C^*_r(G)_{(M,\sigma)})&\cong H^1_W(\widetilde{X},\mathcal{W}^\gamma).
\end{aligned}
\]
\end{corollary}

\begin{proof}
We have $\dim \widetilde{X}=\dim X$ as CW-complexes, so for obvious degree reasons we must have $H^i_W(\widetilde{X},\mathcal{W}^\gamma)=0$ for all $i\geq 3$. This implies, again for trivial degree reasons, that each of the differentials $d^r$ in the spectral sequence $E$ from Corollary \ref{cor:widetildeX-spectral-sequence}, for $r\geq 3$, must be zero, and thus the spectral sequence collapses at the $E^2$ page. The convergence of the spectral sequence then means that the filtration 
\[
F_{*,p}\coloneqq \image\left( K_*(\mathcal{C}_p) \xrightarrow{K_*(\text{inclusion})} K_*(\mathcal{C})\right)
\]
of $K_*(C^*_r(G)_{(M,\sigma)})$ induced by the skeletal filtration of $\widetilde{X}$ has 
\[
F_{p+q,p}/F_{p+q,p-1} \cong E^2_{p,q} \cong \begin{cases} H^{n-p+1}_W(\mathcal{X},\mathcal{W}^\gamma) & \text{if $q+n$ is odd,}\\ 0 & \text{if $q+n$ is even,}
\end{cases}
\] 
where $n=\dim X$. All of the assertions about $K_*(C^*_r(G)_{(M,\sigma)})$ follow immediately from this identity, and from the fact (clear from the definition) that $H^0_W(\widetilde{X},\mathcal{W}^\gamma)$ is always torsion-free.
\end{proof}

\subsection{\texorpdfstring{A spectral sequence for $K_*(C(X,E,W))$}{A spectral sequence for K*(C(X,E,W))}}\label{sec:KCXEW}

The spectral sequence obtained from the space $\widetilde{X}$ (Corollary \ref{cor:widetildeX-spectral-sequence}) can be used for practical computations of $K$-theory groups of the components $C^*_r(G)_{(M,\sigma)}$ of the $C^*$-algebra of a $p$-adic reductive group $G$: see the examples in the next section. In this section we will provide a second picture of this spectral sequence, based on the space $X$. This second picture is perhaps less well suited to practical computations, because it seems to depend on the normalisation of the intertwining operators in  a more delicate way than the first picture: the characters $\i_{w,\chi}$ appear in the definition. (Cf.~Remark \ref{rem:i=1}.) On the other hand, since this second picture is based directly on the torus $X$---which is closely related to an apartment in the Bruhat-Tits building of $G$---rather than on the modified space $\widetilde{X}$, this second picture of the spectral sequence may turn out to be useful in investigating the Baum-Connes isomorphism.

We continue with the notation $G$, $M$, $X$, $\sigma$, $W$, etc, established in Section \ref{sec:p-adic-background}. To compactify the notation we will write
\[
C\coloneqq C(X,E,W) = \left\{ \sum_{w\in W} b_w w\in C(X)\rtimes_{\gamma} W\ \middle|\ \begin{aligned} &\forall \chi\in X,\ \forall w'\in W_\chi',\ \forall w\in W:\\ & b_{w'w}(\chi)  = \i_{w',\chi}\gamma(w',w) b_w(\chi)\end{aligned} \right\}.
\]
We choose and fix a $W$-CW structure on the torus $X$, and let $n$ denote the dimension of $X$ as a real manifold. For each $p\in \{0,\ldots,n+1\}$. consider the open set $Y_p\coloneqq X\setminus X_{n-p}\subseteq X$. Setting $C_p\coloneqq C(Y_p, E_{Y_p}, W)$, we have an increasing sequence of ideals
\begin{equation}\label{eq:C-filtration}
0 = C_0 \subset C_1 \subset \cdots \subset C_{n+1}=C
\end{equation}
in $C$.

The filtration \eqref{eq:C-filtration} gives rise (again, by \cite[Theorem 2.1]{Schochet-I}) to a spectral sequence $(E^r(C),d^r)$, with $d^r: E^r_{p,q}(C) \to  E^r_{p-r,q+r-1}(C)$, which converges to $K_*(C)$ and has $E^1_{p,q}(C) = K_{p+q}( C_p/C_{p-1})$.

\begin{theorem}\label{thm:C-spectral-sequence}
The spectral sequence $E(C)$ has
\[
E^2_{p,q}(C) = \begin{cases} H^{n-p+1}_W(X,\mathcal{R}) & \text{if $q+n$ is odd,}\\ 0 & \text{if $q+n$ is even,}\end{cases}
\]
and the differential $d^2$ is zero. This spectral sequence collapses rationally at the $E^2$ page, giving an isomorphism
\[
K_*(C^*_r(G)_{(M,\sigma)})\otimes_{\Z}\Q \cong \bigoplus_{j\geq 0}H_W^{*+2j}(X,\mathcal{R})\otimes_{\Z}\Q.
\]
\end{theorem}
Here $\mathcal{R}$ is the coefficient system $\mathcal{R}_z = \Rep_{\gamma,\overline{\i_z}}(W_z)$: see Lemma \ref{lem:R-def}.

\begin{proof}
Lemma \ref{lem:C-localisation} implies that for each $p$ we have a canonical isomorphism
\[
C_p/C_{p-1} \cong C(Y_p\setminus Y_{p-1}, E_{Y_p\setminus Y_{p-1}}, {W})=C(X_j\setminus X_{j-1}, E_{X_j\setminus X_{j-1}}, W),
\]
where $j\coloneqq n-p+1$. We again identify $X_j\setminus X_{j-1}$, the union of the open $j$-cells in $X$, with the product $B^j\times Z_j$. 

It is a straightforward matter to verify, using \eqref{eq:PO}, that 
\[
C(B^j\times Z_j,E_{B^j\times Z_j},W)=P_{\i,Z_j}\left(C_0(B^j\times Z_j)\rtimes_\gamma W\right),
\]
where $P_{\i,Z_j}$ is the projection in the centre of the multiplier algebra of $C_0(B^j\times Z_j)\rtimes_\gamma W$ defined by
\begin{equation}\label{eq:Pj}
P_{\i,Z_j} = \sum_{z\in Z_j} |{W}'_z|^{-1}\sum_{w\in W'_z} \i_{w,z} \delta_{B^j\times\{z\}} w,
\end{equation}
with $\delta_{B^j\times\{z\}}\in C_b(B^j\times Z_j)$ denoting the characteristic function of the open cell $e_z=B^j\times \{z\}$.

Now let $\mathcal{C}=C(X)\rtimes_\gamma W$, and let $E(\mathcal{C})$ denote the spectral sequence in $K$-theory arising, as explained above, from the filtration of $X$ by its skeleta. The inclusion of $C^*$-algebras $C\into \mathcal{C}$ is compatible with the filtrations \eqref{eq:mathcalC-filtration} and \eqref{eq:C-filtration}, so it induces a map of spectral sequences $E(C)\to E(\mathcal{C})$. The split-exactness of $K$-theory ensures that this map is injective on the $E^1$ page. 

The isomorphism $C_0(B^j\times Z_j)\rtimes_\gamma W\xrightarrow{\cong} C_0(B_j)\otimes\left( C(Z_j)\rtimes_\gamma W\right)$ sends $P_{\i,Z_j}$ to the projection $1_{C_0(B_j)}\otimes Q_{\i,Z_j}$, where $Q_{\i,Z_j}\in C(Z_j)\rtimes_\gamma W$ is given by
\[
Q_{\i,Z_j} = \sum_{z\in Z_j} |W'_{z}|^{-1}\sum_{w\in W'_z} \i_{w,z} \delta_z w.
\]
The isomorphism
\[
C(Z_j)\rtimes_\gamma W \xrightarrow{\cong} \bigoplus_{Wz\in W\backslash Z_j} C(Wz)\rtimes_\gamma W
\]
sends $Q_{\i,Z_j}$ to $\bigoplus_{Wz\in W\backslash Z_j} P_{\i,Wz}$, where $P_{\i,Wz}$ is as in \eqref{eq:PO-def}. We saw in the proof of Lemma \ref{lem:C-completeness} that the Morita equivalence $C(Wz)\rtimes_\gamma W\sim C^*_{\gamma}(W_z)$ sends the central projection $P_{\i,Wz}$ to the central projection $P_{\i,z}$, whose action on $K$-theory is to project $\Rep_{\gamma}(W_z)$ onto its direct-summand $\Rep_{\gamma,\overline{\i_z}}(W_z)$. Putting all of this together, we conclude that the image of $E^1(C)$ in $E^1(\mathcal{C})$ corresponds, under the isomorphism $E^1_{p,q}(\mathcal{C})\cong C^j(X,\mathcal{W}^\gamma)^W$ of \eqref{eq:E1widetildeC-iso} (where $q+n$ is odd), to the subcomplex $C^j(X,\mathcal{R})^W$. Taking cohomology, we obtain $E^2_{p,q}(C)\cong H^j_W(X,\mathcal{R})$ for $q+n$ odd. When $q+n$ is even, on the other hand, we have $E^1_{p,q}(C)\subseteq E^1_{p,q}(\widetilde{C})=0$, so $E^2_{p,q}(C)=0$ too. The differential $d^2$ is zero for the same reason as in Theorem \ref{thm:Bredon-spectral-sequence}.

We are left to prove that the spectral sequence collapses rationally at the $E^2$ page. We will do this by exhibiting an isomorphism between this spectral sequence and the one in Corollary \ref{cor:widetildeX-spectral-sequence}. 

Let $\widetilde{\mathcal{C}}\coloneqq C(\widetilde{X})\rtimes_\gamma W$. Recalling that $C=C(X,E,W)$ and $\mathcal{C}\coloneqq C(X)\rtimes_\gamma W$, we consider the composition of $C^*$-algebra homomorphisms
\[
\Psi : C \xrightarrow{\text{include}} \mathcal{C} \xrightarrow{\Pi^*} \widetilde{\mathcal{C}},
\]
where $\Pi^*:C(X)\rtimes_\gamma W\to C(\widetilde{X})\rtimes_\gamma W$ is the homomorphism induced by the $W$-equivariant gluing map $\Pi:\widetilde{X}\to X$. Equip $\widetilde{X}$ with the $W$-CW-structure induced by our chosen $W$-CW-structure on $X$. Since $\Pi$ maps open $j$-cells to open $j$-cells, the homomorphism $\Psi:C\to\widetilde{C}$ respects the filtrations by skeleta on either side, and thus induces a morphism of spectral sequences $E(\Psi):E(C)\to E(\widetilde{\mathcal{C}})$, where $E(\widetilde{\mathcal{C}})$ is the spectral sequence from Corollary \ref{cor:widetildeX-spectral-sequence}.

Computing this induced map on the $E^1$ page, one finds a diagram
\begin{equation}\label{eq:Phi-diagram}
\begin{gathered}
\xymatrix{
E^1_{p,q}(C) \ar[r]^-{E(\Psi)} & E^1_{p,q}(\widetilde{C}) \\
\displaystyle\bigoplus_{Wz\in W\backslash Z_j} \Rep_{\gamma,\overline{\i_z}}(W_z) \ar[r] \ar[u]^-{\cong} & \displaystyle\bigoplus_{W\widetilde{z}\in W\backslash \widetilde{Z}_j} \Rep_\gamma(W_{\widetilde{z}}) \ar[u]^-{\cong}
}
\end{gathered}
\end{equation}
in which the vertical arrows are the isomorphisms established above using Bott periodicity and the imprimitivity theorem, and where the bottom arrow is defined so as to make the diagram commute. 

To compute the bottom arrow in \eqref{eq:Phi-diagram}, we first note that Lemma \ref{lem:Pi-bijection-on-orbits} (which says that $\Pi$ induces a bijection $W\backslash Z_j \xrightarrow{\cong} W\backslash \widetilde{Z}_j$) reduces the computation to that of a single summand $\Rep_{\gamma,\overline{\i_z}}(W_z)\to \Rep_\gamma(W_{\widetilde{z}})$, where $\Pi(\widetilde{z})=z$. Recalling (from Example \ref{ex:imprimitivity}) that the Morita equivalence $C(Wz)\rtimes_\gamma W\sim C^*_\gamma(W_z)$ is implemented by the $C(Wz)\rtimes_\gamma W$-$C^*_\gamma(W_z)$ bimodule $(C(Wz)\rtimes_\gamma W)\delta_z$, and similarly for $\widetilde{z}$, we find that the bottom arrow in \eqref{eq:Phi-diagram} sends each $\gamma$-representation $H$ of $W_z$ (lying over $\overline{\i_z}$) to the $\gamma$-representation
\[
\delta_{\widetilde{z}}(C(W\widetilde z)\rtimes_\gamma W)\Pi^*(\delta_z) \otimes_{C^*_\gamma(W_z)} H
\]
of $W_{\widetilde z}$. We have 
\[
\delta_{\widetilde z}(C(W\widetilde z)\rtimes_\gamma W)\Pi_j^*(\delta_z)= \lspan\{ \delta_{\widetilde z}w\ |\ w\in W_z\},
\]
which is isomorphic to $C^*_\gamma(W_z)$ as a $C^*_\gamma(W_{\widetilde z})$-$C^*_\gamma(W_z)$ bimodule. So the bottom arrow in \eqref{eq:Phi-diagram} is given simply by restriction of $\gamma$-representations, from $W_z$ to its subgroup $W_{\widetilde{z}}$, for each $z$. 

Now Lemma \ref{lem:W_chi-splitting} implies that $W_z=W'_z\rtimes W_{\widetilde{z}}$, while Lemma \ref{lem:Wprime-normal} implies that the character $\overline{\i_z}$ of $W'_z$ is normalised by $W_z$. Thus the restriction of $\gamma$-representations from $W_z$ to $W_{\widetilde{z}}$ gives an isomorphism $\Rep_{\gamma,\overline{\i_z}}(W_z)\xrightarrow{\cong} \Rep_\gamma(W_{\widetilde{z}})$, and so the bottom arrow in \eqref{eq:Phi-diagram} is an isomorphism. This implies that the map of spectral sequences $E(\Psi)$ is an isomorphism on the $E^1$ pages, and therefore that $E(\Psi)$ is an isomorphism everywhere. Since we know that $E(\widetilde{\mathcal{C}})$ collapses over $\Q$ at the $E^2$ page, the same is true of $E(C)$.
\end{proof}

\subsection{Relation to the ABPS conjectures}\label{sec:ABPS}

In this section we will briefly indicate how the results of this paper are related to conjectures of Aubert, Baum, Plymen, and Solleveld, that have recently been verified by Solleveld. 

Let $M$ be a Levi subgroup of a $p$-adic reductive group $G$, let $\sigma$ be a unitary supercuspidal representation of $M$, and let $\germ{s}$ denote the Bernstein component in the tempered dual of $G$ corresponding to the pair $(M,\sigma)$. Thus $\germ{s}$ consists of all those isomorphism classes of irreducible tempered unitary representations of $G$ whose space of smooth vectors can be embedded as a subquotient of a representation parabolically induced from some possibly-non-unitary unramified twist of $\sigma$. (This construction goes back to \cite{Bernstein-centre}; see \cite{Renard}, for instance, for an exposition.) We let $C^*_r(G)_{\germ{s}}$ denote the direct summand of $C^*_r(G)$ corresponding to this Bernstein component. In terms of the Plancherel decomposition of $C^*_r(G)$ (Theorem \ref{thm:Plancherel}) we have
\begin{equation}\label{eq:ABPS-Cstar}
C^*_r(G)_{\germ{s}} = C^*_r(G)_{(M,\sigma)} \oplus \bigoplus_{[M_\tau,\tau]} C^*_r(G)_{(M_\tau,\tau)}
\end{equation}
where the sum is over  the (finite) set of equivalence classes of non-cuspidal discrete-series representations $\tau$ of Levi subgroups $M_\tau\subseteq G$ lying in the Bernstein component $\germ{s}$. Aubert, Baum, Plymen, and Solleveld gave a conjectural formula for $K_*(C^*_r(G)_{\germ{s}})$ in terms of twisted equivariant $K$-theory in \cite[Conjecture 5]{ABPS}, which has recently been verified in \cite[Corollary 6.26]{Solleveld-notes}. On the other hand, as we have seen in this section, our Theorem \ref{thm:Plancherel-2} gives a formula, again in terms of twisted equivariant $K$-theory, for each of the summands on the right-hand side of \eqref{eq:ABPS-Cstar}. Our $K$-theoretic result is, in this sense, a refinement of the ABPS formula; but the connection is not as straightforward as that characterisation might suggest since, as we shall now explain, it involves a nontrivial isomorphism. 

Let $X_M$ denote the torus of unitary unramified characters of $M$, and let $W_\sigma$ and $\gamma_\sigma$ be the finite group and $2$-cocycle associated to $\sigma$ in Definition \ref{def:W} and Theorem \ref{thm:pI-properties}. We deploy the corresponding notation $X_{M_\tau}$, $W_\tau$, and $\gamma_\tau$ for each pair $(M_\tau,\tau)$ appearing on the right-hand side of \eqref{eq:ABPS-Cstar}. Taking $K$-theory on both sides of \eqref{eq:ABPS-Cstar} and applying \cite[Corollary 6.26]{Solleveld-notes} on the left and Theorem \ref{thm:Plancherel-2} on the right, we obtain 
\begin{equation}\label{eq:ABPS-K-iso}
K^*_{W_\sigma,\gamma_\sigma}(X_M) \cong K^*_{W_\sigma,\gamma_\sigma}\left(\widetilde{X}_M\right) \oplus \bigoplus_{\tau} K^*_{W_\tau,\gamma_\tau}\left(\widetilde{X}_{M_\tau}\right).
\end{equation}
Here the tildes denote blow-ups along the points with nontrivial scalar intertwining groups, as in Definition \ref{def:widetildeX}.

The isomorphism \eqref{eq:ABPS-K-iso} reflects a lot of interesting representation theory. At the most superficial level, for instance, it shows that the Bernstein component $\germ{s}$ can contain a non-cuspidal discrete-series representation only if $(W_\sigma)_\chi'\neq 1$ for some $\chi\in X_M$. This last fact is of course already known, thanks to the highly non-trivial chain of equivalences connecting scalar normalised intertwining operators to zeros of the Plancherel density at unitary unramified characters, to poles of the Plancherel density at non-unitary unramified characters, to non-cuspidal discrete-series representations. (See \cite{Silberger-dimension,Silberger-special,Silberger-discrete,Heiermann} for the various links in this chain.) In its full depth the isomorphism \eqref{eq:ABPS-K-iso} describes, in terms of the geometric blow-up construction $X\mapsto \widetilde{X}$,  precisely how the connection between scalar intertwiners and non-cuspidal discrete series is manifested in $K$-theory. 

A corresponding description of how this connection plays out at the level of the tempered dual  is  given by comparing \cite[Conjecture 2]{ABPS}---which was verified in \cite[Theorem 9.9]{Solleveld-endomorphism}---with our Corollary \ref{cor:extended-quotient}. This comparison gives a bijection of twisted extended quotients
\begin{equation}\label{eq:ABPS-duals}
\left( {X}/\!\!/ W_\sigma\right)_{\gamma_\sigma} \cong \left(\widetilde{X}_M/\!\!/ W_\sigma\right)_{\gamma_\sigma}\sqcup \bigsqcup_{[M_\tau,\tau]} \left( \widetilde{X}_{M_\tau}/\!\!/ W_\tau\right)_{\gamma_\tau}.
\end{equation}
It should be possible, and it would certainly be interesting, to establish the isomorphism \eqref{eq:ABPS-K-iso} and the bijection \eqref{eq:ABPS-duals} directly, using known results from representation theory. The fact that these identifications can alternatively be obtained as a byproduct of computing the structure and $K$-theory of $C^*_r(G)$ in two different ways is noteworthy in its own right.

\section{Examples}\label{sec:examples}

\subsection{Discrete-series components}\label{sec:discrete-series}

Suppose that $M=G$, so that $X$ is the torus of unramified unitary characters of $G$, and $\sigma$ is a discrete-series representation of $G$. The Weyl group $W_G$ is trivial, so we have $W=X_\sigma = \{\psi\in X\ |\ \sigma\cong \sigma\otimes\psi\}$. Clearly $W$, as a subgroup of $X$, acts freely on $X$, and in particular the groups $W'_\chi$ are all trivial, so we have
\[
C^*_r(G)_{(G,\sigma)} \sim C(X,E,W) = C(X)\rtimes_\gamma X_\sigma.
\]
In this case we have $\widetilde{X}=X$, again because of the triviality of the groups $W'_\chi$.

If the cocycle $\gamma$ is trivial then the freeness of the action of $X_\sigma$ on $X$ implies that $C(X)\rtimes X_\sigma$ is Morita equivalent to $C(X/X_\sigma)$ (\cite{Rieffel-applications-transformation}), and so we have $K_*(C^*_r(G)_{(G,\sigma)})\cong K^*(X/X_\sigma)$.

\subsection{\texorpdfstring{Trivial $R$-groups}{Trivial R-groups}}\label{subsec:R-trivial}

Suppose that all of the induced representations $\Ind_P^G(\sigma\otimes\chi)$ are irreducible. Harish-Chandra's intertwining algebra theorem implies that this occurs if and only if $W'_\chi = W_\chi$ for every $\chi\in X$; and Lemma \ref{lem:W_chi-splitting} implies that the latter occurs if and only if the $W$-action on the modified parameter space $\widetilde{X}$ is free. We have, as always (by Theorem \ref{thm:Plancherel-2}), $C^*_r(G)_{(M,\sigma)} \sim C(\widetilde{X})\rtimes_\gamma W$. If the cocycle $\gamma$ is trivial then the freeness of the $W$-action implies that
\[
C^*_r(G)_{(M,\sigma)}\sim C(\widetilde{X}/W) \cong C(X/W).
\]
(A very similar result was noted by Plymen in \cite[2.7 Corollary]{Plymen-Reduced}.) If we don't assume that $\gamma$ is trivial, we can still conclude that $C^*_r(G)_{(M,\sigma)}$ is a continuous-trace algebra over $\widetilde{X}/W\cong X/W$ (\cite[Theorem 5, p.18]{Wassermann-thesis}; cf.~\cite[p.168]{Rosenberg-Mackey}).  In particular, if $H^3(X/W,\Z)$ is trivial then we can again conclude that $C^*_r(G)_{(M,\sigma)}\sim C(X/W)$.

An interesting example of this special case has recently been studied in  \cite{Aubert-Plymen-Klein}. (The example, which is due originally to Kutzko, has previously been studied from different angles in \cite[Section 4]{Roche-parabolic-induction}, \cite[11.8]{Goldberg-Roche-SLn}, and \cite{AFO}.) Here $G=\SL_8(F)$, $M$ is the block-diagonal Levi subgroup $(\GL_2(F)\times \GL_2(F)\times \GL_4(F))\cap G$, and $\sigma$ is a certain supercuspidal representation of $M$. Roche showed in \cite{Roche-parabolic-induction} that all of the parabolically induced representations $\Ind_P^G(\sigma\otimes\chi)$ are irreducible. The group $W$ in this example has order $2$, so the cocycle $\gamma$ can be trivialised (see \eqref{eq:gamma-inverse}), and we consequently have $C^*_r(G)_{(M,\sigma)} \sim C(X/W)$. Aubert and Plymen point out in \cite{Aubert-Plymen-Klein} that $X/W$ is  a Klein bottle, and so we have
\[
\begin{aligned}
K_0(C^*_r(G))_{(M,\sigma)} & \cong H^0(X/W,\Z)\oplus H^2(X/W,\Z) \cong \Z \oplus \Z/2\Z,\\
K_1(C^*_r(G))_{(M,\sigma)} & \cong H^1(X/W,\Z)\cong \Z.
\end{aligned}
\]

\subsection{\texorpdfstring{$\dim X=1$}{dim X = 1}}\label{sec:1d}

Suppose that the torus $X$ is one-dimensional (as a real manifold.) We have $W_M=\{1\}$ or $\{1,s\}$, where $s\chi=\chi^{-1}$ for each $\chi\in X$. Given a discrete-series representation $\sigma$ of $M$, we consider the following two cases:

\paragraph*{Case 1: either $W_M=\{1\}$, or $W_M=\{1,s\}$ where for every $\chi\in X$ we have ${}^s\sigma\not\cong \sigma\otimes\chi$.} Then we have $W=X_\sigma$, a finite subgroup of $X$. Since the $W$-action on $X$ is free we have $X'=\emptyset$, $\widetilde{X}=X$, and $H^*_W(X,\mathcal{W}^\gamma)=H^*(X/X_\sigma, \Z)$, the ordinary cohomology of the circle $X/X_\sigma$ with integer coefficients. Corollary \ref{cor:dimension-2} thus gives
\[
K_0(C^*_r(G)_{(M,\sigma)})\cong K_1(C^*_r(G)_{(M,\sigma)})\cong \Z
\]
in this case.

\paragraph*{Case 2: $W_M=\{1,s\}$, and ${}^s\sigma\cong \sigma\otimes\chi$ for some $\chi\in X$.} By replacing $\sigma$ with $\sigma\otimes\psi$, where $\psi\in X$ has $\psi^2=\chi$, we may in fact assume that ${}^s\sigma\cong \sigma$. We then have $W=X_\sigma\rtimes \langle s\rangle$, a dihedral group (including the degenerate cases $|X_\sigma|=1$ or $2$). The points in $X$ with nontrivial stabiliser for the $W$-action fall into two orbits, each permuted simply transitively by $X_\sigma$; let us call these two orbits $\mathcal{O}_\bullet$ and $\mathcal{O}_\circ$. To illustrate in the case where $|X_\sigma|=3$:

\begin{center}
$X$ : \quad \begin{tikzpicture}[baseline=(current bounding box.center)]
\draw (0,0) circle (1)
node[draw,circle, fill=black,inner sep=2pt] at (0:1) {}
node[draw,circle, fill=black,inner sep=2pt] at (120:1) {}
node[draw,circle, fill=black,inner sep=2pt] at (240:1) {}
node[draw,circle, fill=white,inner sep=2pt] at (60:1) {}
node[draw,circle, fill=white,inner sep=2pt] at (180:1) {}
node[draw,circle, fill=white,inner sep=2pt] at (300:1) {};

\end{tikzpicture}
\end{center}
The black vertices are the elements of $\mathcal{O}_{\bullet} = X_\sigma$, while the white vertices are the elements of $\mathcal{O}_{\circ}$. The group $W$ is the symmetry group of the triangle with vertices $\mathcal{O}_\bullet$ (or, equivalently, $\mathcal{O}_\circ$.)

Returning to the general case, there are three essentially different possibilities for the set $X'$:
(a) $X'=\emptyset$; (b) $X'=\mathcal{O}_\bullet$; or (c) $X'=\mathcal{O}_\bullet\cup \mathcal{O}_\circ$. (The fourth case, where $X'=\mathcal{O}_\circ$, is essentially identical to case (b).) 

We can visualise $\germ{a}_M^*$ as a line in which the preimages of the orbits $\mathcal{O}_\bullet$ and $\mathcal{O}_\circ$ under the quotient map $\germ{a}_M^*\to X$ are marked by $\bullet$s and $\circ$s, respectively:
\begin{center}
$\germ{a}_M^*$\quad :\quad 
\begin{tikzpicture}[scale=1.25] 
\draw[thick,dotted] (-3,0) -- (-2.75,0);
\draw (-2.75,0) -- (3.75,0);
\draw[thick, dotted] (3.75,0) -- (4,0);
\node[draw,circle, fill=white,inner sep=2pt] at (-2,0) {};
\node[draw,circle, fill=black,inner sep=2pt] at (-1,0) {};
\node[draw,circle, fill=white,inner sep=2pt] at (0,0) {};
\node[draw,circle, fill=black,inner sep=2pt] at (1,0) {};
\node[draw,circle, fill=white,inner sep=2pt] at (2,0) {};
\node[draw,circle, fill=black,inner sep=2pt] at (3,0) {};
\end{tikzpicture}
\end{center}

Referring to the cases (a), (b), and (c) defined above, the subset $(\germ{a}_M^*)'$ of $\germ{a}_M^*$ is either (a) $\emptyset$, (b) all the $\bullet$s, or (c) the $\bullet$s and the $\circ$s. To form $\widetilde{\germ{a}_M^*}$ we remove $(\germ{a}_M^*)'$ from $\germ{a}_M^*$, leaving a disjoint union of open intervals (one interval in case (a), countably many in cases (b) and (c).) The space $\widetilde{\germ{a}_M^*}$ is the disjoint union of the closure of these intervals:
\begin{enumerate}[\rm(a)]
\item \qquad 
\begin{tikzpicture}[scale=1.25] 
\draw[thick,dotted] (-3,0) -- (-2.75,0);
\draw (-3,0) -- (3.75,0);
\draw[thick, dotted] (3.75,0) -- (4,0);
\node[draw,circle, fill=white,inner sep=2pt] at (-2,0) {};
\node[draw,circle, fill=black,inner sep=2pt] at (-1,0) {};
\node[draw,circle, fill=white,inner sep=2pt] at (0,0) {};
\node[draw,circle, fill=black,inner sep=2pt] at (1,0) {};
\node[draw,circle, fill=white,inner sep=2pt] at (2,0) {};
\node[draw,circle, fill=black,inner sep=2pt] at (3,0) {};
\end{tikzpicture}

\item \qquad  
\begin{tikzpicture}[scale=1.25]
\def\t{.15}
\draw[thick,dotted] (-3,0) -- (-2.75,0);
\draw (-2.75,0) -- (-2,0);
\draw (-2,0) -- (-1-\t,0);
\draw (-1+\t,0) -- (0,0);
\draw (0,0) -- (1-\t,0);
\draw (1+\t,0) -- (2,0);
\draw (2,0) -- (3-\t,0);
\draw (3+\t,0) -- (3.75,0);
\draw[thick,dotted] (3.75,0) -- (4,0);

\node[draw,circle, fill=white,inner sep=2pt] at (-2,0) {};
\node[draw,circle, fill=black,inner sep=2pt] at (-1-\t,0) {};
\node[draw,circle, fill=black,inner sep=2pt] at (-1+\t,0) {};
\node[draw,circle, fill=white,inner sep=2pt] at (0,0) {};
\node[draw,circle, fill=black,inner sep=2pt] at (1-\t,0) {};
\node[draw,circle, fill=black,inner sep=2pt] at (1+\t,0) {};
\node[draw,circle, fill=white,inner sep=2pt] at (2,0) {};
\node[draw,circle, fill=black,inner sep=2pt] at (3-\t,0) {};
\node[draw,circle, fill=black,inner sep=2pt] at (3+\t,0) {};
\end{tikzpicture}

\item \qquad  
\begin{tikzpicture}[scale=1.25]
\def\t{.15}
\draw[thick,dotted] (-3,0) -- (-2.75,0);
\draw (-2.75,0) -- (-2-\t,0);
\draw (-2+\t,0) -- (-1-\t,0);
\draw (-1+\t,0) -- (0-\t,0);
\draw (0+\t,0) -- (1-\t,0);
\draw (1+\t,0) -- (2-\t,0);
\draw (2+\t,0) -- (3-\t,0);
\draw (3+\t,0) -- (3.75,0);
\draw[thick,dotted] (3.75,0) -- (4,0);

\node[draw,circle, fill=white,inner sep=2pt] at (-2-\t,0) {};
\node[draw,circle, fill=white,inner sep=2pt] at (-2+\t,0) {};
\node[draw,circle, fill=black,inner sep=2pt] at (-1-\t,0) {};
\node[draw,circle, fill=black,inner sep=2pt] at (-1+\t,0) {};
\node[draw,circle, fill=white,inner sep=2pt] at (0-\t,0) {};
\node[draw,circle, fill=white,inner sep=2pt] at (0+\t,0) {};
\node[draw,circle, fill=black,inner sep=2pt] at (1-\t,0) {};
\node[draw,circle, fill=black,inner sep=2pt] at (1+\t,0) {};
\node[draw,circle, fill=white,inner sep=2pt] at (2-\t,0) {};
\node[draw,circle, fill=white,inner sep=2pt] at (2+\t,0) {};
\node[draw,circle, fill=black,inner sep=2pt] at (3-\t,0) {};
\node[draw,circle, fill=black,inner sep=2pt] at (3+\t,0) {};
\end{tikzpicture}
\end{enumerate}

Taking the quotient by the lattice $\Lambda$, the space $\widetilde{X}$ is either a circle or a disjoint union of closed intervals. To illustrate in the case $|X_\sigma|=3$:

\begin{center}
\begingroup
\setlength{\tabcolsep}{20pt} 
\begin{tabular}{ccc}
Case (a) & Case (b) & Case (c) \\[5pt]

\begin{tikzpicture}
\draw (0,0) circle (1)
node[draw,circle, fill=black,inner sep=2pt] at (0:1) {}
node[draw,circle, fill=black,inner sep=2pt] at (120:1) {}
node[draw,circle, fill=black,inner sep=2pt] at (240:1) {}
node[draw,circle, fill=white,inner sep=2pt] at (60:1) {}
node[draw,circle, fill=white,inner sep=2pt] at (180:1) {}
node[draw,circle, fill=white,inner sep=2pt] at (300:1) {};

\end{tikzpicture}

&

\begin{tikzpicture}
\def\t{10}
\draw[domain=0+\t:120-\t] plot ({cos(\x)}, {sin(\x)});
\draw[domain=120+\t:240-\t] plot ({cos(\x)}, {sin(\x)});
\draw[domain=240+\t:360-\t] plot ({cos(\x)}, {sin(\x)})
node[draw,circle, fill=black,inner sep=2pt] at (\t:1) {}
node[draw,circle, fill=white,inner sep=2pt] at (60:1) {}
node[draw,circle, fill=black,inner sep=2pt] at (120-\t:1) {}
node[draw,circle, fill=black,inner sep=2pt] at (120+\t:1) {}
node[draw,circle, fill=white,inner sep=2pt] at (180:1) {}
node[draw,circle, fill=black,inner sep=2pt] at (240-\t:1) {}
node[draw,circle, fill=black,inner sep=2pt] at (240+\t:1) {}
node[draw,circle, fill=white,inner sep=2pt] at (300:1) {}
node[draw,circle, fill=black,inner sep=2pt] at (360-\t:1) {};

\end{tikzpicture}

&

\begin{tikzpicture}
\def\t{10}
\draw[domain=0+\t:60-\t] plot ({cos(\x)}, {sin(\x)});
\draw[domain=60+\t:120-\t] plot ({cos(\x)}, {sin(\x)});
\draw[domain=120+\t:180-\t] plot ({cos(\x)}, {sin(\x)});
\draw[domain=180+\t:240-\t] plot ({cos(\x)}, {sin(\x)});
\draw[domain=240+\t:300-\t] plot ({cos(\x)}, {sin(\x)});
\draw[domain=300+\t:360-\t] plot ({cos(\x)}, {sin(\x)})
node[draw,circle, fill=black,inner sep=2pt] at (\t:1) {}
node[draw,circle, fill=white,inner sep=2pt] at (60-\t:1) {}
node[draw,circle, fill=white,inner sep=2pt] at (60+\t:1) {}
node[draw,circle, fill=black,inner sep=2pt] at (120-\t:1) {}
node[draw,circle, fill=black,inner sep=2pt] at (120+\t:1) {}
node[draw,circle, fill=white,inner sep=2pt] at (180-\t:1) {}
node[draw,circle, fill=white,inner sep=2pt] at (180+\t:1) {}
node[draw,circle, fill=black,inner sep=2pt] at (240-\t:1) {}
node[draw,circle, fill=black,inner sep=2pt] at (240+\t:1) {}
node[draw,circle, fill=white,inner sep=2pt] at (300-\t:1) {}
node[draw,circle, fill=white,inner sep=2pt] at (300+\t:1) {}
node[draw,circle, fill=black,inner sep=2pt] at (360-\t:1) {};

\end{tikzpicture}

\end{tabular}
\endgroup

\end{center}

A fundamental domain for the $W$-action on $\widetilde{X}$ is a line segment \ \begin{tikzpicture} 
\draw (0,0) -- (1,0)
node[draw,circle, fill=white,inner sep=2pt] at (0,0) {}
node[draw,circle, fill=black,inner sep=2pt] at (1,0) {};
\end{tikzpicture},
whose interior has trivial stabiliser, and with the stabilisers of the two endpoints being:

\begin{center}
\begingroup
\setlength{\tabcolsep}{20pt} 
\begin{tabular}{ccc}
Case (a) & Case (b) & Case (c) \\[5pt]

\begin{tikzpicture}[scale=1.5]
\draw (0,0) -- (1,0)
node[draw,circle, fill=white,inner sep=2pt] at (0,0) {}
node[draw,circle, fill=black,inner sep=2pt] at (1,0) {};
\node at (0,.3) {$\{1,s_1\}$};
\node at (1,.3) {$\{1,s_2\}$};
\end{tikzpicture}

&

\begin{tikzpicture}[scale=1.5]
\draw (0,0) -- (1,0)
node[draw,circle, fill=white,inner sep=2pt] at (0,0) {}
node[draw,circle, fill=black,inner sep=2pt] at (1,0) {};
\node at (0,.3) {$\{1,s_1\}$};
\node at (1,.3) {$\{1\}$};
\end{tikzpicture}

&

\begin{tikzpicture}[scale=1.5]
\draw (0,0) -- (1,0)
node[draw,circle, fill=white,inner sep=2pt] at (0,0) {}
node[draw,circle, fill=black,inner sep=2pt] at (1,0) {};
\node at (0,.3) {$\{1\}$};
\node at (1,.3) {$\{1\}$};
\end{tikzpicture}

\end{tabular}
\endgroup

\end{center}
(Here $s_1$ and $s_2$ are involutions in $W$.)

Since the stabilisers of the points of $\widetilde{X}$ all have order $\leq 2$, we can arrange that the cocycle $\gamma$ is trivial on all of the stabilisers (see \eqref{eq:gamma-inverse}), and hence ignore $\gamma$ for this computation. In case (a), the Bredon cohomology $H^*_W(\widetilde{X},\mathcal{W}^\gamma)$ is the cohomology of the  cochain complex (concentrated in degrees $0$ and $1$)
\begin{equation}\label{eq:1d-case2a-complex}
\Rep \{1,s_1\} \oplus \Rep \{1,s_2\} \xrightarrow{\res \oplus -\res} \Rep \{1\},
\end{equation}
where $\res$ denotes restriction of representations. We can decompose $\Rep\{1,s_i\} = \Z(\triv) \oplus \Z(\triv-\sign)$, where $\triv$ and $\sign$ are respectively the trivial and the nontrivial characters. Since $\res(\triv-\sign)=0$, the complex \eqref{eq:1d-case2a-complex} decomposes into the direct sum of the complexes
\[
\Z \oplus \Z \xrightarrow{\id \oplus -\id} \Z \qquad \text{and}\qquad \Z\oplus \Z \xrightarrow{0} 0,
\]
or in other words, the ordinary cellular cohomology complex of the space
\begin{center}
\begin{tikzpicture}[scale=1.5]
\draw (0,0) -- (1,0)
node[draw,circle, fill=white,inner sep=2pt] at (0,0) {}
node[draw,circle, fill=black,inner sep=2pt] at (1,0) {};
\node[draw,circle, fill=white,inner sep=2pt] at (0,.2) {};
\node[draw,circle, fill=black,inner sep=2pt] at (1,.2) {};
\end{tikzpicture}
\end{center}
We conclude that in case (a) we have  $H^0_W(\widetilde{X},\mathcal{W}^\gamma)\cong \Z^3$ while $H^1_W(\widetilde{X},\mathcal{W}^\gamma)=0$. A similar computation shows that in case (b) we have
\[
H^*_W(\widetilde{X},\mathcal{W}^\gamma) = H^*\left(\begin{tikzpicture}
\draw (0,0) -- (1,0)
node[draw,circle, fill=white,inner sep=2pt] at (0,0) {}
node[draw,circle, fill=black,inner sep=2pt] at (1,0) {};
\node[draw,circle, fill=white,inner sep=2pt] at (0,.25 ) {};
\end{tikzpicture}
, \Z \right) = \begin{cases} \Z^2 & (\ast=0)\\ 0 & \text{(otherwise.)}\end{cases}
\]
In case (c) the $W$-action on $\widetilde{X}$ is free, so in this case
\[
H^*_W(\widetilde{X},\mathcal{W}^\gamma) = H^*\left(\begin{tikzpicture}
\draw (0,0) -- (1,0)
node[draw,circle, fill=white,inner sep=2pt] at (0,0) {}
node[draw,circle, fill=black,inner sep=2pt] at (1,0) {};
\end{tikzpicture}
, \Z \right) = \begin{cases} \Z & (\ast=0)\\ 0 & \text{(otherwise.)}\end{cases}
\]

In summary:

\begin{corollary}\label{cor:dim-1}
If the torus $X_M$ of unramified unitary characters of $M$ has dimension $1$, then for each discrete-series representation $\sigma$ of $M$ the pair of $K$-theory groups $(K_0(C^*_r(G)_{(M,\sigma)}), K_1(C^*_r(G)_{(M,\sigma)}))$ is one of the following:
\[
(\Z,\Z),\quad (\Z^3,0), \quad (\Z^2,0),\quad (\Z,0).
\]
\hfill\qed
\end{corollary}

\subsection{\texorpdfstring{$\Sp_4(F)$}{Sp4(F)}}\label{sec:Sp4}

In this section we will compute $K_*(C^*_r(G)_{(M,\sigma)})$ for all possible pairs $(M,\sigma)$, where $G$ is the symplectic group $\Sp_4(F)$. We use the following coordinates for $G$:
\[
G = \{g\in \GL_4(F)\ |\ g^{\transpose}Jg=J\},
\]
where $J\smallbmat{a\\b\\c\\d}=\smallbmat{c\\d\\-a\\-b}$, and where $\transpose$ denotes the transpose. 

Up to conjugacy there are four possibilities for the Levi subgroup $M$, which we shall consider in turn: $M=G$; $M\cong \GL_2(F)$; $M\cong \SL_2(F)\times \GL_1(F)$; and $M\cong \GL_1(F)\times \GL_1(F)$. In each of these cases the torus $X_M$ has dimension $\leq 2$, so Corollary \ref{cor:dimension-2} applies: that is, the spectral sequence computing $K_*(C^*_r(G)_{(M,\sigma)})$ collapses (already over $\Z$) at the $E^2$ page, giving an identification of the $K$-theory groups in terms of the Bredon cohomology of the associated space $\widetilde{X}$. As explained in Remark \ref{rem:gamma}, we can (and will) assume in all of these computations that the $2$-cocycle $\gamma$ is trivial.

\subsubsection{\texorpdfstring{$M=\Sp_4(F)$}{M=Sp4(F)}}

Since the group $G$ is semisimple, its only unramified character is the trivial one. So for each discrete-series representation $\sigma$ of $G$ we have $C^*_r(G)_{(G,\sigma)}\cong \Compact(H_\sigma)$, and so $K_0(C^*_r(G)_{(G,\sigma)})\cong \Z$ and $K_1(C^*_r(G)_{(G,\sigma)})=0$.

\subsubsection{\texorpdfstring{$M\cong \GL_2(F)$ and $M\cong \SL_2(F)\times \GL_1(F)$}{M = GL2(F) and M = SL2(F) x GL1(F)}}

In each of these cases we have $X\cong \T$, so the analysis of Section \ref{sec:1d} shows that the pair $(K_0,K_1)$ for the $C^*$-algebra $C^*_r(G)_{(M,\sigma)}$ is $(\Z,\Z)$, $(\Z^3,0)$, $(\Z^2,0)$, or $(\Z,0)$.

\subsubsection{\texorpdfstring{$M\cong \GL_1(F)\times \GL_1(F)$}{M = GL1(F) x GL1(F)}}

Let
\[
M = \left\{ g_{(a,b)}\coloneqq \smallbmat{a & & & \\ & b & & \\ & &  a^{-1} & \\ & & & b^{-1}}\ \middle|\ a,b\in F^\times\right\}.
\]
(Empty spaces in matrices indicate zeros.) We choose a uniformiser $q\in F$, and let $\mathcal{O}$ denote the ring of integers of $F$. In this case we have $M=A_M$. The map sending $(n,m)\in \Z^2$ to the rational character $g_{(a,b)}\mapsto a^n b^m$ of $M$ gives an isomorphism $\Hom_F(M,F^\times)\cong \Z^2$. For notational convenience we will rescale the coordinates in $\germ{a}_M^*=\Hom_F(M,F^\times)\otimes_{\Z} \R\cong \R^2$ so that the covering map $\Theta:\germ{a}_M^*\to X$ sends each point $(x,y)\in \R^2$ to the unramified character
\[
\Theta(x,y): g_{(aq^n,bq^m)}\mapsto \exp(\pi i(xn+ym)),
\]
for $a,b\in \mathcal{O}^\times$ and $n,m\in \Z$. In these coordinates we have $\Lambda=(2\Z)^2 \subset \R^2$.

The Weyl group $W_M=N_G(M)/M$ is generated by the automorphisms
\[
s: g_{(a,b)}\mapsto g_{(b,a)}\qquad \text{and}\qquad t:g_{(a,b)}\mapsto g_{(a,b^{-1})}
\]
of $M$. Letting $-1$ denote the element $(st)^2\in W_M$, we have $W_M=\{\pm 1, \pm s, \pm t,\pm st\}$. The action of $W_M$ on $\germ{a}_M^*\cong \R^2$ is given by $s\mapsto \smallbmat{& 1\\ 1&}$ and $t\mapsto \smallbmat{1 & \\ & -1}$. Note that this mapping sends $-1\in W_M$ to $-\id_{\R^2}$. The root hyperplanes $V_\rho\subset \germ{a}_M^*$ are the lines $x=0$, $y=0$, $x=y$, and $x=-y$.

Since $M$ is abelian the discrete-series representations of $M$ are precisely the unitary characters $M\to \T$, and two such characters are equivalent if and only if they are equal. Thus we have
\[
E^2(M)=\left\{ \sigma=(\sigma_1,\sigma_2):g_{a,b}\mapsto \sigma_1(a)\sigma_2(b)\ \middle|\ \sigma_1,\sigma_2\in\widehat{F^\times}\right\}.
\]
In the action of $W_M$ on $E^2(M)$ the generators $s,t\in W_M$ act by ${}^s(\sigma_1,\sigma_2)=(\sigma_2,\sigma_1)$ and ${}^t(\sigma_1,\sigma_2)=(\sigma_1,\overline{\sigma_2})$. 

Every $X$-orbit in $E^2(M)$ contains a unique character from the set
\[
E^2(M)^\circ\coloneqq \{\sigma=(\sigma_1,\sigma_2)\in E^2(M)\ |\ \sigma_1(q)=\sigma_2(q)=1\} \cong \widehat{\mathcal{O}^\times}\times \widehat{\mathcal{O}^\times}.
\]
The action of $W_M$ on $E^2(M)$ stabilises the subset $E^2(M)^\circ$, and so for each $\sigma\in E^2(M)^\circ$ the group $W$ of Definition \ref{def:W} is given by
\[
W=(W_M)_\sigma = \{w\in W_M\ |\ {}^w\sigma=\sigma\}.
\]
Each $W_M$-orbit in $E^2(M)^\circ$ contains a character $\sigma$ falling into one of the following eight cases. (In this table `order' refers to the order of elements in the group $\widehat{\mathcal{O}^\times}$.)

\begin{center}
\def\arraystretch{1.2}
\begin{tabular}{c|l|c}
Case\# & Condition on $\sigma=(\sigma_1,\sigma_2)$ & $W$ \\
\hline
1 & $\sigma_1=\sigma_2=1$ & $W_M$ \\
2 & $\sigma_1=\sigma_2$, $\order(\sigma_i)=2$ & $W_M$  \\
3 & $\sigma_1=\sigma_2$, $\order(\sigma_i)>2$ & $\{1,s\}$  \\
4 & $\order(\sigma_1)=2$, $\sigma_2=1$   & $\{\pm 1,\pm t\}$ \\
5 & $\sigma_1\neq \sigma_2$, $\order(\sigma_i)=2$ & $\{\pm 1,\pm t\}$  \\
6 & $\order(\sigma_1)>2$, $\sigma_2=1$   & $\{1,t\}$  \\
7 & $\order(\sigma_1)>2$, $\order(\sigma_2)=2$ &  $\{1,t\}$ \\
8 & $\sigma_1\neq \sigma_2$, $\order(\sigma_i)>2$  & $\{1\}$
\end{tabular}
\end{center}

Now fix a $\sigma\in E^2(M)^\circ$. A theorem of Keys \cite[Theorem 1, p.364]{Keys} allows us to identify the group $W'_\chi$ for each unramified character $\chi\in X$. We write $\chi=(\chi_1,\chi_2)$, where $\chi_1,\chi_2\in F^\times\to \T$ are the unramified characters for which $\chi(g_{(a,b)})=\chi_1(a)\chi_2(b)$. The group $W'_\chi$ is generated by those reflections $\pm s,\pm t$ that it contains. Computing using Keys's result, we find the following criteria:

\begin{equation}\label{eq:r-table}
\def\arraystretch{1.2}
\begin{tabular}{c|l}
Reflection $r$ & $r\in W'_\chi$ if and only if: \\
\hline
$\phantom{-}s$ & $\sigma_1=\sigma_2$ and $\chi_1=\chi_2$ \\
$-s$ & $\sigma_1=\overline{\sigma_2}$ and $\chi_1=\overline{\chi_2}$\\
$\phantom{-}t$ & $\sigma_2=1$ and $\chi_2=1$\\
$-t$ & $\sigma_1=1$ and $\chi_1=1$\\
\end{tabular}
\end{equation}

To proceed further we examine each of the eight cases in turn; see the details below. The results are:

\begin{center}
\def\arraystretch{1.2}
\begin{tabular}{c|c|c|c}
Case\# & $H^0_W(\widetilde{X},\mathcal{W}^\gamma)$ & $H^1_W(\widetilde{X},\mathcal{W}^\gamma)$ & $H^2_W(\widetilde{X},\mathcal{W}^\gamma)$ \\
\hline
1 & $\Z^2$ & $0$ & $0$ \\
2 & $\Z^4$ & $0$ & $0$ \\
3 & $\Z\phantom{^0}$ & $\Z$ & $0$ \\
4 & $\Z^6$ & $0$ & $0$ \\
5 & $\Z^9$ & $0$ & $0$ \\
6 & $\Z^2$ & $\Z^2$ & $0$ \\
7 & $\Z^3$ & $\Z^3$ & $0$ \\
8 & $\Z\phantom{^0}$ & $\Z^2$ & $\Z$ 
\end{tabular}
\end{center}

Recall that corollary \ref{cor:dimension-2} implies that that 
\[
K_0(C^*_r(G)_{(M,\sigma)})\cong H^0_W(\widetilde{X},\mathcal{W}^\gamma)\oplus H^2_W(\widetilde{X},\mathcal{W}^\gamma),\quad  K_1(C^*_r(G)_{(M,\sigma)})\cong H^1_W(\widetilde{X},\mathcal{W}^\gamma),
\]
so the table above gives the complete list of possibilities for the $K$-theory of $C^*_r(G)_{(M,\sigma)}$ when $M$ is minimal.

\paragraph{Case 1: $\sigma$ is trivial.}

In this case we have $W=W_M$. To identify the space $\widetilde{X}$, note that all of the conditions on $\sigma$ in table \eqref{eq:r-table} are satisfied, and so $(\germ{a}_M^*)'$ is the union of the translates, by the lattice $\Lambda = (2\Z)^2$, of the root hyperplanes $x=0$, $y=0$, $x=y$, and $x=-y$ in $\germ{a}_M^*\cong \R^2$. Removing these lines from $\germ{a}_M^*$ leaves a disjoint union of open triangles, and $\widetilde{\germ{a}_M^*}$ is the disjoint union of the closures of these triangles. In a picture, showing the square $[-2,2]^2$ in $\R^2\cong \germ{a}_M^*$:

\begin{center}
\begin{tabular}{ccc}
$\germ{a}_M^*\setminus (\germ{a}_M^*)'$ &  & $\widetilde{\germ{a}_M^*}$ \\ \\
\begin{tikzpicture}[dotted,thick]
\draw (-2,-2) -- (-2,2);
\draw (0, -2) -- (0, 2);
\draw (2, -2) -- (2, 2);
\draw (-2,-2) -- (2, -2);
\draw (-2, 0) -- (2, 0);
\draw (-2,2) -- (2,2);

\draw (-2,-2) -- (2,2);
\draw (-2,0) -- (0,2);
\draw (0,-2) -- (2,0);

\draw (-2,2) -- (2,-2);
\draw (-2,0) -- (0,-2);
\draw (0,2) -- (2,0);
\end{tikzpicture}
&
\qquad\qquad
&
\begin{tikzpicture}
\def\t{.15}
\foreach \x in {-2,0}{
\foreach \y in {-2,0}{
\draw (\x+\t*0.9239,\y+\t*0.3827) -- (\x+1,\y+\t*0.3827+1-\t*0.9239) -- (\x+2-\t*0.9239 , \y+\t*0.3827) -- (\x+\t*0.9239,\y+\t*0.3827);

 \draw (\x+\t*0.3827,\y+\t*0.9239) -- (\x+\t*0.3827+1-\t*0.9239,\y+1) -- ( \x+\t*0.3827, \y+2-\t*0.9239 ) -- (\x+\t*0.3827, \y+\t*0.9239);

 \draw (\x+\t*0.9239,\y+2-\t*0.3827) -- (\x+1,\y+2-\t*0.3827-1+\t*0.9239) -- (\x+2-\t*0.9239 ,\y+ 2-\t*0.3827) -- (\x+\t*0.9239,\y+2-\t*0.3827);

 \draw (\x+2-\t*0.3827,\y+\t*0.9239) -- (\x+2-\t*0.3827-1+\t*0.9239,\y+1) -- ( \x+2-\t*0.3827, \y+2-\t*0.9239 ) -- (\x+2-\t*0.3827, \y+\t*0.9239);

}}

\end{tikzpicture}

\end{tabular}
\end{center}

The part of the picture  inside the square $[0,2]^2$ is a fundamental domain for the $\Lambda$ action, so we arrive at the following picture of $\widetilde{X}$:
\begin{equation}\label{eq:Xtilde-example}
\begin{tikzpicture}[baseline=(current bounding box.center),scale=1.5]
\def\t{.2}
\foreach \x in {0}{
\foreach \y in {0}{
\draw (\x+\t*0.9239,\y+\t*0.3827) -- (\x+1,\y+\t*0.3827+1-\t*0.9239) -- (\x+2-\t*0.9239 , \y+\t*0.3827) -- (\x+\t*0.9239,\y+\t*0.3827);

 \draw (\x+\t*0.3827,\y+\t*0.9239) -- (\x+\t*0.3827+1-\t*0.9239,\y+1) -- ( \x+\t*0.3827, \y+2-\t*0.9239 ) -- (\x+\t*0.3827, \y+\t*0.9239);

 \draw (\x+\t*0.9239,\y+2-\t*0.3827) -- (\x+1,\y+2-\t*0.3827-1+\t*0.9239) -- (\x+2-\t*0.9239 ,\y+ 2-\t*0.3827) -- (\x+\t*0.9239,\y+2-\t*0.3827);

 \draw (\x+2-\t*0.3827,\y+\t*0.9239) -- (\x+2-\t*0.3827-1+\t*0.9239,\y+1) -- ( \x+2-\t*0.3827, \y+2-\t*0.9239 ) -- (\x+2-\t*0.3827, \y+\t*0.9239);

}}

\end{tikzpicture}
\end{equation}
Note that none of the edges in this picture are glued together; $\widetilde{X}$ is the disjoint union of four $2$-simplices.

In the action of $W=W_M$ on $\widetilde{X}$, the element $s$ reflects across the diagonal axis from bottom-left to top-right in \eqref{eq:Xtilde-example}, while $t$ reflects across the horizontal axis through the centre of the picture. There is a natural $W$-invariant CW-structure:
\begin{equation}\label{eq:Xtilde-example2}
\begin{tikzpicture}[baseline=(current bounding box.center),scale=1.5]
\def\t{.3}

\draw (\t*0.9239,\t*0.3827) -- (1,\t*0.3827+1-\t*0.9239) -- (2-\t*0.9239 , \t*0.3827) -- (\t*0.9239,\t*0.3827);

 \draw (\t*0.3827,\t*0.9239) -- (\t*0.3827+1-\t*0.9239,1) -- ( \t*0.3827, 2-\t*0.9239 ) -- (\t*0.3827, \t*0.9239);

 \draw (\t*0.9239,2-\t*0.3827) -- (1,2-\t*0.3827-1+\t*0.9239) -- (2-\t*0.9239 , 2-\t*0.3827) -- (\t*0.9239,2-\t*0.3827);

 \draw (2-\t*0.3827,\t*0.9239) -- (2-\t*0.3827-1+\t*0.9239,1) -- ( 2-\t*0.3827, 2-\t*0.9239 ) -- (2-\t*0.3827, \t*0.9239);

\node[draw,circle,fill,inner sep=1.5pt] at (\t*0.9239,\t*0.3827) {};

\node[draw,circle,fill,inner sep=1.5pt] at (1,\t*0.3827+1-\t*0.9239) {};

\node[draw,circle,fill,inner sep=1.5pt] at (2-\t*0.9239 , \t*0.3827) {};

\node[draw,circle,fill,inner sep=1.5pt] at (1 , \t*0.3827) {};

\draw (1,\t*0.3827+1-\t*0.9239) -- (1 , \t*0.3827);

\node[draw,circle,fill,inner sep=1.5pt] at (\t*0.3827,\t*0.9239) {};

\node[draw,circle,fill,inner sep=1.5pt] at (\t*0.3827+1-\t*0.9239,1) {};

\node[draw,circle,fill,inner sep=1.5pt] at ( \t*0.3827, 2-\t*0.9239 ) {};

\node[draw,circle,fill,inner sep=1.5pt] at (\t*0.3827,1) {};

\draw (\t*0.3827+1-\t*0.9239,1) -- (\t*0.3827,1);

\node[draw,circle,fill,inner sep=1.5pt] at (\t*0.9239,2-\t*0.3827) {};

\node[draw,circle,fill,inner sep=1.5pt] at (1,2-\t*0.3827-1+\t*0.9239) {};

\node[draw,circle,fill,inner sep=1.5pt] at (2-\t*0.9239 , 2-\t*0.3827) {};

\node[draw,circle,fill,inner sep=1.5pt] at (1 , 2-\t*0.3827) {};

\draw (1,2-\t*0.3827-1+\t*0.9239) -- (1 , 2-\t*0.3827);

\node[draw,circle,fill,inner sep=1.5pt] at (2-\t*0.3827,\t*0.9239) {};

\node[draw,circle,fill,inner sep=1.5pt] at (2-\t*0.3827-1+\t*0.9239,1) {};

\node[draw,circle,fill,inner sep=1.5pt] at ( 2-\t*0.3827, 2-\t*0.9239 ) {};

\node[draw,circle,fill,inner sep=1.5pt] at (2-\t*0.3827,1) {};

\draw (2-\t*0.3827-1+\t*0.9239,1) -- (2-\t*0.3827,1);

\end{tikzpicture}
\end{equation}

The top half of the right-hand triangle in  \eqref{eq:Xtilde-example2} is a fundamental domain for the action of $W$. The stabilisers of the vertices and edges in this triangle are indicated in the picture below:

\begin{center}
\begin{tikzpicture}
\node[draw,fill,circle,inner sep=2pt] at (0,0) {};
\node[draw,fill,circle,inner sep=2pt] at (2,2) {};
\node[draw,fill,circle,inner sep=2pt] at (2,0) {};
\draw (0,0) -- (2,2) -- (2,0) -- (0,0);
\node at (-.75,0) {$\{1,t\}$};
\node at (2.6,2) {$\{1\}$};
\node at (2.75,0) {$\{1,t\}$};
\node at (2.6,1) {$\{1\}$};
\node at (1,-.3) {$\{1,t\}$};
\node at (.7,1.3) {$\{1\}$};
\end{tikzpicture}
\end{center}
(The stabiliser of the interior $2$-cell is $\{1\}$.)
Decomposing $\Rep\{1,t\} = \Z(\triv) \oplus \Z(\triv-\sign)$ as explained in Section \ref{sec:1d} gives an identification between the complex computing the Bredon cohomology of $\widetilde{X}$  and the complex of ordinary cellular cochains for the following space:
\begin{center}
\begin{tikzpicture}
\node[draw,fill,circle,inner sep=2pt] at (0,0) {};
\node[draw,fill,circle,inner sep=2pt] at (2,2) {};
\node[draw,fill,circle,inner sep=2pt] at (2,0) {};
\draw (0,0) -- (2,2) -- (2,0) -- (0,0);

\node[draw,fill,circle,inner sep=2pt] at (0,-.3) {};
\node[draw,fill,circle,inner sep=2pt] at (2,-.3) {};
\draw (0,-.3) -- (2,-.3);

\end{tikzpicture}
\end{center}
We conclude that $H^0_W(\widetilde{X},\mathcal{W})\cong \Z^2$, while $H^1_W(\widetilde{X},\mathcal{W}) = H^2_W(\widetilde{X},\mathcal{W})=0$.

\paragraph{Case 2: $\sigma_1=\sigma_2$ and $\order(\sigma_i)=2$.}

We have $W=W_M$. Consulting \eqref{eq:r-table} shows that $(\germ{a}_M^*)'$ is the union of the $(2\Z)^2$-translates of the root hyperplanes $y=x$ and $y=-x$ in the space $\germ{a}_M^*\cong \R^2$. Removing these lines from $\germ{a}_M^*$ leaves a disjoint union of open squares, and $\widetilde{\germ{a}_M^*}$ is the disjoint union of the closures of these squares:

\begin{center}
\begin{tabular}{ccc}
$\germ{a}_M^*\setminus (\germ{a}_M^*)'$ &  & $\widetilde{\germ{a}_M^*}$ \\ \\
\begin{tikzpicture}[dotted,thick]

\draw (-2,-2) -- (2,2);
\draw (-2,0) -- (0,2);
\draw (0,-2) -- (2,0);

\draw (-2,2) -- (2,-2);
\draw (-2,0) -- (0,-2);
\draw (0,2) -- (2,0);
\end{tikzpicture}
&
\qquad\qquad
&
\begin{tikzpicture}
\def\t{.15}
\draw[fill=gray] (0+\t,0) -- (1,1-\t) -- (2-\t,0) -- (1,-1+\t) -- (0+\t,0);
\draw (0,\t) -- (1-\t,1) -- (0,2-\t) -- (-1+\t,1) -- (0,\t);
\draw (-\t,0) -- (-1,1-\t) -- (-2+\t,0) -- (-1,-1+\t) -- (-\t,0);
\draw[fill=gray] (0,-\t) -- (1-\t,-1) -- (0,\t-2) -- (-1+\t,-1) -- (0, -\t);
\draw (2,\t) -- (1+\t,1) -- (2,2-\t);
\draw (\t,2) -- (1,1+\t) -- (2-\t,2);
\draw (-\t,2) -- (-1,1+\t) -- (\t-2,2);
\draw (-2,\t) -- (-1-\t,1) -- (-2,2-\t);
\draw (2,-\t) -- (1+\t,-1) -- (2,-2+\t);
\draw (\t,-2) -- (1,-1-\t) -- (2-\t,-2);
\draw (-\t,-2) -- (-1,-1-\t) -- (\t-2,-2);
\draw (-2,-\t) -- (-1-\t,-1) -- (-2,-2+\t);

\end{tikzpicture}

\end{tabular}
\end{center}

The two shaded squares in the picture of $\widetilde{\germ{a}_M^*}$ constitute a fundamental domain for the $\Lambda$-action, so we can identify $\widetilde{X}$ with those two squares. 

In the action of $W=\{\pm 1, \pm s, \pm t, \pm st\}$ on $\widetilde{X}$, the elements $\pm s$ and $\pm st$ exchange the two squares, while the group $\{\pm 1,\pm t\}$ acts on each square: $-1$ by a $180^\circ$ rotation, $t$ by reflecting across the horizontal axis, and $-t$ by reflecting across the vertical axis. The upper-right-hand quadrant of one of the squares is a fundamental domain for the $W$-action. This fundamental domain is a $2$-simplex, whose vertex- and edge-stabiliser groups are indicated below:

\begin{center}
\begin{tikzpicture}
\node[draw,fill,circle,inner sep=2pt] at (0,0) {};
\node[draw,fill,circle,inner sep=2pt] at (0,2) {};
\node[draw,fill,circle,inner sep=2pt] at (2,0) {};
\draw (0,0) -- (0,2) -- (2,0) -- (0,0);
\node at (-1,0) {$\{\pm 1,\pm t\}$};
\node at (-.85,2) {$\{1,-t\}$};
\node at (2.75,0) {$\{1,t\}$};
\node at (-0.85,1) {$\{1,-t\}$};
\node at (1,-.3) {$\{1,t\}$};
\node at (1.3,1.3) {$\{1\}$};
\end{tikzpicture}
\end{center}
(The stabiliser of the interior $2$-cell is $\{1\}$.)

We write $\widehat{\{\pm 1,\pm t\}}= \{\triv,\sign,\rho_+,\rho_-\}$, where $\triv$ is the trivial character, $\sign(\pm t)=-1$, $\rho_+(\pm t)=\pm 1$, and $\rho_-(\pm t)=\mp 1$. We  decompose $\Rep\{\pm1,\pm t\}$ as follows:
\[
\Rep\{\pm 1,\pm t\}=\Z(\triv) \oplus \Z(\triv-\rho_+)\oplus \Z(\triv-\rho_-) \oplus \Z(\triv+\sign-\rho_+ - \rho_-).
\]
This decomposition, together with the decomposition of $\Rep\{1,t\}$ and $\Rep\{1,-t\}$ used in Case 1, gives an identification between the Bredon cohomology complex of $\widetilde{X}$ and the simplicial cohomology complex of the following space:

\begin{center}
\begin{tikzpicture}
\node[draw,fill,circle,inner sep=2pt] at (0,0) {};
\node[draw,fill,circle,inner sep=2pt] at (0,2) {};
\node[draw,fill,circle,inner sep=2pt] at (2,0) {};
\draw (0,0) -- (0,2) -- (2,0) -- (0,0);

\node[draw,fill,circle,inner sep=2pt] at (0,-.3) {};
\node[draw,fill,circle,inner sep=2pt] at (2,-.3) {};
\draw (0,-.3) -- (2,-.3);

\node[draw,fill,circle,inner sep=2pt] at (-.3,0) {};
\node[draw,fill,circle,inner sep=2pt] at (-.3,2) {};
\draw (-.3,0) -- (-.3,2);
\node[draw,fill,circle,inner sep=2pt] at (-.3,-.3) {};
\end{tikzpicture}
\end{center}
We thus conclude that the cohomology is isomorphic to $\Z^4$ in degree zero, and $0$ in degree $\geq 1$.

\paragraph{Case 3: $\sigma_1=\sigma_2$ and $\order(\sigma_i)>2$.}

We have $W=\{1,s\}$. Consulting \eqref{eq:r-table} shows that $(\germ{a}_M^*)'$ is the union of the $(2\Z)^2$-translates of the root hyperplane $y=x$ in the space $\germ{a}_M^*\cong \R^2$. Removing these lines from $\germ{a}_M^*$ leaves a disjoint union of open diagonal strips, and $\widetilde{\germ{a}_M^*}$ is the disjoint union of the closures of these strips.

The lattice $\Lambda=(2\Z)^2$ can be decomposed as $\Z(2,0)\oplus \Z(2,2)$, where the generator $(2,0)$ acts on $\widetilde{\germ{a}_M^*}$ by moving each diagonal strip onto a neighboring strip, while the generator $(2,2)$ acts by a translation in each strip. The quotient space $\widetilde{X}=\widetilde{\germ{a}_M^*}/\Lambda$ is thus a cylinder. The nontrivial element $s\in W$ acts  by rotating this cylinder $180^\circ$ around its axis of rotational symmetry, and then reflecting in the plane perpendicular to that axis. This operation has no fixed points, so in this case the action of $W$ on $\widetilde{X}$ is free, and the Bredon cohomology of $\widetilde{X}$ is  the ordinary cellular cohomology of the quotient $\widetilde{X}/W$. This quotient is a M\"obius strip, and so we conclude that the Bredon cohomology of $\widetilde{X}$ is isomorphic to $\Z$ in degrees $0$ and $1$, and is $0$ in higher degrees.

\paragraph{Case 4: $\order(\sigma_1)=2$, $\sigma_2=1$.}

We have $W=\{\pm 1,\pm t\}$. Consulting \eqref{eq:r-table} shows that $(\germ{a}_M^*)'$ is the union of the lines $y=2n$, for $n\in \Z$, in the space $\germ{a}_M^*\cong \R^2$. Removing these lines from $\germ{a}_M^*$ leaves a disjoint union of open horizontal strips, and $\widetilde{\germ{a}_M^*}$ is the disjoint union of the closures of these strips. The quotient $\widetilde{X}=\widetilde{\germ{a}_M^*}/\Lambda$ is a closed cylinder, obtained by gluing together the two vertical sides of the square $[0,2]^2$. The elements $t$ and $-t$ in $W$ act on $\widetilde{X}$ by reflection through the lines $y=1$ and $x=1$, respectively. The (image in the cylinder of the) square $[0,1]^2$ is a fundamental domain for the this action. The stabilisers of the vertices and edges of this fundamental domain are indicated below:
\begin{center}
\begin{tikzpicture}[scale=1]
\node[draw,circle,fill,inner sep=2pt] at (0,0) {};
\node[draw,circle,fill,inner sep=2pt] at (0,2) {};
\node[draw,circle,fill,inner sep=2pt] at (2,0) {};
\node[draw,circle,fill,inner sep=2pt] at (2,2) {};

\draw (0,0) -- (0,2) -- (2,2) -- (2,0) -- (0,0);

\node at (-1,0) {$\{1,-t\}$};
\node at (-1,2) {$\{\pm 1,\pm t\}$};
\node at (-1,1) {$\{1,-t\}$};
\node at (3,0) {$\{1,-t\}$};
\node at (3,1) {$\{1,-t\}$};
\node at (3,2) {$\{\pm 1,\pm t\}$};
\node at (1,-.5) {$\{1\}$};
\node at (1,2.5) {$\{1,t\}$};

\end{tikzpicture}
\end{center}
(The stabiliser of the interior $2$-cell is $\{1\}$.)

Decomposing $\Rep\{1,t\}$, $\Rep\{1,-t\}$, and $\Rep\{\pm 1,\pm t\}$ as in Case 2 gives an identification between the Bredon cohomology complex for this square, and the ordinary cellular cohomology complex for the space

\begin{center}
\begin{tikzpicture}[scale=1]
\node[draw,circle,fill,inner sep=2pt] at (0,0) {};
\node[draw,circle,fill,inner sep=2pt] at (0,2) {};
\node[draw,circle,fill,inner sep=2pt] at (2,0) {};
\node[draw,circle,fill,inner sep=2pt] at (2,2) {};
\draw (0,0) -- (0,2) -- (2,2) -- (2,0) -- (0,0);

\node[draw,circle,fill,inner sep=2pt] at (-.3,0) {};
\node[draw,circle,fill,inner sep=2pt] at (-.3,2) {};
\draw (-.3,0) -- (-.3,2);

\node[draw,circle,fill,inner sep=2pt] at (0,2.3) {};
\node[draw,circle,fill,inner sep=2pt] at (2,2.3) {};
\draw (0,2.3) -- (2,2.3);

\node[draw,circle,fill,inner sep=2pt] at (2.3,2) {};
\node[draw,circle,fill,inner sep=2pt] at (2.3,0) {};
\draw (2.3,2) -- (2.3,0);

\node[draw,circle,fill,inner sep=2pt] at (-.3,2.3) {};

\node[draw,circle,fill,inner sep=2pt] at (2.3,2.3) {};

\end{tikzpicture} 

\end{center}

The Bredon cohomology is thus isomorphic to $\Z^6$ in degree $0$, and $0$ in higher degrees.

\paragraph{Case 5: $\sigma_1\neq \sigma_2$ and $\order(\sigma_1)=\order(\sigma_2)=2$.}

We have $W=\{\pm 1,\pm t\}$. None of the conditions on $\sigma$ in \eqref{eq:r-table} are satisfied, and so we have $(\germ{a}_M^*)'=\emptyset$, $\widetilde{\germ{a}_M^*}=\germ{a}_M^*$, and $\widetilde{X}=X$. Realising the torus $X$ as the result of identifying opposite sides in the square $[0,2]^2\subseteq \R^2\cong \germ{a}_M^*$, the symmetry $t\in W$ acts by reflection across the line $y=1$, while the symmetry $-t\in W$ acts by reflection across the line $x=1$. The (image in the torus of the) square $[0,1]^2$ is a fundamental domain for this action. Labelling each vertex and edge of this square by its stabiliser in $W$, we have

\begin{center}
\begin{tikzpicture}[scale=1]
\node[draw,circle,fill,inner sep=2pt] at (0,0) {};
\node[draw,circle,fill,inner sep=2pt] at (0,2) {};
\node[draw,circle,fill,inner sep=2pt] at (2,0) {};
\node[draw,circle,fill,inner sep=2pt] at (2,2) {};

\draw (0,0) -- (0,2) -- (2,2) -- (2,0) -- (0,0);

\node at (-1,0) {$\{\pm1,\pm t\}$};
\node at (-1,2) {$\{\pm 1,\pm t\}$};
\node at (-1,1) {$\{1,-t\}$};
\node at (3,0) {$\{\pm 1,\pm t\}$};
\node at (3,1) {$\{1,-t\}$};
\node at (3,2) {$\{\pm 1,\pm t\}$};
\node at (1,-.5) {$\{1,t\}$};
\node at (1,2.5) {$\{1,t\}$};

\end{tikzpicture}
\end{center}
(The stabiliser of the interior $2$-cell is $\{1\}$.) Decomposing $\Rep\{1,t\}$, $\Rep\{1,-t\}$, and $\Rep\{\pm 1,\pm t\}$ as in Case 4 identifies the Bredon cohomology complex with the ordinary cellular cohomology complex of the space

\begin{center}
\begin{tikzpicture}[scale=1]
\node[draw,circle,fill,inner sep=2pt] at (0,0) {};
\node[draw,circle,fill,inner sep=2pt] at (0,2) {};
\node[draw,circle,fill,inner sep=2pt] at (2,0) {};
\node[draw,circle,fill,inner sep=2pt] at (2,2) {};
\draw (0,0) -- (0,2) -- (2,2) -- (2,0) -- (0,0);

\node[draw,circle,fill,inner sep=2pt] at (-.3,0) {};
\node[draw,circle,fill,inner sep=2pt] at (-.3,2) {};
\draw (-.3,0) -- (-.3,2);

\node[draw,circle,fill,inner sep=2pt] at (0,2.3) {};
\node[draw,circle,fill,inner sep=2pt] at (2,2.3) {};
\draw (0,2.3) -- (2,2.3);

\node[draw,circle,fill,inner sep=2pt] at (2.3,2) {};
\node[draw,circle,fill,inner sep=2pt] at (2.3,0) {};
\draw (2.3,2) -- (2.3,0);

\node[draw,circle,fill,inner sep=2pt] at (0,-.3) {};
\node[draw,circle,fill,inner sep=2pt] at (2,-.3) {};
\draw (0,-.3) -- (2,-.3);

\node[draw,circle,fill,inner sep=2pt] at (-.3,2.3) {};
\node[draw,circle,fill,inner sep=2pt] at (2.3,2.3) {};
\node[draw,circle,fill,inner sep=2pt] at (-.3,-.3) {};
\node[draw,circle,fill,inner sep=2pt] at (2.3,-.3) {};

\end{tikzpicture} 
\end{center}

The Bredon cohomology is thus $\Z^9$ in degree $0$, and $0$ in higher degrees.

\paragraph{Case 6: $\order(\sigma_1)>2$, $\sigma_2=1$.}

We have $W=\{1,t\}$. The space $\widetilde{\germ{a}_M^*}$ is the same as in Case 4: a union of horizontal strips. The quotient $\widetilde{X}=\widetilde{\germ{a}_M^*}/\Lambda$ is a cylinder, obtained by gluing together the two vertical sides of the square $[0,2]^2$. 

The nontrivial element $t\in W$ acts on the cylinder by reflection through a plane orthogonal to the axis of rotational symmetry. The top half of the cylinder constitutes a fundamental domain for this action; the stabliser of the middle circle is $\{1,t\}$, while the other points in the cylinder have trivial stabiliser. Using the same decomposition of $\Rep\{1,t\}$ as we did in Case 1, we get an identification between the Bredon cohomology of $\widetilde{X}$  and the ordinary cellular cohomology of the disjoint union of a cylinder and a circle. Thus the Bredon cohomology is $\Z^2$ in degrees $0$ and $1$, and $0$ in higher degrees.

\paragraph{Case 7: $\order(\sigma_1)>2$ and $\order(\sigma_2)=2$.}

We have $W=\{1,t\}$, and since none of the conditions on $\sigma$ in \eqref{eq:r-table} are satisfied we have $(\germ{a}_M^*)'=\emptyset$ and $\widetilde{X}=X$. The symmetry $t$ acts on the torus $X$ by a reflection as in Case 6, and a fundamental domain for this action is a cylinder, in which the two boundary circles have stabiliser $\{1,t\}$ while the other points have trivial stabiliser. Decomposing $\Rep\{1,t\}$ as in Case 1 identifies the Bredon cohomology complex of $X$ with the ordinary cohomology complex of the disjoint union of a cylinder and two circles, and so we find that the Bredon cohomology is isomorphic to $\Z^3$ in degrees $0$ and $1$, and $0$ in higher degrees.

\paragraph{Case 8: $\sigma_1\neq\sigma_2$ and $\order(\sigma_i)>2$.}

Here we have $W=\{1\}$, $(\germ{a}_M^*)'=\emptyset$ and $\widetilde{X}=X$. So the Bredon cohomology of $X$ is, in this case, just the ordinary cellular cohomology of the $2$-torus $X$; thus the cohomology is isomorphic to $\Z$ in degrees $0$ and $2$, isomorphic to $\Z^2$ in degree $1$, and $0$ in higher degrees.

 \bibliographystyle{alpha}
 \bibliography{p-adic.bib}

\end{document}